\documentclass[11pt]{article}
\usepackage{amsmath}
\usepackage{amssymb}
\usepackage{amsfonts, amsthm}
\usepackage{array}
\usepackage{booktabs}
\usepackage{comment}
\usepackage{boldfonts}
\usepackage{subcaption}
\usepackage{CJKutf8}
\usepackage{bm,multirow}
\usepackage[a4paper]{geometry}
\usepackage{cancel}
\usepackage{graphicx} 
\usepackage{import}
\usepackage{pdfpages}
\usepackage{transparent}
\usepackage{xcolor,makecell}
\usepackage[show]{ed}
\numberwithin{equation}{section}

\makeatletter
\def\oversortoftilde#1{\mathop{\vbox{\m@th\ialign{##\crcr\noalign{\kern3\p@}%
      \sortoftildefill\crcr\noalign{\kern3\p@\nointerlineskip}%
       \hfil\displaystyle{#1}\hfil \crcr}}}\limits}
\def\sortoftildefill{ \m@th \setbox\z@\hbox{$\braceld$}%
  \braceld\leaders\vrule \@height\ht\z@ \@depth\z@ \braceru }
\makeatother

\def\bH{\mathbf{H}}
\def\bv{\mathbf{v}}
\def\bu{\mathbf{u}}
\def\bw{\mathbf{w}}
\def\bE{\mathbf{E}}
\def\bn{\mathbf{n}}
\def\bx{\mathbf{x}}

\def\be{\mathbf{e}}
\def\cT{\mathcal{T}}
\def\cI{\mathcal{I}}

\DeclareMathOperator{\difdiv}{div}

\DeclareMathOperator{\curl}{curl}
\DeclareMathOperator{\iimp}{imp}
\DeclareMathOperator{\rot}{rot}

\newcommand{\imp}[2]{{\rm imp}_{#1}(#2)}
\newcommand{\imph}[2]{{\rm imp}^h_{#1}(#2)}

\newcommand{\ignore}[1]{}

\newcommand{\ri}{{\mathrm{i}}}

\newtheorem{definition}{Definition}[section]
\newtheorem{theorem}[definition]{Theorem}
\newtheorem{corollary}[definition]{Corollary}
\newtheorem{proposition}[definition]{Proposition}
\newtheorem{assumption}[definition]{Assumption}
\newtheorem{lemma}[definition]{Lemma}
\newtheorem{remark}[definition]{Remark}
\newtheorem{conjecture}[definition]{Conjecture}
\newtheorem{cor}[definition]{Corollary}

\usepackage{hyperref}
\definecolor{myblue}{rgb}{0,0,0.6}
\definecolor{lightGray}{RGB}{198,198,198}
\hypersetup{colorlinks=true, linkcolor=myblue,citecolor=myblue,filecolor=myblue,urlcolor=myblue}

\usepackage{soul}

\definecolor{escol}{rgb}{0,0,0.8}
\definecolor{estcol}{rgb}{0,0.5,0}
\definecolor{esnewcol}{rgb}{0,0.5,0}

\newcommand{\mythmname}[1]{\emph{\textbf{(#1)}}}

\title{Convergence of parallel overlapping  
domain decomposition methods with impedance boundary conditions
for time-harmonic Maxwell equations in heterogeneous media

}

\author{%
  Luyu Cen$^{1,2}$,
  Shihua Gong$^{1,2,4}$\thanks{Corresponding author. Email: \texttt{gongshihua@cuhk.edu.cn}},
  Euan A.~Spence$^{3}$,
  Yue Yu$^{1,2}$
}
\newcommand{\authoraffiliations}{%
  \footnotesize
  $^{1}$School of Science and Engineering, The Chinese University of Hong Kong, Shenzhen, Guangdong 518172, China\\
  $^{2}$Shenzhen International Center for Industrial and Applied Mathematics, Shenzhen Research Institute of Big Data, Shenzhen, Guangdong 518172, China\\
  $^{3}$Department of Mathematical Sciences, University of Bath, Bath BA2 7AY, United Kingdom\\
  $^{4}$Shenzhen Loop Area Institute, Shenzhen, Guangdong 518038, China
}
\date{\today}
\begin{document}
\maketitle
\vspace{-1.5em}
\begin{center}
\authoraffiliations
\end{center}
\vspace{0.5em}

\begin{abstract}
This paper analyzes the convergence of parallel overlapping 
domain-decomposition methods with impedance boundary conditions for the time-harmonic Maxwell equations in heterogeneous media. 
We prove that the parallel iterative method is well-posed in an appropriate function space, and characterize the error propagation operator through impedance-to-impedance maps that describe interactions between neighboring subdomains. For strip  domain decompositions, we derive explicit convergence estimates in terms of the norms of the impedance-to-impedance maps. 
At the discrete level, we develop the finite-element counterpart of these results based on N\'{e}d\'{e}lec-element discretisations.
Under the assumption that the discrete impedance-to-impedance maps approximate their continuous counterparts as the mesh is refined, we show that the discrete method inherits the convergence behavior of the continuous method.
We illustrate this theory with 
numerical experiments for strip domain decompositions, and also present numerical
experiments for checkerboard domain decompositions that go beyond our theory. 
\end{abstract}

\section{Introduction}\label{sec:intro}

\paragraph{Context: the numerical solution of the time-harmonic Maxwell equations}

The time-harmonic Maxwell equations model electromagnetic wave propagation in a wide range of applications, including radar, wireless communications, and medical imaging. At high frequencies, however, their numerical solution becomes particularly challenging. Indeed, to resolve the oscillations in the solution, one needs to use sufficiently fine meshes, leading to very large discrete systems. Furthermore, these systems are typically non-Hermitian and highly indefinite, making the design of scalable and robust iterative solvers a difficult task. The situation is even more complicated in heterogeneous media, where spatially varying material parameters further affect both the underlying propagation of the wave and its numerical approximation. 

\paragraph{Finite-element methods and the pollution effect.}
Just as for the Helmholtz equation, finite-element methods with fixed polynomial degree suffer from the \emph{pollution effect} when applied to the time-harmonic Maxwell equations; i.e., with a fixed number of points per wavelength, the accuracy decreases as the frequency increases \cite{BaSa:00}.

Sufficient conditions for accuracy on the mesh width as a function of frequency and (fixed) polynomial degree 
are then given in 
\cite{CGS1} (for the standard FEM with first- or second-family N\'ed\'elec elements) 
\cite{LuWu:25}
(for the standard FEM with second-family 
N\'ed\'elec elements)
\cite{Lu2026preasymptotic,lu2019continuous} (for continuous interior penalty methods) 
and \cite{feng2014absolutely}
for interior penalty discontinuous Galerkin methods. Furthermore \cite{nicaise2020convergence,melenk2024wavenumber, MeWo:26} show that -- just as in the Helmholtz case -- increasing the polynomial degree like $\log k$, where $k$ is the frequency, removes the pollution effect.

\paragraph{Existing theory of two-level domain-decomposition methods for Maxwell.}
By the discussion in the last paragraph, the linear systems arising from the FEM grow at least like $k^d$ as $k\to\infty$. There has been much research on the efficient solution of these linear systems using domain-decomposition (DD) methods.
One class of DD methods for which there exists rigorous theory are two-level methods; i.e., subdomain solves (on overlapping subdomains) are combined with a global coarse solve. 
Convergence theory for two-level methods
with piecewise-polynomial coarse spaces 
used as a GMRES preconditioner 
was given for fixed $k$ in 
\cite{gopalakrishnan2003overlapping}, and for imaginary $k$ (i.e., for problems with absorption) in \cite{bonazzoli2019domain}, with the recent paper \cite{bonazzoli2026analysis} obtaining analogues of the latter results in the case of heterogeneity.
If the absorption is 
small enough, then the problem with absorption is a good preconditioner for the Maxwell equations with $k$ real \cite{spence2026preconditioning}, and this fact was then used in \cite{li2025hybrid} to prove results about two-level DD for Maxwell with a particular problem-adapted coarse space (i.e., where the coarse-space functions are precomputed).

\paragraph{The subject of this paper:~theory of one-level domain-decomposition methods for Maxwell.}
This paper is concerned with the theory of one-level overlapping domain-decomposition methods for the time-harmonic Maxwell equations, for which, to our knowledge, there is no existing convergence theory (at either at the continuous or the discrete level). 

It is important to emphasise that the character of one-level methods is very different from that of the two-level methods described above. Indeed, in two-level methods, the coarse space must resolve the propagative behaviour of the waves (otherwise the behaviour is worse than that of the corresponding one-level method; see, e.g., \cite[Table 9]{bootland2021comparison}) and the subdomains are taken to be small, usually of diameter $\sim k^{-1}$, and so do not see the propagation. In contrast, understanding how the subdomain problems see the propagation of the waves is crucial in understanding one-level methods (as demonstrated in the Helmholtz context in \cite{gong2022convergence, galkowski2024convergence}).
We emphasise that, when using a two-level method, the coarse problem is usually still very large, and so one often solves it 
using a one-level method (see, e.g., \cite{bootland2021comparison, bonazzoli2019domain}); in this case, therefore, understanding the behaviour of the one-level method is necessary for a full two-level analysis. 

To further illustrate the difference between one- and two-level methods, we note that the theory of two-level methods described above is all based on showing that the field of values of the preconditioned system does not contain the origin, with GMRES convergence results then following from the Elman's estimate \cite{El:82, EiElSc:83} and its variants \cite{BeGoTy:06}.
Figure \ref{fig:FOV} shows that the field of values of the one-level method 
considered in the present paper
(defined precisely in Definition \ref{def:discrete} below)
 contains the origin, and thus the standard field-of-values argument cannot be used. 
 
\begin{figure}
    \centering
\includegraphics[width=0.65\textwidth]{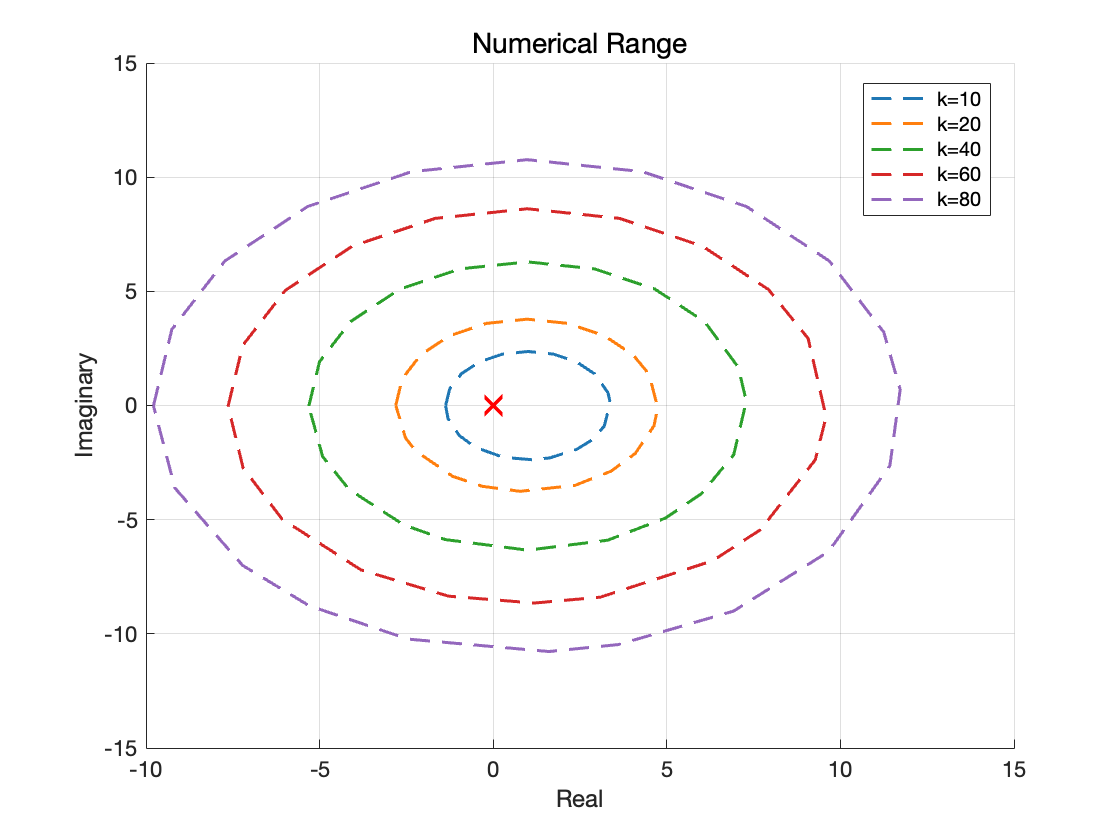}
    \caption{The field of values (FOV) of RAS-imp (with 2 subdomains) for different values of $k$ -- all contain the origin (marked with a red cross).}\label{fig:FOV}
\end{figure}

\paragraph{The contribution of this paper.}

This paper studies a parallel overlapping Schwarz method with impedance boundary conditions on subdomain interfaces (see Definition \ref{def:Schwarz} below). 
At the discrete level, this method can be viewed as a particular type of \emph{restricted additive Schwarz (RAS)} method (see \S\ref{sec:equivalence} below).

For the Helmholtz equation (with constant coefficients), a convergence theory for the parallel Schwarz method with impedance boundary conditions and its discrete analogue was given in \cite{gong2022convergence, gong2023convergence, lafontaine2023sharp}. In more detail:~\cite{gong2022convergence} considered the Schwarz iteration as a fixed-point iteration at the PDE level, 
and formulated the error propagation operator in terms of impedance-to-impedance maps between the subdomain boundaries.
The paper 
\cite{lafontaine2023sharp} studied the norms of the impedance-to-impedance maps at high frequency at the PDE level for strip-wise decompositions, and 
\cite{gong2023convergence} showed that the corresponding finite element discretization inherits the convergence behavior of the continuous problem.

This paper studies of one-level Maxwell DD methods with impedance transmission conditions from the perspective introduced (in the Helmholtz context) in \cite{gong2022convergence, gong2023convergence, lafontaine2023sharp}. One important difference is that, while \cite{gong2022convergence, gong2023convergence, lafontaine2023sharp} considered \emph{constant-coefficient} Helmholtz problems, we work from the start with \emph{variable-coefficient} Maxwell problems.
Remark \ref{rem:compare} at the end of the paper summarises to what extent the results of this paper obtain the Maxwell analogues of the Helmholtz results in 
\cite{gong2022convergence, gong2023convergence, lafontaine2023sharp}.

The main contributions of this paper are as follows.
\begin{itemize}
\item
We develop a PDE-level framework for parallel overlapping Schwarz methods with impedance transmission conditions for time-harmonic Maxwell equations in \emph{variable-coefficient} media. We prove well-posedness of the iteration for general Lipschitz subdomains (using the harmonic-analysis results of \cite{Ne:67,JeKe:95,costabel1990remark}), derive the subdomain error equations, and show that the resulting error propagation operator can be expressed in terms of impedance-to-impedance maps between artificial interfaces.
\item For strip decompositions, we give sufficient conditions on the operator norms of these impedance-to-impedance maps for the parallel Schwarz method to converge, thereby reducing the analysis of the convergence to bounds on interface maps rather than on the full error propagation operator directly. Analogous formulations are available for more general decompositions, at the cost of more complicated interface geometries (see \cite[\S3.2 and \S6.2]{gong2022convergence}).
\item At the discrete level, we show that the N\'ed\'elec finite element discretization of the parallel Schwarz method can be interpreted as a restricted additive Schwarz (RAS-imp) method, and we formulate the discrete error propagation operator in terms of discrete impedance-to-impedance maps. Under the assumption that the discrete maps approximate their continuous counterparts as $h\to 0$, we prove that the discrete method inherits the convergence behavior of the continuous one.
\item We provide numerical experiments for two-dimensional Maxwell problems that validate the power contractivity of the discrete error propagation operator, illustrate the boundedness of the impedance-to-impedance maps, and demonstrate the effectiveness of RAS-imp as both a stationary iteration and a Krylov preconditioner for strip and checkerboard decompositions in homogeneous and heterogeneous media.
\end{itemize}

The remainder of the paper is organized as follows.
Section~\ref{sec:wp} introduces the Maxwell interior impedance problem and establishes well-posedness of the continuous parallel Schwarz method.
Section~\ref{sec:errorprop} derives the error equations on each subdomain and rewrites the error propagation operator in terms of impedance-to-impedance maps.
Section~\ref{sec:4} gives sufficient conditions on the norms of these maps for convergence on strip decompositions.
Sections~\ref{sec:5} and~\ref{sec:6} develop the N\'ed\'elec discretization, identify the discrete method with a RAS-imp preconditioned Richardson iteration, and show that the discrete scheme inherits the convergence behavior of the continuous one as $h\to 0$.
Section~\ref{sec:numerical} reports numerical experiments; in particular, \S\ref{sec:num4} computes the norms of the impedance-to-impedance maps and verifies the assumptions used in Section~\ref{sec:4}.

\section{Well-posedness of the domain-decomposition method at the continuous level}\label{sec:wp}

\subsection{The Maxwell interior impedance problem and associated notation}\label{sec:stab}

For $D\subset\mathbb{R}^3$ a bounded Lipschitz open set with boundary $\partial D${\color{blue},}
let $(\cdot, \cdot)_{D}$ and $\langle \cdot, \cdot \rangle_{\Gamma}$ denote the complex $L^2(D)$ inner product and the complex $L^2(\partial D)$ inner product defined, respectively, as
\[
(\bv, \bw)_{D} = \int_{D} \bv\cdot \overline{\bw} \, d\bx \quad \text{and} \quad \langle \bv, \bw \rangle_{\partial D} = \int_{\partial D} \bv\cdot \overline{\bw} \, ds.
\]
Below, we abbreviate $\|\cdot\|_{L^2(D)}$ and $\|\cdot\|_{L^2(\partial D)}$ by $\|\cdot\|_{0,D}$ and $\|\cdot\|_{0,\partial D}$, respectively.
For $s\geq 0$, let $\bH^s(D):=(H^s(D))^3$, and let $\|\cdot\|_{s,D}$ denote the corresponding norm. 

We use the following vector-valued function spaces:
\begin{align*}
	&\mathbf{L}^2(D):=(L^2(D))^3,\\
	&\bH(\text{div}, D):=\{\bv\in \mathbf{L}^2(D)\mid \difdiv\bv \in L^2(D)\},\\
	&\bH(\text{curl}, D) :=\{\bv\in \mathbf{L}^2(D)\mid \curl\bv\in\mathbf{L}^2(D)\},\\ 
    &\bH^{\frac12}(\curl,D) = \{\bv\in \bH^{s}(D)\mid \curl \bv \in \bH^{\frac12}(D)\}.
\end{align*}
Let $\bn$ be the outward-pointing unit normal vector to $\partial D$ and 
\begin{align*}
	&\mathbf{L}_T^2(\partial D) :=\{\bv\in \mathbf{L}^2(\partial D)\mid \bv\cdot\bn=0\text{ a.e. on }\partial D\}.
\end{align*}
For $\bv \in C^\infty(\overline{D})^3$, we define the standard trace operators 
\begin{equation*}
    \gamma_n(\bv) := \bv|_{\partial D} \cdot \bn, 
    \quad
    \gamma_t(\bv):= \bn \times \bv|_{\partial D}, \quad\text{ and } \quad \gamma_T(\bv):= (\bn \times \bv|_{\partial D})\times \bn.
\end{equation*}
Recall that $\gamma_n$, $\gamma_t$, and $\gamma_T$ can be extended by continuity to act on, respectively, $\bH(\difdiv, D)$, $\bH(\curl,D)$, $\bH(\curl, D)$ (see, e.g., \cite[\S3.5.2-3.5.3]{monk2003finite}).
In what follows, we abbreviate 
$\gamma_t(\bu)$ by $\bn\times \bu$
and 
$\gamma_T(\bu)$ by $\bu_T$.
Let 
  \begin{equation*}
	  \mathbf{X}_\text{imp}(D) := \{\bu\in \bH(\text{curl},D)\mid \bu_T\in \mathbf{L}^2_T(\partial D)\}
  \end{equation*}
  with associated norm
  \begin{equation}
      \label{eq:Ximp}
  \|\bu\|_{\iimp,D}^2 :=\|\curl\bu\|_{0,D}^2 + \|\bu\|_{0,D}^2 + \|\bu_T\|_{0,\partial D}^2.
  \end{equation}

\begin{definition}[Maxwell interior impedance problem]
\label{def:IIP}
Given a bounded, simply connected, Lipschitz open set $\Omega\subset\mathbb{R}^3$, 
$L^\infty$ scalar-valued coefficients $\mu,\epsilon$, and $\lambda$ satisfying
\begin{equation}\label{eq:coe_bound}
        0<\mu_{min}\leq \mu(\bx)\leq \mu_{max}<\infty,\quad 0<\epsilon_{min}\leq \epsilon(\bx)\leq \epsilon_{max}<\infty\quad \text{for a.e. }\bx\in \Omega
\end{equation}
\begin{equation}\label{eq:coe_bound2}
    0<\lambda_{\min}\leq \lambda(\bx) \leq \lambda_{\max}<\infty\quad \text{ for a.e. }\bx\in \partial \Omega,
\end{equation}
and data $\mathbf{f} \in \bL^2(\Omega)$ and $\mathbf{g}_T\in \mathbf{L}^2_T(\partial \Omega)$, find $\bu\in \mathbf{X}_{\rm imp}(\Omega)$ such that 
\begin{equation}\label{eq:iip}
\left\{
\begin{aligned}
        \curl(\mu^{-1}\curl\bu) - k^2 \epsilon\bu = \mathbf{f} ~~&\text{ in } \Omega,\\
        \mu^{-1}\curl \bu\times\bn - i k\lambda \bu_T = \mathbf{g}_T~~ &\text{ on } \partial \Omega.
\end{aligned}
\right.
\end{equation}
\end{definition}

The variational form of problem \eqref{eq:iip} is:~find $\bu\in \bX_{\iimp}(\Omega)$ such that
\begin{equation}\label{eqn:var_global}
	a(\bu,\bv)=F(\bv)\quad \text{ for all }\quad\bv\in \mathbf{X}_{\text{imp}}(\Omega),
\end{equation}
where the sesquilinear form $a: \mathbf{X}_\text{imp}(\Omega)\times\mathbf{X}_\text{imp}(\Omega)\rightarrow \mathbb{C}$ is defined by 
  \begin{equation}
	  a(\bu,\bv) := (\mu^{-1}\curl\bu,\curl\bv)_\Omega - k^2(\epsilon\bu,\bv)_\Omega - \ri k \langle\lambda \bu_T ,\bv_T\rangle_{\partial \Omega}
      \label{eq:bilinear_form}
  \end{equation}
and the anti-linear functional $F$ is defined by
  \begin{equation}\label{e:F}
	F(\bv) :=( \mathbf{f},\bv)_\Omega + \langle \mathbf{g}_T,\bv_T\rangle_{\partial\Omega}.
\end{equation}

For $D$ a bounded Lipschitz open set, let 
\begin{equation*}
        \bm{U}(D):=\Big\{\bu\in \bm{X}_{\text{imp}}(D)\mid \difdiv\bu\in L^2(D),~\curl(\mu^{-1}\curl\bu)\in \bL^2(D),~\curl\bu\times\bn\in \bL^2_T(\partial D)\Big\}
\end{equation*}
\noindent with the norm 
    \begin{align}\nonumber
\|\bu\|_{\bm{U}(D)}^2=\|\bu\|_{0,D}^2 + \|\curl\bu\|_{0,D}^2+
\|\difdiv\bu\|_{0,D}^2+\|\curl(\mu^{-1}\curl\bu)\|_{0, D}^2\\
+\|\curl\bu\times\bn\|_{0,\partial D}^2 + \|\bu_T\|^2_{0,\partial D}.\label{eq:normU}
    \end{align}
We call
\begin{equation}\label{eq:U0}
        \bm{U}_0(D):=\big\{u\in \bm{U}(D)\mid \curl(\mu^{-1}\curl\bu)-k^2\epsilon\bu=\bm{0}
    \big\}
\end{equation}
the space of ``Maxwell-harmonic functions on $D$".

\begin{lemma}[Well-posedness of the interior impedance problem in $\mathbf{X}_{\rm imp}(\Omega)$]
\label{lem:stability}
Suppose that $\Omega\subset\mathbb{R}^3$ is a bounded, simply connected, Lipschitz open set. Suppose that $\epsilon$ and $\mu$ 
satisfy \eqref{eq:coe_bound} and, additionally, 
are piecewise $W^{1,\infty}$ with respect to a partition of $\Omega$ satisfying \cite[ Assumption 1]{BCT12}.
Let $\lambda\in L^\infty(\partial D)$ satisfy \eqref{eq:coe_bound2}.
Given $k>0$, there exists $C>0$ such that given 
$\mathbf{f} \in \bL^2(\Omega)$ and $\mathbf{g}_T\in \mathbf{L}^2_T(\partial \Omega)$ there exists $u\in \mathbf{X}_{\rm imp}(\Omega)$ satisfying the Maxwell interior impedance problem of Definition \ref{def:IIP} and the bound
\begin{equation}\label{eq:stability_imp}
\|\bu\|_{\mathbf{X}_{\rm imp}(\Omega)}\leq C\big(\|{\bf f}\|_{0,\Omega}+\|\bg_T\|_{0,\partial \Omega}\big).
    \end{equation}
\end{lemma}

We make the following two immediate remarks:
\begin{itemize}
    \item The order of quantifiers in Lemma \ref{lem:stability} implies that $C$ depends on $\Omega,\epsilon, \mu$, and $k$. For some results giving $C$ explicitly in terms of $\epsilon,\mu$, and $k$, see \cite[\S4]{chaumont-frelet2023explicit}.
    \item Partitions of $\Omega$ satisfying \cite[Assumption 1]{BCT12} (see also \cite[(i)–(iii) in statement of Proposition 2.11]{LRX16}) are very general, and can include, e.g., certain collections of infinitely-many subdomains concentrating at a point/on a boundary; see the discussion around \cite[Figure 1]{BCT12}.
    \end{itemize}

\begin{proof}[References for the proof of Lemma \ref{lem:stability}]
For $\Omega,\epsilon,$ and $\mu$ as in the statement of the theorem, the time-harmonic Maxwell equations satisfy a unique continuation principle by \cite[Theorem 2.1]{BCT12}/\cite[Proposition 2.11]{LRX16}. The standard argument  using the Fredholm alternative after removing the nullspace of the curl operator
then shows that the interior impedance problem is wellposed
(see, e.g., \cite[Section 4]{monk2003finite}).
\end{proof}

\begin{lemma}
    [Well-posedness of the interior impedance problem in $\bm{U}(\Omega)$]\label{lem:well pose}
    Under the assumptions of Lemma \ref{lem:stability} and the additional assumption that $\difdiv \mathbf{f}\in L^2(\Omega)$, there exists a unique solution $\bu\in \bm{U}(\Omega)$ of the Maxwell interior impedance problem of Definition \ref{def:IIP} satisfying
\begin{equation}\label{eq:stabilityU}
        \|\bu\|_{\bm{U}(\Omega)}^2\leq C\big(\|{\bf f}\|^2_{0,\Omega} + \|\bg_T\|^2_{0,\partial \Omega} +  \|\difdiv{\bf f}\|^2_{0,\Omega}\big).
    \end{equation}
    \end{lemma}
    \begin{proof}
        By the Maxwell PDE \eqref{eq:iip} and the bound \eqref{eq:stability_imp},$\|\curl(\mu^{-1}\curl\bu)\|_{0,\Omega}\leq k^2\|\epsilon\bu\|_{0,\Omega}+\|{\bf f}\|_{0,\Omega}\leq C\|{\bf f}\|_{0,\Omega}$. Similarly, by the boundary condition in \eqref{eq:iip} and \eqref{eq:stability_imp},  $\|\mu^{-1}\curl\bu\times\bn\|_{0,\partial \Omega}\leq k\|\lambda\bu_T\|_{0,\partial \Omega}+\|\bg_T\|_{0,\partial \Omega}\leq C(\|\bg_T\|_{0,\partial \Omega}+\|{\bf f}\|_{0,\Omega})$.
By taking the divergence of the PDE in \eqref{eq:iip}, we see that
\begin{equation*}
k^2\|\difdiv(\epsilon\bu)\|_{0,\Omega}^2 = \|\difdiv {\bf f}\|_{0,\Omega}^2,
\end{equation*}
and combining these bounds with the definition of the norm $\|\cdot\|_{\bU(\Omega)}$ \eqref{eq:normU} gives the result \eqref{eq:stabilityU}.
    \end{proof}

\subsection{An alternative characterisation of the space \texorpdfstring{$\bm{U}(D)$}{U(D)}}\label{sec:W}
 
For $D$ a bounded Lipschitz open set, let
\begin{equation}\label{e:W}
\bm{W}(D):=\Big\{\bu\in \bH^{\frac12}(\curl,D)\cap \bH(\difdiv,D)\mid\curl(\mu^{-1}\curl\bu)\in \bL^2(D)\Big\}
\end{equation}
with norm 
\begin{equation*}
    \|\bu\|_{\bm{W}(D)}^2:=\|\bu\|^2_{\frac{1}{2},D}+\|\curl\bu\|^2_{\frac{1}{2},D}+\|\curl(\mu^{-1}\curl\bu)\|^2_{0, D}+\|\difdiv\bu\|^2_{0,D}.
\end{equation*}


\begin{theorem}\label{thm:equivalent space}
 Let $D$ be a bounded, simply connected, open set, with $\partial D$ Lipschitz and connected. If $\mu \in W^{1,\infty}(D)$ then $\bm{U}(D)=\bm{W}(D)$.
\end{theorem}

The key corollary of Theorem \ref{thm:equivalent space} for our purposes is the following.

\begin{cor}\label{cor:key1}
     Let $D$ be a bounded, simply connected, open set, with $\partial D$ Lipschitz and connected. 
    If  $\mu \in W^{1,\infty}(D)$, $\bu\in \bm{U}(D)$ and $D'$ is a Lipschitz subdomain of $D$, then the restriction of $\bu$ to $D'$ is in $\bm{U}(D')$.
\end{cor}

\begin{proof}
With $\bm{U}(D)$ characterised via $\bm{W}(D)$ by Theorem \ref{thm:equivalent space} -- importantly without any conditions on $\partial D$ -- the result follows immediately.
    \end{proof}

\cite[Theorem 2]{costabel1990remark} proves that $\bm{U}(D)\subset \bm{W}(D)$, but the reverse inclusion (and hence the result of Theorem \ref{thm:equivalent space}) follows by the same argument. Since this result is key to the present paper, and relies on the harmonic-analysis results of \cite{Ne:67,JeKe:95}, we recap the proof here.

\begin{theorem}\mythmname{\cite[Theorem 2]{costabel1990remark}}\label{thm:Costabel}
Let $D$ be a bounded, simply connected, open set, with $\partial D$ Lipschitz and connected.  
Suppose that $\bu\in \mathbf{L}^2(D)$, $\difdiv \bu \in L^2(D) $, and $\curl \bu \in \mathbf{L}^2(D)$.    
Then 

(a) $\bn \times \bu \in \mathbf{L}^2(\partial D)$ if and only if $\bn\cdot \bu\in \mathbf{L}^2(\partial D)$.

(b) $\bn \times \bu \in \mathbf{L}^2(\partial D)$ if and only if $\bu \in \mathbf{H}^{1/2}(D)$.
\end{theorem}  

We first show how Theorem \ref{thm:equivalent space} follows from Theorem \ref{thm:Costabel}. 

\begin{proof}[Proof of Theorem \ref{thm:equivalent space} using Theorem \ref{thm:Costabel}]
It is sufficient to prove that, for $\bu$ such that
$\bu \in \mathbf{L}^2(D)$, $\curl \bu \in \bL^2(D)$, $\difdiv\bu \in L^2(D)$, and $\curl (\mu^{-1}\curl \bu)\in \bL^2(D)$,
both
\begin{equation}
    \label{e:equiv1}
    \bu_T \in \bL^2(\partial D)  \text{ if and only if } \bu \in \bH^{1/2}(D)
\end{equation}
and
\begin{equation}
    \label{e:equiv2}
    \curl\bu\times \bn \in \bL^2(\partial D)  \text{ if and only if } \curl\bu \in \bH^{1/2}(D).
\end{equation}
The equivalence \eqref{e:equiv1} follows directly from Theorem \ref{thm:Costabel}. The equivalence \eqref{e:equiv2} follows from applying Theorem \ref{thm:Costabel} to $\bv:= \mu^{-1}\curl \bu$. Indeed, $\bv \in \bL^2(D)$ and $\curl \bv \in \bL^2(D)$ by assumption, and $\difdiv\bv \in L^2(D)$ since $\mu\in W^{1,\infty}(D)$. Theorem \ref{thm:Costabel} then implies that 
\begin{equation*}
    (\mu^{-1}\curl\bu)\times \bn \in \bL^2(\partial D)  \text{ if and only if } \mu^{-1}\curl\bu \in \bH^{1/2}(D)
\end{equation*}
and then \eqref{e:equiv2} follows since the product of an $H^{1/2}$ function and a $W^{1,\infty}$ function is $H^{1/2}$ by, e.g., \cite[Theorem 1.4.1.1]{Gr:85} (where we have used that $W^{1,\infty}(D)= C^{0,1}(\overline{D})$ by, e.g, \cite[\S4.2.3, Theorem 5]{EvGa:92}).
\end{proof}

We now prove Theorem \ref{thm:Costabel}. This proof uses the following classic result.

\begin{theorem}\label{thm:JKN}
Let $D$ be a bounded, Lipschitz open set. Suppose that $v\in H^1(D)$ and $\Delta v \in L^2(D)$. Then

(a) $v\in H^1(\partial D)$ if and only if $\partial_n v \in L^2(\partial D)$, and 

(b) $v \in H^1(\partial D)$ if and only if $v\in H^{3/2}(D).$
\end{theorem}

\begin{proof}[References for the proof of Theorem \ref{thm:JKN}]
The result (a) is due to Ne\v{c}as; see \cite[\S5.1.2 and \S5.2.1]{Ne:67} and \cite[Theorem 4.24]{mclean2000strongly}.
The result (b) follows from the results of Jerison and Kenig \cite{JeKe:95}; see, e.g., \cite[Corollary B.5]{ChSp:24}.
\end{proof}

\begin{proof}[Proof of Theorem \ref{thm:Costabel}]
By, e.g., \cite[Chapter I, Theorem 3.4]{GiRa:86} or \cite[Lemma 2.4]{Hi:02}, there exists $\bw\in \bH^1(D)$ such that $\curl \bw= \curl \bu$ and $\difdiv \bw=0$ in $D$.
Set $\bz:=\bu-\bw$ and observe that $\bz \in \bL^2(D)$, $\curl \bz=\textbf{0}$, and $\difdiv\bz \in L^2(D)$. Since $\partial D$ is simply connected, there exists $v \in H^1(D)$ such that $\nabla v= \bm{z}$ (see, e.g., \cite[Lemma 2.2]{Hi:02}, \cite[\S3.7]{monk2003finite} and \cite[Lemma 3.1]{hiptmair2007nodal}). Observe that $\Delta v = \difdiv \bz \in L^2(D)$, so that the assumptions of Theorem \ref{thm:JKN} are satisfied.
Now
\[
\bn \times \bu = \bn \times \nabla v + \bn \times \bw 
\quad\text{ and }\quad
\bn \cdot\bu = \bn \cdot \nabla v + \bn \cdot \bw.
\]
Since $\bw \in \bH^1(D)$, both $\bn \times \bw $ and $\bn \cdot \bw$ are in $\bL^2(\partial D)$ by the standard trace theorem (see, e.g., \cite[Theorem 3.37]{mclean2000strongly})
and the fact that $\bn\in \bL^\infty(\partial D)$ (by, e.g., \cite[Page 96]{mclean2000strongly}). Part (a) of Theorem \ref{thm:JKN} therefore implies Part (a) of Theorem \ref{thm:Costabel}. 
Finally $v \in H^{3/2}(D)$ if and only if $v\in H^1(D)$ with $\nabla v \in H^{1/2}(D)$ if and only if $\bu \in \bH^{1/2}(D)$, so that (since $\bu-\nabla \phi =\bw \in \bH^1(D)$) Part (b) of Theorem \ref{thm:JKN} implies Part (b) of Theorem \ref{thm:Costabel}.
\end{proof}

\subsection{Boundedness of impedance maps}

The follow result is a consequence of Lemma \ref{lem:well pose} and  Corollary \ref{cor:key1}, and is needed for the well-posedness results in the next subsection.

\begin{lemma}[Boundedness of impedance maps]\label{lem:bound imp}
    Let $D$ be a bounded, simply connected, open set, with $\partial D$ Lipschitz and connected, and let $D'$ be a Lipschitz subdomain of $D$.
Let $\mu,\epsilon\in W^{1,\infty}(D)$ satisfy \eqref{eq:coe_bound}, and let  $\lambda\in L^\infty(D)$ satisfy \eqref{eq:coe_bound2}. Given $k>0$, there exists $C, C'>0$ such that the following is true.
 
(i) For all 
    $\bu\in \bm{U}(D)$,
\begin{equation}\label{eq:bullet1}
            \|\mu^{-1}\curl \bu\times\bn - \ri k\lambda \bu_T\|_{0,\partial D'}\leq C\|\bu\|_{\bm{U}(D')}\leq C'\|\bu\|_{\bm{U}(D)}.
        \end{equation}
         
  (ii) For all $\bu \in \bm{U}_0(D)$, 
    
        \begin{equation*}
            \|\mu^{-1}\curl \bu\times\bn - \ri k\lambda \bu_T\|_{0,\partial D'}\leq   C\|\mu^{-1}\curl \bu\times\bn - \ri k\lambda \bu_T\|_{0,\partial D
            }.
        \end{equation*}
\end{lemma}
    \begin{proof}
        For the first bullet point, the first inequality follows from the definition \eqref{eq:normU} of the norm on $\bm{U}(D)$ and the second inequality follows from Corollary \ref{cor:key1}. 
        For the second bullet point, by \eqref{eq:bullet1} we only need to show that, for $\bu\in \bm{U}_0(D)$ (defined by \eqref{eq:U0}), 
        \begin{equation*}
            \|\bu\|_{\bm{U}(D)}\leq C\|\mu^{-1}\curl \bu\times\bn - \ri k\lambda  \bu_T\|_{0,\partial D},
        \end{equation*}
        which follows from \eqref{eq:stabilityU}.
    \end{proof}

\subsection{The parallel Schwarz method and its wellposedness}

We decompose $\Omega$ into $N$ overlapping open Lipschitz subdomains $\Omega_\ell$, that is $\overline{\Omega}=\cup_{\ell=1}^N\overline{\Omega_\ell}$. 
Let $\{\chi_\ell\}_{\ell=1}^N$ be a partition of unity (PoU) of $\Omega$ subordinate to the open cover $\{\Omega_\ell\}_{\ell=1}^N$; i.e., 
\begin{equation}\label{eq:chi}
   \begin{gathered}
\operatorname{supp}\chi_\ell\subset \overline{\Omega_\ell}
        ,\quad 0\leq \chi_\ell(\bx)\leq 1\mbox{ for }\bx\in\overline{\Omega_\ell},\quad \sum_\ell \chi_\ell(\bx)=1 \mbox{ for }\bx\in\overline{\Omega},
       \end{gathered}
\end{equation}
and assume, additionally, that each $\chi_\ell \in C^{0,1}(\overline{\Omega_\ell})= W^{1,\infty}(\Omega_\ell)$.

\begin{definition}[Parallel Schwarz method with impedance boundary condition]
\label{def:Schwarz}
Let  $\mu, \epsilon, \lambda, \mathbf{f},$ and $\mathbf{g}$ be as in the Maxwell interior impedance problem of Definition \ref{def:IIP}.
Given $\{\Omega_\ell\}_{\ell=1}^N$ and $\lambda \in L^\infty(\partial \Omega)$, let $\widetilde{\lambda} \in L^\infty(\partial\Omega_\ell)$ for all $\ell=1,\ldots,N$, be such that $\widetilde{\lambda}= \lambda$ on $\partial\Omega$. 

Given 
$\bu^{n}$, define $\bu_\ell^{n+1}$ by 
\begin{equation}\label{eq:iterative-strong-local}
\left\{
\begin{aligned}
       & (\curl(\mu^{-1}\curl) - k^2\epsilon) \bu_\ell^{n+1}  = \mathbf{f} \quad \text{ in } \Omega_\ell,\\
       & (\mu^{-1}\curl\bu_{\ell}^{n+1}  -\ri k\lambda \bn_{\ell}\times \bu_\ell^{n+1})\times\bn_{\ell}  =  \mathbf{g}_T\quad \text{ on }  \partial\Omega_\ell \cap \partial \Omega,\\
       & (\mu^{-1}\curl\bu_{\ell}^{n+1}  -\ri k\widetilde{\lambda} \bn_{\ell}\times \bu_\ell^{n+1})\times\bn_{\ell}
        = (\mu^{-1}\curl\bu^{n} -\ri k\widetilde{\lambda} \bn_{\ell}\times \bu^{n})\times\bn_{\ell}\quad \text{ on } \partial\Omega_\ell\backslash\partial\Omega,
\end{aligned}
\right.
\end{equation}
and then define $\bu^{n+1}$ by
\begin{equation}\label{eq:global sol}
    \bu^{n+1}:=
\sum_{\ell=1}^N \chi_\ell\bu_\ell^{n+1}.
\end{equation}
\end{definition} 

Observe that the parallel Schwarz method can be viewed \emph{either} as taking $\bu^n$ to $\bu^{n+1}$
(as in Definition \ref{def:Schwarz})
 \emph{or} as taking 
$\bm{\mathfrak{u}}^n = (\bu_1^n,\bu_2^n,\cdots,\bu_N^n)$ to  
$\bm{\mathfrak{u}}^{n+1} = (\bu_1^{n+1},\bu_2^{n+1},\cdots,\bu_N^{n+1})$ 
 (with the boundary value problems defining $\bu_\ell^{n+1}$ involving $\bu^n = \sum_{\ell=1}^n \chi_\ell \bu_\ell^n$). 

We now give two well-posedness results for the parallel Schwarz method, corresponding to each of these two interpretations. The result involving $\bm{\mathfrak{u}}^n$ requires the following notation:~let 
\begin{equation}
    \label{eq:mathbbU}
\mathbb{U}:=\prod_{\ell=1}^N\bm{U}(\Omega_\ell)\quad\mbox{and}\quad \mathbb{U}_0:=\prod_{\ell=1}^N \bm{U}_0(\Omega_\ell)
\end{equation}
equipped with the norm $\|\cdot\|^2_{\mathbb{U}}=\sum_{\ell=1}^N \|\cdot\|^2_{\bm{U}(\Omega_\ell)}$.

\begin{theorem}[Well-posedness of the parallel  Schwarz method in $\mathbb{U}$]\label{thm:wellpose}
Let  $\mu, \epsilon, \lambda, \mathbf{f},$ and $\mathbf{g}$ be as in the Maxwell interior impedance problem of Definition \ref{def:IIP}. Suppose $\Omega_\ell,~\ell=1,2,\cdots, N$ are bounded Lipschitz open sets, and let $\widetilde{\lambda}$ be as in Definition \ref{def:Schwarz}.
Suppose further that each $\Omega_\ell$ is simply connected with connected boundary.
    Given $\bm{\mathfrak{u}}^n = (\bu_1^n,\bu_2^n,\cdots,\bu_N^n)\in \mathbb{U}(\Omega)$, let $\bu_\ell^{n+1} (\ell=1,2,\cdots,N)$ be defined by \eqref{eq:iterative-strong-local}. Then $\bm{\mathfrak{u}}^{n+1}\in \mathbb{U}(\Omega)$.
\end{theorem}
\begin{proof}
By the definition of $\mathbb{U}$ \eqref{eq:mathbbU}, 
we need to show that $\bu_\ell^{n+1}\in \bm{U}(\Omega_\ell)$ for all $\ell=1,\ldots, N$. Since $\bu_\ell^{n+1}$ satisfies the PDE \eqref{eq:iterative-strong-local}, by Lemma \ref{lem:well pose}, it is sufficient to show that 
$(\mu^{-1}\curl\bu^{n} -\ri k\widetilde{\lambda} \bn_{\ell}\times \bu^{n})\times\bn_{\ell} \in \bL^2_T(\partial \Omega_\ell\setminus \partial\Omega)$.
Since $\bu^n:= \sum_{j=1}^N \chi_j \bu_j^n$,
    \begin{align}\nonumber
        &(\mu^{-1}\curl\bu^{n} -\ri k\widetilde{\lambda} \bn_{\ell}\times \bu^{n})\times\bn_{\ell}\\ \nonumber
        &\qquad=\bigg(\mu^{-1}\curl\bigg(\sum_{j=1}^N \chi_j \bu_j^n\bigg) -\ri k\widetilde{\lambda} \bn_{\ell}\times \bigg(\sum_{j=1}^N \chi_j \bu_j^n\big)\big)\times\bn_{\ell}\\  \nonumber
        &\qquad= \sum_{j=1}^N \bigg(\mu^{-1}(\nabla\chi_j\times\bu_j^n + \chi_j\curl\bu_j^n) -\ri k\widetilde{\lambda} \bn_{\ell}\times (\chi_j \bu_j^n)\bigg)\times\bn_{\ell}\\
        &\qquad= \sum_{j=1}^N \chi_j(\mu^{-1}\curl\bu_j^n -\ri k\widetilde{\lambda} \bn_{\ell}\times \bu_j^n)\times\bn_{\ell} + \sum_{j=1}^N \mu^{-1}(\nabla\chi_j\times\bu_j^n)\times\bn_{\ell}.
        \label{e:lastDay1}
    \end{align}
By Lemma \ref{lem:bound imp}(i),
$\bu_j^n \in \bm{U}(\Omega_j)$ has an $L^2$ impedance trace on any Lipschitz interface in $\Omega_j$; therefore $(\mu^{-1}\curl\bu_j^n -\ri k\widetilde{\lambda} \bn_{\ell}\times \bu_j^n)\times\bn_{\ell}\in \bL^2_T(\partial \Omega_\ell\setminus \partial\Omega).$

To deal with the second term on the right-hand side of \eqref{e:lastDay1}, first observe that, by assumption, $\bu_j^n \in \bU(\Omega_j)= \bW(\Omega_j)$ by Theorem \ref{thm:equivalent space}, so that $\bu_j^n \in \bH^{1/2}(\Omega_j)$ by the definition of $\bW(\Omega_j)$ \eqref{e:W}. Thus, for any $\ell$ such that $\partial \Omega_\ell \cap \Omega_j\neq \emptyset$, Theorem \ref{thm:Costabel} implies that $\bu_j^n \in \bL^2_T(\partial\Omega_\ell\setminus \partial\Omega)$ (since  
$\bu_j^n=(\bn_\ell\times \bu_j^n)\times \bn_\ell + (\bn_\ell\cdot \bu_j^n) \bn_\ell$ on $\partial \Omega_\ell\setminus \partial\Omega$).
Since $\chi_j\in C^{1,1}(\overline{\Omega_j})= W^{1,\infty}(\Omega_j)$, $\nabla\chi_j\in L^\infty$ and thus 
$(\nabla\chi_j\times\bu_j^n)\times\bn_{\ell}\in \bL^2_T(\partial \Omega_\ell\setminus \partial\Omega)$.

 We have therefore shown that 
$(\mu^{-1}\curl\bu^{n} -\ri k\widetilde{\lambda} \bn_{\ell}\times \bu^{n})\times\bn_{\ell} \in \bL^2(\partial \Omega_\ell\setminus \partial\Omega)$ and the proof is complete.
\end{proof}

The second well-posedness result requires a regularity assumption.

\begin{assumption}[Regularity of $\bu_\ell^{n+1}$]\label{ass:regularity}
$\{\Omega_\ell\}_{\ell=1}^N, \mu, \epsilon, \lambda,$ $\bg_T$, and $\{\chi_\ell\}_{\ell=1}^N$ are such that the solution $\bu_\ell^{n+1}$ of \eqref{eq:iterative-strong-local} is such that $(\nabla \chi_\ell \cdot \nabla) \bu_\ell^{n+1} \in \bL^2(\Omega_\ell)$.  
    \end{assumption}

We first prove a well-posedness result under Assumption \ref{ass:regularity}, and then discuss why Assumption \ref{ass:regularity} is required and when we expect it to hold (see Remark \ref{rem:commutator1}, 
Conjecture \ref{con:regularity}, and the associated discussion below). 

\begin{theorem}[Well-posedness of the parallel Schwarz method in $\bU(\Omega)$ under Assumption \ref{ass:regularity}]
\label{thm:wellpose2}
Let  $\mu, \epsilon, \mathbf{f},$ and $\mathbf{g}$ be as in the Maxwell interior impedance problem of Definition \ref{def:IIP}.
Suppose, additionally, that
\begin{itemize}
\item each $\chi_\ell \in C^{1,1}(\overline{\Omega_\ell})$ for all $\ell=1,\ldots, N$,
\item $\mu,\epsilon\in W^{1,\infty}(\Omega)$, 
\item $\mathbf{f}\in \bH(\difdiv,\Omega)$, 
\item
$\{\Omega_\ell\}_{\ell=1}^N, \mu, \epsilon, \lambda,$ $\bg_T$, and $\{\chi_\ell\}_{\ell=1}^N$ are such that Assumption \ref{ass:regularity} holds, \item each $\Omega_\ell$ is simply connected with connected boundary, and
\item $\widetilde{\lambda}$ is as in Definition \ref{def:Schwarz}.
\end{itemize}
Then, given $\bu^n\in \bU(\Omega)$, $\bu^{n+1}$ defined by Definition \ref{def:Schwarz} is in $ \bU(\Omega)$.
\end{theorem}

\begin{proof}
By Lemma \ref{lem:bound imp}(i) and the fact that $\bu^n\in \bU(\Omega)$ (by assumption), the impedance data in the problem \eqref{eq:iterative-strong-local} defining $\bu_\ell^{n+1}$ is in $\bL^2(\partial \Omega_j)$.
We now apply Lemma \ref{lem:well pose} to obtain that $\bu_\ell^{n+1}\in \bU(\Omega_j)$; the assumptions of Lemma \ref{lem:well pose} are satisfied with $\Omega$ replaced by $\Omega_j$ by the fact about the impedance data noted above, along with the assumptions that $\difdiv {\bf f}\in L^2(\Omega)$,
$\mu,\epsilon\in W^{1,\infty}(\Omega)$, and $\widetilde{\lambda}$ is as in Definition \ref{def:Schwarz}.

We now need to show that $\bu_\ell^{n+1}:= \sum_{\ell=1}^N \chi_\ell \bu_\ell^{n+1}$ is in $\bU(\Omega)$. 
By linearity, it is sufficient to show that $\chi_\ell \bu_\ell^{n+1}\in \bU(\Omega)$ for all $\ell=1,\ldots,N$.

Since $\mu\in W^{1,\infty}(\Omega_j)$, $\bU(\Omega_j)=\bW(\Omega_j)$ by Theorem \ref{thm:equivalent space}, and thus $\bu_\ell^{n+1}$ and $\curl \bu_\ell^{n+1}$ are both in $\bH^{1/2}(\Omega_j)$. 
Recall that the product of a $C^{0,1}$ function and an $H^{1/2}$ function is in $H^{1/2}$ by, 
e.g., \cite[Theorem 1.4.1.1]{Gr:85}.
Thus, since $\chi$ and $\nabla\chi$ are both $C^{0,1}$, $\chi_\ell \bu_\ell^{n+1}$ and $\curl (\chi_\ell \bu_\ell^{n+1}) = \nabla \chi_\ell \times \bu_\ell^{n+1} + \chi_\ell \curl \bu_\ell^{n+1}$ are both in $\bH^{1/2}(\Omega_\ell)$ and hence (by extension by zero) in $\bH^{1/2}(\Omega)$. By  
 Theorem \ref{thm:equivalent space} and the definition of $\bW(\Omega)$ \eqref{e:W}, it therefore remains to show that 
$$
\curl\bigl(\mu^{-1}\curl(\chi_\ell \bu_\ell^{n+1})\bigr)\in \bL^2(\Omega)\quad\text{ and } \quad
\difdiv(\chi_\ell \bu_\ell^{n+1}) \in L^2(\Omega).
$$
Since $\chi_\ell$ is supported in $\Omega_\ell$, 
it is sufficient to show that 
$$
\curl\bigl(\mu^{-1}\curl(\chi_\ell \bu_\ell^{n+1})\bigr)\in \bL^2(\Omega_\ell)\quad\text{ and } \quad
\difdiv(\chi_\ell \bu_\ell^{n+1}) \in L^2(\Omega_\ell).
$$
For the condition on the divergence, 
taking the divergence of the PDE in \eqref{eq:iterative-strong-local} and using
the assumption that $\difdiv\mathbf{f}\in L^2(\Omega)$, we see that  
$\difdiv(\epsilon \bu_\ell^{n+1}) \in L^2(\Omega)$; since $\epsilon \geq \epsilon_{\min}>0$ by \eqref{eq:coe_bound}, this implies that  $\difdiv \bu_\ell^{n+1} \in L^2(\Omega)$
and then $\difdiv(\chi_\ell \bu_\ell^{n+1}) \in L^2(\Omega)$ by the product rule for differentiation and the fact that $\chi \in C^{0,1}(\overline{\Omega})=W^{1,\infty}(\Omega)$. 
Now 
\begin{align*}
\quad\curl\bigl(\mu^{-1}\curl(\chi_\ell\bu_\ell^{n+1})\bigr)
&=\curl\bigl(\chi_\ell\mu^{-1}\curl\bu_\ell^{n+1}\bigr)
+
\curl\bigl(\mu^{-1}\nabla\chi_\ell\times\bu_\ell^{n+1}\bigr)  \\
 &\hspace{-0.5cm}=
\nabla\chi_\ell\times\bigl(\mu^{-1}\curl\bu_\ell^{n+1}\bigr)
+
\chi_\ell\curl\bigl(\mu^{-1}\curl\bu_\ell^{n+1}\bigr) 
+
\nabla\mu^{-1}\times
\bigl(\nabla\chi_\ell\times\bu_\ell^{n+1}\bigr)
\\
&\hspace{3cm}+
\mu^{-1}\curl\bigl(\nabla\chi_\ell\times\bu_\ell^{n+1}\bigr).
\end{align*}
The PDE \eqref{eq:iterative-strong-local} implies that $\curl\bigl(\mu^{-1}\curl\bu_\ell^{n+1}\bigr)\in \bL^2(\Omega_\ell)$. Since $\bu_\ell^{n+1}, \curl \bu_\ell^{n+1}\in \bL^2(\Omega)$, to show that 
$\curl(\mu^{-1}\curl(\chi_\ell\bu_\ell^{n+1}))\in \bL^2(\Omega)$ we only need to show that 
$$\curl\bigl(\nabla\chi_\ell\times\bu_\ell^{n+1}\big)
=
(\bu_\ell^{n+1}\cdot\nabla) \nabla\chi_\ell 
- (\nabla\chi_\ell\cdot\nabla)\bu_\ell^{n+1} + (\nabla \cdot \bu_{\ell}^{n+1})\nabla\chi_\ell - (\Delta\chi_\ell)\bu_\ell^{n+1}
$$
is in $\bL^2(\Omega)$. 
However, 
since $\nabla \chi_\ell \in C^{0,1}(\overline{\Omega})= W^{1,\infty}(\Omega)$, 
and by Assumption \ref{ass:regularity},
all the terms in the last displayed equation are in $\bL^2(\Omega)$ and the proof is complete.
\end{proof}

\begin{remark}[The origin of Assumption \ref{ass:regularity}]
\label{rem:commutator1}
In the proof of Theorem \ref{thm:wellpose2}, we saw that 
\begin{equation}\label{e:commutator}
\curl\bigl(\mu^{-1}\curl(\chi_\ell\bu_\ell^{n+1})\bigr)- \chi_\ell \curl \big(\mu^{-1}\curl \bu_\ell^{n+1}\big)
\end{equation}
involved $(\nabla\chi_\ell \cdot\nabla) \bu_\ell^{n+1}$ --  i.e., first derivatives of $\bu_\ell^{n+1}$ -- and such derivatives are not controlled by $\bu_\ell^{n+1}$ belonging to $\bU(\Omega_\ell)$.

We highlight that this issue is not present in the Helmholtz setting, since the Helmholtz analogue of $\bU(\Omega)$, $U(\Omega):=\{ u\in H^1(\Omega): \Delta u\in L^2(\Omega), \partial u/\partial n\in L^2(\partial\Omega)\}$, contains $H^1(\Omega)$, and so controls the first derivatives in  
$\Delta (\chi_\ell u_\ell^{n+1})- \chi_\ell \Delta u_\ell^{n+1}$ (see \cite[Lemma 2.11]{gong2022convergence}, with the well-posedness result in $\bm{U}(\Omega)$ then \cite[Theorem 2.12]{gong2022convergence}).
\end{remark}

\paragraph{Discussion on when we expect Assumption \ref{ass:regularity} to hold.}
Without specific assumptions on the geometry of the subdomains and the partition of unity $\{\chi_\ell\}_{\ell=1}^N$, the term  
$(\nabla\chi_\ell \cdot\nabla) \bu_\ell^{n+1}$ 
involves arbitrary first derivatives of $\bu_\ell^{n+1}$.
However, $\bH^1$ regularity of Maxwell solutions requires that the impedance data be $\bH^{1/2}$ 
(since the impedance boundary condition involves the tangential trace, which is in $\bH^{1/2}$ if the function is in $\bH^1$) and this will not be the case for $\bu_\ell^{n+1}$ because of the lack of compatibility 
between $\bg_T$ and the impedance trace of $\bu^n$ 
in \eqref{eq:iterative-strong-local} when $\partial \Omega_\ell\setminus \partial\Omega$ meets $\partial\Omega_\ell \cap \partial\Omega$.

\begin{conjecture}\label{con:regularity}
Suppose that
(i) 
$$\Omega :=\big\{ (x_1,x_2,x_3) \in \mathbb{R}^3 \,|\, x_1 \in [0,L],\,(x_2,x_3)\in \Omega_p \big\},
$$
where $\Omega_p\subset \mathbb{R}^2$ is a Lipschitz polygon, (ii) each $\chi_\ell$ is a function of only the $x_1$ variable, and (iii) each $\chi_\ell$ is zero in a neighbourhood of $\partial\Omega_\ell\setminus \partial\Omega$.
Then Assumption \ref{ass:regularity} holds.
\end{conjecture}

The assumption in 
Conjecture \ref{con:regularity} that each $\chi_\ell$ is a function of only the $x_1$ variable essentially restricts attention to the strip decompositions studied from \S\ref{sec:strip} onward, where $\Omega_\ell$ is only overlapped by $\Omega_{\ell-1}$ and $\Omega_{\ell+1}$ (strictly speaking, Conjecture \ref{con:regularity} allows more overlaps, but in practice these would be redundant). 

Proving Conjecture \ref{con:regularity} is beyond the scope of this paper, but we now outline why we expect it to hold. 

In the set up of Conjecture \ref{con:regularity}, 
each 
${\rm supp}(\nabla\chi_\ell)$ is an annular-type region whose boundary intersects $\partial\Omega_\ell$ only on $\partial\Omega_\ell\cap\partial\Omega$; i.e., in this region (where we care about the regularity of $\bu_\ell^{n+1}$) the impedance data for $\bu_\ell^{n+1}$ only comes from $\bg_T$ (and so there is not the matching issue described above).

The Maxwell solution $\bu_\ell^{n+1} \in \bL^2(\Omega_\ell)$ by wellposedness of the interior impedance problem (Lemma \ref{lem:stability}). $\bu_\ell^{n+1}$ has weak singularities near the edges of $\Omega_\ell$ created by the corners of the 2-d polygon $\Omega_p$; i.e., $\bu_\ell^{n+1} \sim r^\alpha$ for $\alpha$ depending on $\Omega_p$, where $r$ is the distance to the corner of $\Omega_p$ measured in polar coordinates in the $(x_2,x_3)$ plane. 
However, one expects the derivatives of $\bu_\ell^{n+1}$ in the $x_1$ direction to have the same regularity as $\bu_\ell^{n+1}$ itself (since differentiating in the $x_1$ direction will not affect the $r^\alpha$ asymptotics); i.e., we expect $(\nabla \chi_\ell \cdot\nabla)\bu_\ell^{n+1} \in \bL^2(\Omega_\ell)$.

For perfect electric conductor boundary conditions, singularities of Maxwell solutions were famously characterised in terms of singularities of Laplace solutions in \cite{costabel2000singularities, costabel1999singularities}, and global Maxwell regularity results were obtained from corresponding Laplace ones in, e.g., \cite[\S4.5d]{CoDaNi:10}, \cite[Theorem 5.5.5]{moiola2011trefftz}. We expect that these type of arguments can also be used to prove Conjecture \ref{con:regularity}.

\subsection{A norm on the space \texorpdfstring{$\bm{U}_0(D)$}{U_0(D)}}

The following result assumes that the impedance parameter $\lambda$ is constant; recall that this is natural when the impedance boundary condition is considered as a first-order approximation to the radiation condition. Without loss of generality, we set $\lambda$ to one.

\begin{theorem}[Norm on $\bm{U}_0(D)$]\label{thm:equivalent norm}
   Suppose that $\mu$ and $\epsilon$ satisfy the assumptions of Lemma \ref{lem:stability}, and, additionally, $\lambda=1$. Then
\begin{equation*}
    \|\bu\|^2_{1,k,\partial D}:=\|\mu^{-1}\curl\bu\times\bn\|^2_{0,\partial D}+k^2\|\bu_T\|^2_{0,\partial D}
\end{equation*}
is  a norm on $\bm{U}_0(D)$
equivalent to $\|\cdot\|_{\bm{U}(D)}$ \eqref{eq:normU}.
Furthermore,
\begin{equation}\label{eq:=boundary}
    \|\bu\|_{1,k,\partial D}^2=\|\pm\mu^{-1}\curl\bu\times\bn-ik\bu_T\|_{0,\partial D}^2.
\end{equation}
\end{theorem}    \begin{proof}
        For $\bu\in \bm{U}_0(D)$ and $\lambda=1$, 
\begin{equation*}
            \begin{aligned}
                0=&\int_D(\curl(\mu^{-1}\curl\bu)-k^2\epsilon\bu)\cdot\overline{\bu}\\ 
                =&\int_D(\mu^{-1}\curl\bu\cdot\curl\overline{\bu}-k^2\epsilon|\bu|^2)+\int_{\partial D} \bn\times\mu^{-1}\curl\bu\cdot\overline{\bu}_T\\ 
=&\int_D(\mu^{-1}|\curl\bu|^2-k^2\epsilon|\bu|^2)+\int_{\partial D} \bn\times\mu^{-1}\curl\bu\cdot\overline{\bu}_T,
            \end{aligned}
         \end{equation*}
         so that
         \begin{equation*}
            \Im\left(\int_{\partial D} \bn\times\mu^{-1}\curl\bu\cdot\overline{\bu}_T\right)=0.
         \end{equation*}
         Thus, 
         \begin{equation*}
            \begin{aligned}
                \|\pm\mu^{-1}\curl\bu\times\bn-ik\bu_T\|_{0,\partial D}^2&=\int_{\partial D}(|\mu^{-1}\curl\bu\times\bn|^2+k^2|\bu_T|^2)\\ 
                &\quad\quad\mp2k\Im\left(\int_{\partial D} \mu^{-1}\curl\bu\times\bn\cdot\overline{\bu}_T\right)\\
                &=\|\mu^{-1}\curl\bu\times\bn\|^2_{0,\partial D}+k^2\|\bu_T\|^2_{0,\partial D}.
            \end{aligned}
         \end{equation*}
         For the norm equivalence, by the triangle inequality and the definition of $\|\cdot\|_{\bm{U}(D)}$, 
         \begin{equation*}
            \begin{aligned}
\|\bu\|_{1,k,\partial D}^2&=\|\mu^{-1}\curl\bu\times\bn-ik\bu_T\|_{0,\partial D}^2\\&\leq\|\mu^{-1}\curl\bu\times\bn\|^2_{0,\partial D}+k^2\|\bu_T\|^2_{0,\partial D}\leq C \|\bu\|_{\bm{U}(D)}^2.
            \end{aligned}
         \end{equation*}
         Furthermore, for $\bu \in \bm{U}_0(D)$, by \eqref{eq:stabilityU},
         \begin{equation*}
    \|\bu\|_{\bm{U}(D)}^2\leq C\|\mu^{-1}\curl\bu\times\bn-ik\bu_T\|_{0,\partial D}^2=C\|\bu\|_{1,k,\partial D}^2.
         \end{equation*}
     \end{proof}
Given the norm $\|\cdot\|_{1,k,\partial D}$ on $\bm{U}_0(D)$,
     we define a norm on $\mathbb{\bm{U}}_0$ by
     \begin{equation}\label{eq:normU0}
\|\mathfrak{u}\|_{\mathbb{U}_0}^2:=\sum_{\ell=1}^N\|\bu_\ell\|_{1,k,\partial\Omega_\ell}^2.
     \end{equation}

\section{The error propagation operator and impedance-to-impedance map}\label{sec:errorprop}

\subsection{The error propagation operator}

Let $\bu$ be the solution to \eqref{eq:iip} and $\bu_\ell^n$ be
defined by \eqref{eq:iterative-strong-local} (i.e., $\bu_\ell^n$ is 
the local solution on $\Omega_\ell$ at the $n$-th iteration). Let 
\begin{equation*}
	\be_\ell^{n}:=\bu|_{\Omega_\ell} - \bu_\ell^{n};
\end{equation*}
i.e., 
$\be_\ell^{n}$ is the local error on $\Omega_\ell$. 
Then the global error is given by 
\begin{equation*}
    \be^n:=\bu-\bu^n=\sum_{\ell=1}^N\chi_\ell\bu_\ell-\sum_{\ell=1}^N\chi_\ell\bu_\ell^n=\sum_{\ell=1}^N\chi_\ell\be_\ell^n.
\end{equation*}
We denote the impedance data for $\bv\in \bm{U}(\Omega_\ell)$ by
\begin{equation}\label{eq:imp_map}
	\imp{\ell}{\bv}:=(\mu^{-1}\curl\bv-\ri k\lambda \bn_\ell\times\bv)\times\bn_\ell \quad\text{on }
    \partial\Omega_\ell
\end{equation}
With this definition, 
\eqref{eq:iip} and \eqref{eq:iterative-strong-local} imply that 
$\be_\ell^{n+1}$ satisfies
\begin{equation} \label{eq:format1}                       
            \left\{\begin{array}{ll}    
                            (\curl(\mu^{-1}\curl) - k^2\epsilon)\be_\ell^{n+1} = \bm{0} &\text{ in }\Omega_\ell,\\
                            \imp{\ell}{\be_\ell^{n+1}} =\imp{\ell}{\be^{n}}=\sum_{j=1}^N \imp{\ell}{\chi_j \be_j^n} &\text{ on }\partial\Omega_\ell\backslash\partial\Omega,\\
                            \imp{\ell}{\be_\ell^{n+1}} = \bm{0} &\text{ on }\partial\Omega_\ell\cap\partial\Omega.
          \end{array}\right.                         
\end{equation}
This motivates us to define the local error propagation operator $\mathcal{T}_{\ell,j}:\bm{U}_0(\Omega_j)\rightarrow \bm{U}_0(\Omega_\ell)$ such that for all $\bv_j\in \bm{U}_0(\Omega_j)$, $\mathcal{T}_{\ell,j}\bv_j\in \bm{U}_0(\Omega_\ell)$ is the solution to  
\begin{equation}\label{eqn:local_error}
            \left\{\begin{array}{ll}
                            (\curl(\mu^{-1}\curl) - k^2\epsilon)\mathcal{T}_{\ell,j}\bv_j = \bm{0} &\text{ in }\Omega_\ell,\\
			    \imp{\ell}{\mathcal{T}_{\ell,j}\bv_j} =\imp{\ell}{\chi_j\bv_j} &\text{ on } \partial\Omega_\ell\setminus\partial\Omega,\\
                            \imp{\ell}{\mathcal{T}_{\ell,j}\bv_j} = \bm{0} &\text{ on }\partial\Omega_\ell\cap\partial\Omega
          \end{array}\right.
\end{equation}
Comparing \eqref{eq:format1} and \eqref{eqn:local_error}, we see that
\begin{equation}\label{eq:errorsum}
    \be_{\ell}^{n+1} = \sum_{j=1}^N \cT_{\ell,j} \be_j^{n}.
\end{equation}
Observe that, from the definition \eqref{eqn:local_error} and the uniqueness of the interior impedance problem, for $\mathcal{T}_{\ell,j}$ to be non-zero, ${\rm supp}\chi_j$ must intersect $\partial\Omega_\ell$ (we use this fact below to simplify $\boldsymbol{\mathcal{T}}$ in the case of strip  decompositions; see \S\ref{sec:strip}).

We introduce the vector of errors $\mathfrak{e}^n\in \mathbb{U}_0$ defined by
  $$
  \mathfrak{e}^n := (\be_1^n, \be_2^n, \dotsc, \be_N^n)^\top,
  \quad\text{ 
  so that }\quad
  \mathfrak{e}^n = \boldsymbol{\mathcal{T}}\mathfrak{e}^{n-1},
  $$
where $\boldsymbol{\mathcal{T}}:\mathbb{U}_0\rightarrow \mathbb{U}_0$ is the error propagation matrix.




\subsection{Relationship between the error propagation operator and impedance-to-impedance maps}
\label{sec:strip}

We now show how the product of local error propagation operators is related to   the \emph{impedance-to-impedance map} between interior boundaries of neighboring subdomains.
For simplicity, we restrict attention to strip domain decompositions, defined below. 
 but we emphasise that analogous results hold for general decompositions, at the cost of more complicated interface geometries and interactions between the subdomains (see \cite[\S3.2]{gong2022convergence} and the discussion in \cite[\S6.2]{gong2022convergence}).


\begin{definition}\label{def:strip}
We say that $\{\Omega_\ell\}_{\ell=1}^N$ is a \emph{strip decomposition} if 

(i) $\overline{\Omega}=\cup_{\ell=1}^N\overline{\Omega_\ell}$ and $N\geq 2$,

(ii) each $\Omega_\ell, \ell=1,\ldots,N$,
is simply connected with connected boundary 

(iii) 
for $\ell=1,\ldots,N$, the two intersections 
$\Omega_\ell \cap \Omega_{\ell \pm 1}$ are open subsets of $\mathbb{R}^3$ and nonempty, and 

(iv) $\Omega_\ell\cap \Omega_j =\emptyset$ if $|j-\ell|>1$.

\end{definition}

Recall from the discussion under \eqref{eq:errorsum} that for $\mathcal{T}_{\ell,j}$ to be non-zero, ${\rm supp}\chi_j$ must intersect $\partial\Omega_\ell$. This implies that, for a strip decomposition, all entries of the error transfer operator 
$\boldsymbol{\mathcal{T}}$ are zero apart from $\mathcal{T}_{\ell, \ell-1}$ and $\mathcal{T}_{\ell, \ell+1}$; i.e.,
\eqref{eq:errorsum} simplifies to
\begin{equation*}
	\be_{\ell}^n = \mathcal{T}_{\ell,\ell-1}\be_{\ell-1}^{n-1} + \mathcal{T}_{\ell, \ell+1}\be_{\ell+1}^{n-1}
\end{equation*}
and  $\boldsymbol{\mathcal{T}}$ has the sparse representation
   \begin{equation}\label{eq:splitT}
  \boldsymbol{\mathcal{T}} =\begin{pmatrix}
  0               &       \mathcal{T}_{1,2} \\
  \mathcal{T}_{2,1}       &       0               &\mathcal{T}_{2,3}\\
                  &       \mathcal{T}_{3,2}       &0              &\mathcal{T}_{3,4}\\
                  &                       &       \ddots  &       \ddots& \ddots\\
                  &                       &       \mathcal{T}_{N-1,N-2}&          0       &\mathcal{T}_{N-1,N}\\
                  &                       &                       &       \mathcal{T}_{N,N-1}&0
  \end{pmatrix} :  = \boldsymbol{\mathcal{L}} +\boldsymbol{\mathcal{U}},
  \end{equation}
where $\boldsymbol{\mathcal{L}}$ and $\boldsymbol{\mathcal{U}}$ denote the lower and upper parts of $\boldsymbol{\cT}$, respectively. 

We now relate products 
$\mathcal{T}_{\ell',j}\mathcal{T}_{j,\ell}$ to an appropriate impedance-to-impedance map.
For this definition, it is convenient to use the notation that
$$\Gamma_{\ell,j}: = \partial\Omega_\ell\cap\Omega_j.$$
With this notation, 
observe that for a strip  domain decomposition (as shown in Figure \ref{fig:DD_bdry}),
\begin{equation}
    \label{eq:chi-strip}
\chi_\ell|_{\Gamma_{\ell+1,\ell}}=\chi_\ell|_{\Gamma_{\ell-1,\ell}}=1.
\end{equation}

\begin{figure}[h!]
    \centering
    \includegraphics[width=0.6\linewidth]{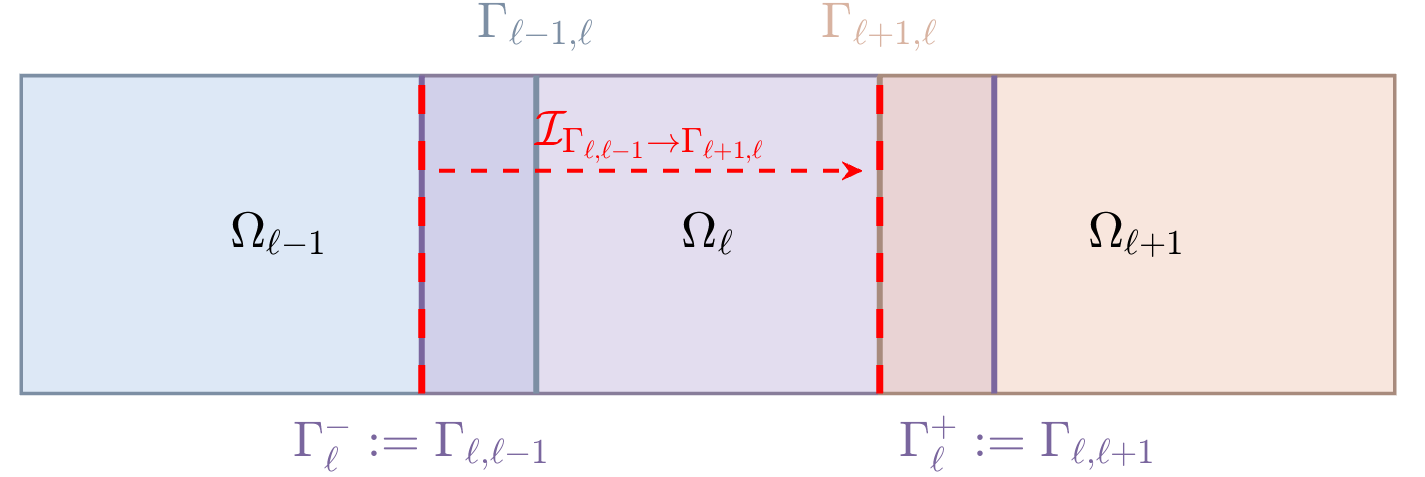}
\caption{Cross-sectional illustration of a strip decomposition (where $\Omega$ and each $\Omega_\ell$ are rectangular cuboids), showing the boundary notation and the impedance-to-impedance map.}
    \label{fig:DD_bdry}
\end{figure}


\begin{definition}\label{def:imp2imp}
	The \emph{impedance-to-impedance map} $\mathcal{I}_{\Gamma_{j,\ell}\rightarrow\Gamma_{\ell',j}}:\bL_{T}^2(\Gamma_{j,\ell})\rightarrow \bL_{T}^2(\Gamma_{\ell',j})$ is defined by 
	\begin{equation}\label{eq:imp2imp}
\mathcal{I}_{\Gamma_{j,\ell}\rightarrow\Gamma_{\ell',j}}\mathbf{g}_T:=\imp{\ell'}{\bv_j} \quad\text{on }\Gamma_{\ell',j},
	\end{equation}
	where $\bv_j\in \bm{U}_0(\Omega_j)$ is the solution of
	  \begin{equation*}
              \left\{\begin{array}{ll}
			      (\curl(\mu^{-1}\curl) - k^2\epsilon)\bv_j = \bm{0} &\text{ in }\Omega_{j},\\
			      \imp{j}{\bv_j} =\mathbf{g}_T &\text{ on }\Gamma_{j,\ell},\\
			      \imp{j}{\bv_j} = \bm{0} &\text{ on }\partial\Omega_{j}\setminus\Gamma_{j,\ell}.
            \end{array}\right.
  \end{equation*}
\end{definition}

The red dashed lines in 
in Figure~\ref{fig:DD_bdry} illustrate the map $\mathcal{I}_{\Gamma_{\ell,\ell-1}\rightarrow\Gamma_{\ell+1,\ell}}$ for a strip decomposition.

\begin{lemma}\label{lem:bdd_I}
If $\bu\in \bm{U}_0(D)$, then the impedance-to-impedance map \eqref{eq:imp2imp} is bounded as an operator from $\bL^2(\Gamma_{j,\ell})$ to $\bL^2(\Gamma_{\ell',j})$.
\end{lemma}
\begin{proof}
The result follows immediately from Lemma \ref{lem:bound imp}:
\begin{equation}\label{eq:imp2imp_bound}
\|\mathcal{I}_{\Gamma_{j,\ell}\rightarrow\Gamma_{\ell',j}}\mathbf{g}_T\|_{0,\Gamma_{\ell',j}}=\|\imp{\ell'}{\bv_j}\|_{0,\Gamma_{\ell',j}}\leq C \|\imp{j}{\bv_j}\|_{0,\partial\Omega_j}=C\|\bg_T\|_{0,\Gamma_{j,\ell}}.
\end{equation}
\end{proof}
\begin{theorem}[Relationship between products of $\cT$ and imp-to-imp maps] \label{thm:error_imp2imp}
Let $\{\Omega_\ell\}_{\ell=1}^N$ be a strip decomposition (in the sense of Definition \ref{def:strip}). Then $\cT_{j,\ell}: \bm{U}_0(\Omega_\ell)\to \bm{U}_0(\Omega_j)$ is a bounded operator and 
	\begin{equation}\label{eq:==}
		\imp{\ell'}{\cT_{\ell',j}\cT_{j, \ell}\bv_\ell} = \mathcal{I}_{\Gamma_{j,\ell}\rightarrow\Gamma_{\ell',j}}\imp{j}{\cT_{j,\ell}\bv_\ell}
        =\mathcal{I}_{\Gamma_{j,\ell}\rightarrow\Gamma_{\ell',j}}\imp{j}{\bv_\ell}
        \quad \text{ for } \bv_\ell\in \bm{U}(\Omega_\ell).
	\end{equation}
\end{theorem}
\begin{proof}
The first equality in \eqref{eq:==} follows from \eqref{eq:imp2imp}. For the second equality, observe that, 
by \eqref{eq:chi-strip}, 
$\chi_\ell$ satisfies $\nabla\chi_\ell = \bm{0}  \mbox{ on }\partial\Omega_j\backslash\partial\Omega,~j\neq \ell$. Then, by the definition of $\cT_{j,\ell}$ \eqref{eqn:local_error}, the definition of ${\rm imp}_j$ \eqref{eq:imp_map}, on $\Gamma_{j,\ell}$,
\begin{equation*}
    \begin{aligned}
        \imp{j}{\cT_{j,\ell}\bv_\ell}=\imp{j}{\chi_\ell\bv_\ell} 
        &=\mu^{-1}\curl(\chi_\ell \bv_\ell)\times\bn_j-\ri k\chi_\ell(\bv_{\ell})_T\\ 
        &=\mu^{-1}\nabla\chi_\ell\times\bv_\ell\times\bn_j+\mu^{-1}\chi_\ell\curl\bv_\ell\times\bn_j-\ri k \chi_\ell(\bv_\ell)_T\\ 
        &=\mu^{-1}\curl\bv_\ell\times\bn_j-\ri k(\bv_\ell)_T
        \\
        &=\imp{j}{\bv_\ell}.
    \end{aligned}
\end{equation*}
Finally, for the boundedness of $\cT_{j,\ell}$, 
by the definition of the local error propagation operator \eqref{eqn:local_error} and the support of $\chi_\ell$, $\imp{j}{\cT_{j,\ell\bv_\ell}}$ vanishes on $\partial\Omega_j\setminus \Gamma_{j,\ell}$. 
Thus, from the 
expression \eqref{eq:=boundary} for $\|\cdot\|_{1,k,\partial\Omega_j}$ and \eqref{eq:chi-strip}, we obtain
\begin{equation*}
\|\cT_{j,\ell}\bv_\ell\|_{1,k,\partial\Omega_j}=\|\imp{j}{\cT_{j,\ell}\bv_\ell}\|_{0,\Gamma_{j,\ell}}= \|\imp{j}{\bv_\ell}\|_{0,\Gamma_{j,\ell}}\leq \|\bv_\ell\|_{1,k,\partial\Omega_j}.
\end{equation*}
  \end{proof}

\section{Convergence of the iterative methods for strip decompositions}\label{sec:4}

In this section we use 
Theorem~\ref{thm:error_imp2imp}
to estimate powers of the error propagation operator $\boldsymbol{\mathcal{T}}$ in terms of norms of certain impedance-to-impedance maps, focusing on strip decompositions, in which case we see that there are essentially two different impedance-to-impedance maps.

We emphasise again that  one could obtain an appropriate analogue of Theorem~\ref{thm:error_imp2imp} for general decompositions (i.e., without any assumption on the subdomains $\{\Omega\}_{\ell=1}^N$) to then, in principle,  understand convergence of the parallel Schwarz method in general. For example, \cite[\S6.3]{gong2022convergence} sketches how one could use this theory in the case of a checkerboard decomposition. However, 
the interactions between the impedance-to-impedance maps are already much more complicated for the checkerboard decomposition compared to the strip decomposition.

Recall that for strip decompositions the partition of unity functions satisfy 
\eqref{eq:chi-strip} and $\boldsymbol{\mathcal{T}}=\boldsymbol{\mathcal{L}}+\boldsymbol{\mathcal{U}}$ \eqref{eq:splitT}.
We introduce the notation (also in Figure \ref{fig:DD_bdry}) that
\begin{equation}\label{eq:Gamma+_}
    \Gamma_\ell^-:= \Gamma_{\ell,\ell-1} \quad \text{ and } \quad \Gamma_\ell^+:=\Gamma_{\ell,\ell+1}.
\end{equation}

Therefore, in the strip decomposition, there four different classes of maps:
$$
\{\mathcal{I}_{\Gamma_\ell^-\to\Gamma_{\ell-1}^+}\}_{\ell=2}^N,
\quad
\{\mathcal{I}_{\Gamma_\ell^+\to\Gamma_{\ell+1}^-}
\}_{\ell=1}^{N-1},
\quad
\{\mathcal{I}_{\Gamma_\ell^-\to\Gamma_{\ell+1}^-}
\}_{\ell=1}^{N-1},
\quad
\{\mathcal{I}_{\Gamma_\ell^+\to\Gamma_{\ell-1}^+}\}_{\ell=2}^N.
$$
By symmetry, we expect 
$\mathcal{I}_{\Gamma_\ell^-\to\Gamma_{\ell-1}^+}$
to behave like
$\mathcal{I}_{\Gamma_\ell^+\to\Gamma_{\ell+1}^-}$, and 
$\mathcal{I}_{\Gamma_\ell^-\to\Gamma_{\ell+1}^-}$
to behave like
$\mathcal{I}_{\Gamma_\ell^+\to\Gamma_{\ell-1}^+}$; indeed, we see in \S\ref{sec:num4} that the norms of the first pair can be written in terms of the norm of a ``left-to-right" (``right-to-left") impedance map, and the norms of the second pair can be written in terms of the norm of a ``left-to-left" (``right-to-right") impedance map (where the adjectives ``left" and ``right" refer to the direction of the normal on the interfaces). 
We therefore let 
\begin{align}
    \rho &= \max_{\ell}\left\{ \left\|\mathcal{I}_{\Gamma_\ell^-\to\Gamma_{\ell-1}^+}\right\|_{L^2\to L^2}, \left\|\mathcal{I}_{\Gamma_\ell^+\to\Gamma_{\ell+1}^-}\right\|_{L^2\to L^2} \right\}, \label{eqn:rho} \\
    \gamma &= \max_{\ell}\left\{ \left\|\mathcal{I}_{\Gamma_\ell^-\to\Gamma_{\ell+1}^-}\right\|_{L^2\to L^2}, \left\|\mathcal{I}_{\Gamma_\ell^+\to\Gamma_{\ell-1}^+}\right\|_{L^2\to L^2} \right\},\label{eqn:gamma}
\end{align}
with these quantities finite by  Lemma \ref{lem:bdd_I}.

Since
\[
\boldsymbol{\mathcal{T}}^2=(\boldsymbol{\mathcal{L}}+\boldsymbol{\mathcal{U}})^2
=\boldsymbol{\mathcal{L}}^2+\boldsymbol{\mathcal{L}}\boldsymbol{\mathcal{U}}+\boldsymbol{\mathcal{U}}\boldsymbol{\mathcal{L}}+\boldsymbol{\mathcal{U}}^2,
\]
we see that to bound powers of $\boldsymbol{\mathcal{T}}$ it is sufficient to bound 
$\boldsymbol{\mathcal{L}}\boldsymbol{\mathcal{U}}$, $\boldsymbol{\mathcal{U}}\boldsymbol{\mathcal{L}}$, $\boldsymbol{\mathcal{L}}^2$, $\boldsymbol{\mathcal{U}}^2$.
$\boldsymbol{\mathcal{L}}$, and $\boldsymbol{\mathcal{U}}$.

  \begin{lemma}[Bounding products of $\bm{\mathcal{L}}$ and $\bm{\mathcal{U}}$ via $\rho$ and $\gamma$]\label{assump:TLU}
 Let $\{\Omega_\ell\}_{\ell=1}^N$ be a strip decomposition (in the sense of Definition \ref{def:strip})
 andw let $\boldsymbol{\mathcal{L}}$ and $\boldsymbol{\mathcal{U}}$ be defined by 
 \eqref{eq:splitT}. Then
  \begin{align}
\|\boldsymbol{\mathcal{L}}\boldsymbol{\mathcal{U}} \mathfrak{u}\|_{\mathbb{U}_0} \le \rho \|\boldsymbol{\mathcal{U}} \mathfrak{u} \|_{\mathbb{U}_0},\quad &\text{and}\quad \|\boldsymbol{\mathcal{U}}\boldsymbol  {\mathcal{L}}\mathfrak{u}\|_{\mathbb{U}_0} \le \rho\|\boldsymbol{\mathcal{L}} \mathfrak{u} \|_{\mathbb{U}_0},\label{eq:1}\\
\|\boldsymbol{\mathcal{L}}^2 \mathfrak{u}\|_{\mathbb{U}_0} \le \gamma  \|\boldsymbol{\mathcal{L}}\mathfrak{u} \|_{\mathbb{U}_0},\quad &\text{and}\quad \|\boldsymbol{\mathcal{U}}^2\mathfrak{u}\|_{\mathbb{U}_0} \le \gamma \|\boldsymbol{\mathcal{U}} \mathfrak{u} \|_{\mathbb{U}_0},\label{eq:2}
\end{align}
  for any $\mathfrak{u} \in \mathbb{U}_0 $, and 
\begin{align}
\|\boldsymbol{\mathcal{L}}\mathfrak{u}\|_{\mathbb{U}_0} \le \sqrt{\rho^2+\gamma^2}\|\mathfrak{u}\|_{\mathbb{U}_0},\quad &\text{and}\quad \|\boldsymbol{\mathcal{U}}\mathfrak{u}\|_{\mathbb{U}_0} \le \sqrt{\rho^2+\gamma^2}\|\mathfrak{u}\|_{\mathbb{U}_0},\label{eq:3}
  \end{align}
    for any $\mathfrak{u} \in \mathbb{U}_0 $ with zero impedance data on $\partial\Omega_\ell\cap\partial\Omega$. 
    \end{lemma}
  \begin{proof}
  By \eqref{eqn:local_error} and \eqref{eq:chi}, $\imp{\ell+1}{\cT_{\ell+1,\ell}\cT_{\ell,\ell+1}\bu_{\ell+1}}=0$ on $\partial\Omega_{\ell+1}\setminus\Gamma_{\ell+1,\ell}$. Combining this with \eqref{eq:normU0}, Theorem \ref{thm:error_imp2imp} and \eqref{eqn:rho}, implies that
            \begin{align*}
\|\boldsymbol{\mathcal{L}}\boldsymbol{\mathcal{U}} \mathfrak{u}\|_{\mathbb{U}_0}^2 = \sum_{\ell=1}^{N-1} \left\| \mathcal{T}_{\ell+1, \ell} \mathcal{T}_{\ell, \ell+1} \bu_{\ell+1} \right\|_{1, k, \partial \Omega_{\ell+1}}^2
&=\sum_{\ell=1}^{N-1}\|\imp{\ell+1}{\cT_{\ell+1,\ell}\cT_{\ell,\ell+1}\bu_{\ell+1}}\|_{0,\Gamma_{\ell+1,\ell}}\\
&\leq \max_{\ell} \| \mathcal{I}_{\Gamma_\ell^+ \to \Gamma_{\ell+1}^-} \|^2 \sum_{\ell=1}^{N-1} \left\| \mathcal{T}_{\ell, \ell+1} \bu_{\ell+1} \right\|_{1, k, \partial \Omega_{\ell}}^2\\&
\leq \rho^2 \|\boldsymbol{\mathcal{U}} \mathfrak{u}\|_{\mathbb{U}_0}^2.
\end{align*}
The proofs of the other three inequalities in \eqref{eq:1} and \eqref{eq:2} are similar.

Using the fact that \(\chi_\ell=1\) on \(\Gamma_{\ell+1}^-\) and
\(\chi_\ell=0\) on \(\Gamma_{\ell+1}^+\), together with the definition of
\(\boldsymbol{\mathcal L}\), Definition~\ref{def:imp2imp}, and the estimates
\eqref{eqn:rho}--\eqref{eqn:gamma}, we obtain \eqref{eq:3}. More precisely,
for any \(\bu_\ell\in\bm U_0(\Omega_\ell)\) 
with zero impedance data on $\partial\Omega_\ell\cap\partial\Omega$, 
\[
    \bu_\ell=\bu_\ell^-+\bu_\ell^+,
\]
where \(\bu_\ell^-\) and \(\bu_\ell^+\) have zero impedance data on
\(\Gamma_\ell^+\) and \(\Gamma_\ell^-\), respectively. By \eqref{eq:=boundary},
\[
    \|\bu_\ell\|_{1,k,\partial\Omega_\ell}^2
    =
    \|\bu_\ell^-\|_{1,k,\partial\Omega_\ell}^2
    +
    \|\bu_\ell^+\|_{1,k,\partial\Omega_\ell}^2.
\]
Using \eqref{eq:normU0}, \eqref{eqn:local_error}, \eqref{eq:chi-strip}, Definition \ref{def:imp2imp}, \eqref{eqn:rho} and \eqref{eqn:gamma}, in that order, we have
\begin{align*}
    \|\boldsymbol{\mathcal L}\mathfrak u\|_{\mathbb U_0}^2
    =
    \sum_{\ell=1}^{N-1}
    \|\mathcal T_{\ell+1,\ell}\bu_\ell\|^2_{1,k,\partial\Omega_{\ell+1}}
    &=\sum_{\ell=1}^{N-1}\|\imp{\ell+1}{\cT_{\ell+1,\ell}\bu_\ell}\|_{0,\Gamma_{\ell+1}^-}^2\\
    &=\sum_{\ell=1}^{N-1}\|\imp{\ell+1}{\bu_\ell}\|_{0,\Gamma_{\ell+1}^-}^2\\
    &\le
    \sum_{\ell=1}^{N-1}
    \left(
    \gamma\|\bu_\ell^-\|_{1,k,\partial\Omega_\ell}
    +
    \rho\|\bu_\ell^+\|_{1,k,\partial\Omega_\ell}
    \right)^2
    \\
    &\le
    (\rho^2+\gamma^2)
    \sum_{\ell=1}^{N-1}
    \|\bu_\ell\|_{1,k,\partial\Omega_\ell}^2
    =
    (\rho^2+\gamma^2)
    \|\mathfrak u\|_{\mathbb U_0}^2 .
\end{align*}
The proof of the second inequality in \eqref{eq:3} is similar.
\end{proof}
 \begin{theorem}[Bound on $\boldsymbol{\mathcal{T}}^N$ ]\label{thm:TLU}
Let the error propagation operator
\(\boldsymbol{\mathcal{T}}:\mathbb{U}_0\to\mathbb{U}_0\)  defined by
\eqref{eq:splitT} and 
\(\rho\) defined by 
\eqref{eqn:rho}
and 
\(\gamma\) defined by \eqref{eqn:gamma}. Then
 $$
\|\boldsymbol{\mathcal{T}}^N\|_{\mathbb{U}_0\to \mathbb{U}_0} \le 2\sqrt{\gamma^2+\rho^2}\left((\gamma+\rho)^{N-1} -\gamma^{N-1}\right).
  $$
 \end{theorem}

\begin{proof}
    Given the bounds in Lemma \ref{assump:TLU}, the proof of this result is then identical to that of 
\cite[Theorem 4.13]{gong2022convergence}; we therefore omit the details.
\end{proof}

Taylor's theorem then implies the following corollary.

\begin{corollary}\label{cor:key}
    Under the assumptions of Theorem \ref{thm:TLU}, 
if $\rho$ is sufficiently small relative to $\gamma$ and $N$ then
$\|\boldsymbol{\mathcal{T}}^N\|_{\mathbb{U}_0\to \mathbb{U}_0}$ is small 
\end{corollary}

 

\section{From the parallel Schwarz method to the discrete RAS-imp method}\label{sec:5}

\subsection{Variational formulation of the parallel Schwarz method}\label{sec:5.1}

Define  the extension-by-zero operator $\mathcal{R}_\ell^\top$ 
by
\begin{equation}\label{eq:dis_ext}
    \mathcal{R}_\ell^\top \bv_\ell 
:=
\begin{cases}
\bv_\ell \text{ in } \Omega_\ell, \\
\bm{0} \text{ in } \Omega \setminus \Omega_\ell.
\end{cases}
\end{equation}
Let $\mathcal{R}$ be defined so that
$ (\mathcal{R}_\ell F)( \bw_\ell) =  F( \mathcal{R}_\ell^\top \bw_\ell)$ for $F$  a bounded linear functional on $L^2(\Omega)$ and $\bw_\ell\in L^2(\Omega)$.
Note that, in general, $\mathcal{R}_\ell^\top \bv_\ell \notin \mathbf{X}_{\iimp}(\Omega)$, but it always belongs to
\[
\bm{Y}(\Omega) := \{ \bv \in \mathbf{L}^2(\Omega) \mid \bv_T \in \mathbf{L}_T^2(\partial\Omega) \}.
\]
Define the sesquilinear form $\alpha:\bm{U}(\Omega)\times \bm{Y}(\Omega)\rightarrow\mathbb{C}$ by
\begin{equation}\label{e:alpha}
	\alpha(\bw, \bv):=\int_\Omega (\curl(\mu^{-1}\curl\bw)-k^2\epsilon\bw)\cdot\overline{\bv} + \int_{\partial\Omega}(\mu^{-1}\curl\bw\times\bn-\ri k \lambda\bw_T)\cdot\overline{\bv}_T.
\end{equation}
Let $\mathcal{A}:\bm{U}(\Omega)\rightarrow \bm{Y}'$ be the linear operator associated to $\alpha$; i.e.,
\begin{equation*}
	(\mathcal{A}\bw, \bv) := \alpha(\bw,\bv)\quad  \text{ for all }\bw\in \bm{U}(\Omega), \bv\in \bm{Y}(\Omega).
\end{equation*}
By integration by parts/Green's identity (see, e.g., \cite[Theorem 3.31]{monk2003finite}).
\begin{equation*}
\alpha(\bu,\bv)=a(\bu,\bv)\quad \text{for all } \bu\in \bU(\Omega)\mbox{ and } \bv\in \bX_{\iimp}(\Omega).
\end{equation*}

We now use the notation that $a_D(\bu,\bv)$ denotes the sesquilinear form $a(\bu,\bv)$ \eqref{eq:bilinear_form} with $\Omega$ replaced by $D$ and $\partial\Omega$ replaced by $\partial D$.

\begin{lemma}\label{lem:boundary_imp}
For $\bw\in \bm{U}(\Omega)$ and all $\bv_\ell\in \mathbf{X}_{\iimp}(\Omega_\ell)$,
	\begin{equation}
	\langle\imp{\ell}{\bw},(\bv_\ell)_T\rangle_{\partial\Omega_\ell\setminus\partial\Omega} = a_{\Omega_\ell}(\bw, \bv_\ell) - \alpha(\bw, \mathcal{R}_\ell^\top\bv_\ell):=b_\ell(\bw,\bv_\ell).
	\label{eqn:boundary_integral}
\end{equation}
\end{lemma}

\begin{proof}
This follows from integration by parts/Green's identity.
\end{proof}

The next lemma gives the variational formulation of the parallel Schwarz  method of Definition \ref{def:Schwarz}. This lemma holds under the assumption that $\bu^n \in \bU(\Omega)$; recall that this is true by Theorem \ref{thm:wellpose2} and the regularity assumption of Assumption \ref{ass:regularity}. 

\begin{lemma}[Variational formulation of the parallel Schwarz method]
If $\bu^n \in \bU(\Omega)$, then the parallel Schwarz method of Definition \ref{def:Schwarz} is equivalent to the following.

Let  $\mu, \epsilon, \mathbf{f},$ and $\mathbf{g}$ be as in the Maxwell interior impedance problem of Definition \ref{def:IIP}. Let $F$ be defined by \eqref{e:F}. 
Given $\bu^n \in \bU(\Omega)$, define  $\bu_\ell^{n+1}\in \bm{X}_{\iimp}(\Omega_\ell)$ by
\begin{equation}\label{e:parallelVar}
    a_{\Omega_\ell}(\bu^{n+1}_\ell,\bv_\ell) = F(\mathcal{R}_\ell^\top\bv_\ell) + a_{\Omega_\ell}(\bu^{n}, \bv_\ell) - \alpha(\bu^{n}, \mathcal{R}_\ell^\top\bv_\ell)\quad \text{ for all } \bv_\ell\in \mathbf{X}_{\iimp}(\Omega_\ell), 
\end{equation}
and set $\bu^{n+1}:= \sum_{\ell=1}^N \chi_\ell \bu_\ell^{n+1}$. 
\end{lemma}

Using the operator $\mathcal{A}$, \eqref{e:parallelVar} is equivalent to 
\begin{equation*}
    \bu_\ell^{n+1} = \bu^{n}|_{\Omega_\ell} + A_\ell^{-1}\mathcal{R}_\ell(F-\mathcal{A}\bu^{n}).
\end{equation*}
With the extension operator 
related to the PoU \eqref{eq:chi} defined by
\begin{equation}\label{eq:tildeR}
\widetilde{\mathcal{R}}_\ell^\top\bv_\ell:=
    \begin{cases}
    \chi_\ell \bv_\ell  
    \quad&\mbox{on}~\Omega_\ell,\\
    \bm{0}\quad\quad&\mbox{elsewhere},
    \end{cases}
\end{equation}
the parallel Schwarz method of Definition \ref{def:Schwarz} 
is then equivalent to
	$$
	\bu^{n+1}=\sum_\ell \widetilde{\mathcal{R}}_\ell^\top\bu^{n+1}_\ell = \bu^{n} + \sum_\ell\widetilde{\mathcal{R}}_\ell^\top A_\ell^{-1}\mathcal{R}_\ell(F-\mathcal{A}\bu^{n}).
$$

\subsection{RAS-imp is a discrete version of the parallel Schwarz method}\label{sec:equivalence}

\paragraph{Finite element spaces and discrete operators.}
  Let $T^h$ be the triangulation of $\Omega$. Suppose each element $K\in T^h$ is associated with an element map which is a $C^1$-diffeomorphism $F_K:\overline{\hat{K}}\rightarrow\overline{K}$ from the reference tetrahedron $\hat{K}$, and $dF_K$ is the Jacobian matrix of the transformation. 
  Consider the N\'ed\'elec finite element space of the first type: 
	\begin{equation}
	\bm{U}_h(\Omega):=\Big\{\bu_h\in \bH(\curl;\Omega)\mid \bu_h|_K\in \mathbb{N}_{p,d} := \mathbb{P}_{p-1,d}\bigoplus\mathbb{S}_{p,d}, \quad \text{ for all } K\in T^h\Big\},~d=2,3,
	\end{equation}
where  
$$
\mathbb{S}_{p,d} := \big\{q(\mathbf{x})\in \widetilde{\mathbb{P}}_{p,d}\mid q(\mathbf{x})\cdot\mathbf{x}=0\big\}$$
and $\widetilde{\mathbb{P}}_{p,d}$ denotes the space of homogeneous polynomials of degree exactly $p$ in $d$ variables, while $\mathbb{P}_{p-1,d}$ denotes the space of polynomials of total degree at most $p-1$ in $d$ variables.

For simplicity of notation, we restrict our presentation to the lowest order
Nédélec element in this paper, although the construction can be extended to
higher-order elements. 
We denote by \(\mathcal E_h\) the set of all oriented edges of the mesh
\(T_h\). 
The local edge set associated with \(\Omega_\ell\) is defined by
\[
    \mathcal E_{h,\ell}
    :=
    \{e\in\mathcal E_h:\ e \text{ is an edge of an element }
    K\subset \Omega_\ell\}.
\]
The orientation of each local edge is inherited from the corresponding global
edge orientation. For each \(e\in\mathcal E_h\), \(M_e(\cdot)\) denotes the
corresponding edge degree of freedom:
\begin{equation*}
M_{e}(\bm{u})=\frac{1}{\text{length}(e)}\int_{e}\bm{u}\cdot\bm{\tau}  d s.
\end{equation*}

The Galerkin discretization of the Maxwell impedance problem on a general domain \(D\) is given by: find
\(\bu_h\in\bm U_h(D)\) such that
\[
    a_D(\bu_h,\bv_h)=F(\bv_h),
    \quad \quad \text{ for all } \bv_h\in\bm U_h(D).
\]

 For simplicity of notation, we denote $\bm{U}_h:=\bm{U}_h(\Omega)$, $\bm{U}_{h,j}:=\bm{U}_h(\Omega_j)$. We introduce the discrete operators  $A_h:\bm{U}_h\rightarrow \bm{U}'_h$, and  $A_{h,j}:\bm{U}_{h,j}\rightarrow \bm{U}_{h,j}'$ by
  \begin{align}
    (A_{h,j}\bv_{h,j}, \bw_{h,j}):&=a_{\Omega_j}(\bv_{h,j},\bw_{h,j}) \quad \text{ for all }\bv_{h,j},\bw_{h,j}\in \bm{U}_{h,j},\label{eq:dis_local}\\
    (A_h\bv_{h},\bw_{h}):&=a(\bv_h,\bw_h)\quad\qquad \text{ for all }\bv_h,\bw_h\in \bm{U}_h.\label{eq:dis_global}
\end{align}
Let $F_h\in \bm{U}_{h}'$ be defined by 
\begin{equation*}
    F_h(\bv_h) := F(\bv_h)\quad\text{for all }\bv_h\in \bm{U}_h.
\end{equation*}

\paragraph{Restriction, prolongation, and weighted prolongation.}
Let the prolongation operator $\mathcal{R}^\textrm{T}_{h,D}:\bm{U}_h(D)\rightarrow \bm{U}_h(\Omega)$ satisfy
\begin{equation*}
M_e(\mathcal{R}^\textrm{T}_{h,\ell}\bv_{h,\ell}) = \begin{cases}
        M_e(\bv_{h,\ell}), &\text{if } e\in \mathcal E_{h,\ell},\\
        0, &\text{if } e\not\in  \mathcal E_{h,\ell},
    \end{cases} 
\end{equation*}
and let the corresponding restriction operator $\mathcal{R}_{h, \ell}: \bm{U}_h'\rightarrow \bm{U}_{h,\ell}'$ be defined by
\begin{equation}
        (\mathcal{R}_{h, \ell}G)(\bv_{h,\ell}) := G(\mathcal{R}^{\textrm{T}}_{h, \ell}\bv_{h,\ell}),\quad G\in \bm{U}_h', \bv_{h,\ell}\in \bm{U}_{h,j}.
        \label{eqn:discrete-r}
\end{equation}

Let the weighted prolongation operator $\widetilde{\mathcal{R}}_{h,\ell}^\textrm{T}:\bm{U}_{h,\ell}\rightarrow \bm{U}_h(\Omega)$ be defined by
\begin{equation*}
M_e(\widetilde{\mathcal{R}}^\textrm{T}_{h,\ell}\bv_{h,\ell}) =
        \chi_{h,\ell,e}M_e(\bv_{h,\ell}).
\end{equation*}
where for each edge \(e\in\mathcal E_h\), $\chi_{h,\ell,e}$ satisfy
\begin{equation*}
    \sum_{\ell=1}^N \chi_{h,\ell,e} = 1;
\end{equation*}
one can set, e.g., $\chi_{h,\ell,e} = \chi_\ell(m_e)$ with $m_e$ being the midpoint of the edge $e$. For a global function \(\bv_h\), we use \(\bv_{h,\ell}\) in what follows to denote its
restriction to \(\Omega_\ell\). Then the edge-based partition-of-unity weights
imply
\begin{equation}\label{eq:tildeRh}
    \sum_{\ell=1}^N 
    \widetilde{\mathcal{R}}_{h,\ell}^{\textrm T}(\bv_{h,\ell}) 
    = \bv_h ,
\end{equation}

\paragraph{Discrete parallel Schwarz iteration.}
For any $\bv\in \mathbf{X}_{\text{imp}}(\Omega)$ and $\bw_{h,\ell}\in \bm{U}_{h,\ell}$, let
\begin{equation}\label{eqn:def-b}
	b_{h,\ell}(\bv,\bw_{h,\ell}) := a_{\Omega_\ell}(\bv,\bw_{h,\ell}) - a(\bv, \mathcal{R}^{\top}_{h,\ell}\bw_{h,\ell}).
\end{equation}
$b_{h,\ell}$ is indeed a boundary integral by the following proposition.
\begin{proposition}
\label{prop:b}
For $(\bv,\bw_{h,\ell})\in \mathbf{X}_{\text{imp}}(\Omega)  \times \bm{U}_{h,\ell}$,
          \begin{equation}\label{eqn:prop-b}
     b_{h,\ell}(\bv, \bw_{h,\ell}) = -\widetilde{a}_{\Omega\setminus\Omega_\ell}(\bv, \mathcal{R}^\top_{h,\ell}\bw_{h,\ell}) + ik\langle \lambda\bv_T, (\mathcal{R}^\top_{h,\ell}\bw_{h,\ell})_T\rangle_{\partial\Omega\setminus\partial\Omega_\ell} - ik\langle\lambda\bv_T, (\bw_{h,\ell})_T  \rangle_{\Gamma_\ell},
          \end{equation}
          where 
\begin{equation*}
\widetilde a_D(\bu,\bv)
:=
(\mu^{-1}\curl\bu,\curl\bv)_D
-
k^2(\epsilon\bu,\bv)_D.
\end{equation*}
	  Furthermore, if $\bw_{h,\ell}\in \bH_{0,\Gamma_\ell}({\rm curl},\Omega_\ell)$ where
\begin{equation}\label{eq:H0G}
          \bH_{0,\Gamma_\ell}({\rm curl},\Omega_\ell):=\{\bv\in \bH(\curl,\Omega_\ell)\mid \bv_T=\bm{0}\text{ on }\Gamma_\ell\},
      \end{equation}
      then 
	  \begin{equation*}
	  	b_{h,\ell}(\bv,\bw_{h,\ell}) = 0.
	  \end{equation*}
\end{proposition}
\begin{proof}
\eqref{eqn:prop-b} follows from the definition of $a_{\Omega_\ell}$ and $a$, and the fact that $\mathcal{R}_{h,\ell}^\top\bw_{h,\ell}$ coincides with $\bw_{h,\ell}$ in $\Omega_\ell$.
If $\bw_{h,\ell}\in \bH_{0,\Gamma_\ell}({\rm curl},\Omega_\ell)$ then $(\bw_{h,\ell})_T=\bm{0}$ on $\Gamma_\ell$.
We thus have $\mathcal{R}_{h,\ell}^\top\bw_{h,\ell}=\bm{0}$ in $\overline{\Omega}\setminus\overline{\Omega_\ell}$, 
all three terms in \eqref{eqn:prop-b}
 vanish.
 \end{proof}

\begin{assumption}[Local well-posedness]\label{assump:local-well-posedness}
For each subdomain \(\Omega_\ell\), the local discrete operator
\(A_{h,\ell}:\bm U_{h,\ell}\to \bm U_{h,\ell}'\)
is invertible.
\end{assumption}
 
\begin{definition}[Discrete parallel Schwarz method]\label{def:discrete}
          Given $\bu^n_{h}\in \bm{U}_h(\Omega)$, 
          let $\bu_{h,\ell}^{n+1}\in \bm{U}_{h,\ell}$ be the solution of the variational problem
\begin{equation}
	a_{\Omega_\ell}(\bu_{h,\ell}^{n+1}, \bv_{h,\ell}) = F_h(\mathcal{R}^\textrm{T}_{h, \ell}\bv_{h,\ell}) + b_{h,\ell}(\bu_h^n,\bv_{h,\ell}), \quad\text{ for all } \bv_{h,\ell}\in \bm{U}_{h,\ell},
          \label{eqn:discrete-local}
\end{equation}
and then set
\begin{equation}
	  \bu_h^{n+1} := \sum_{\ell=1}^N \widetilde{\mathcal{R}}^\textrm{T}_{h, \ell}\bu_{h,\ell}^{n+1}.\label{eqn:discrete-combination}
\end{equation}
\end{definition}

\begin{lemma}[Equivalence between discrete parallel Schwarz and RAS-imp Richardson]
\label{lem:discrete-schwarz-ras}
Assume that the local problems are well-posed (Assumption \ref{assump:local-well-posedness}). Then the discrete parallel
Schwarz iteration defined by \eqref{eqn:discrete-local}--\eqref{eqn:discrete-combination}
coincides with the one-level RAS-imp preconditioned Richardson iteration
\begin{equation}\label{eq:Richardson}
\bu_h^{n+1}
=
\bu_h^n
+
B_h^{-1}(F_h-A_h\bu_h^n),
\quad \text{ where } \quad
B_h^{-1}
=
\sum_{\ell=1}^N
\widetilde{\mathcal R}_{h,\ell}^{\top}
A_{h,\ell}^{-1}
\mathcal R_{h,\ell}.
\end{equation}
\end{lemma}
\begin{proof}
Using the definitions of \(A_{h,\ell}\) and \(A_h\) in
\eqref{eq:dis_local} and \eqref{eq:dis_global}, respectively, together with
the extension \eqref{eq:dis_ext} and restriction operators and the identity for
\(b_{h,\ell}\) in \eqref{eqn:def-b}, the local Schwarz problem
\eqref{eqn:discrete-local} can be written in operator form as
\[
    A_{h,\ell}\bu_{h,\ell}^{n+1}
    =
    \mathcal R_{h,\ell}F_h
    +
    A_{h,\ell}\mathcal R_{h,\ell}\bu_h^n
    -
    \mathcal R_{h,\ell}A_h\bu_h^n .
\]
By Assumption \ref{assump:local-well-posedness}, \(A_{h,\ell}\) is invertible, so that 
\[
    \bu_{h,\ell}^{n+1}
    =
    \mathcal R_{h,\ell}\bu_h^n
    +
    A_{h,\ell}^{-1}
    \mathcal R_{h,\ell}
    (F_h-A_h\bu_h^n).
\]
Substituting this expression into the combination formula
\eqref{eqn:discrete-combination}, we obtain
\[
\begin{aligned}
    \bu_h^{n+1}
    &=
    \sum_{\ell=1}^N
    \widetilde{\mathcal R}_{h,\ell}^{\top}
    \mathcal R_{h,\ell}\bu_h^n
    +
    \sum_{\ell=1}^N
    \widetilde{\mathcal R}_{h,\ell}^{\top}
    A_{h,\ell}^{-1}
    \mathcal R_{h,\ell}
    (F_h-A_h\bu_h^n)=
    \bu_h^n
    +
    B_h^{-1}(F_h-A_h\bu_h^n),
\end{aligned}
\]
where we used that $\widetilde{\mathcal R}_{h,\ell}^{\top}$ 
\eqref{eq:tildeR} satisfies 
partition-of-unity identity
\begin{equation}
        \sum_{\ell=1}^N
    \widetilde{\mathcal R}_{h,\ell}^{\top}
    \mathcal R_{h,\ell}
    =
    I .
    \label{e:discretePOU}
\end{equation}
\end{proof}

\subsection{The discrete error propagation operator}
In this subsection, we derive an operator representation of the error
propagation for the discrete parallel Schwarz iteration. The construction is
based on the local error equation and the associated discrete Maxwell-harmonic
spaces.

Define the local discrete errors and global discrete error by
\begin{equation}
    \label{eq:localDiscrete}
\be_{h,\ell}^n:=\bu_h|_{\Omega_\ell} - \bu_{h,\ell}^n\quad\text{and}\quad \be_h^n :=\bu_h - \bu_h^n.
\end{equation}
The following proposition gives the local error equation, where the new local
error is driven by the previous global error through \(b_{h,\ell}\).
\begin{proposition}[Variational characterisation of the discrete error]\label{prop:ab}
Given $\be_h^n$, 
the local discrete error $\be_{h,\ell}^{n+1}$
satisfies the variational problem
        \begin{equation*}
                a_{\Omega_\ell}(\be_{h,\ell}^{n+1}, \bv_{h,\ell}) = b_{h,\ell}(\be_h^n, \bv_{h,\ell})
\quad\text{ for all }\bv_{h,\ell}\in \bm{U}_{h,\ell},
        \end{equation*}
        and then
        \begin{equation*}
        \be_h^{n+1} = \sum_{\ell=1}^N \widetilde{\mathcal{R}}_{h,\ell}^\top\be_{h,\ell}^{n+1}.
        \end{equation*}
\end{proposition}
\begin{proof}
    By the definition of $\be_{h,\ell}^{n+1}$ \eqref{eq:localDiscrete},
        $$
            a_{\Omega_\ell}(\be_{h,\ell}^{n+1}, \bv_{h,\ell})  = a_{\Omega_\ell}(\bu_h|_{\Omega_\ell}-\bu_{h,\ell}^{n+1}, \bv_{h,\ell})\\
            =a_{\Omega_\ell}( (\bu_h-\bu_h^n)|_{\Omega_\ell}, \bv_{h,\ell}) - a_{\Omega_\ell}(\bu_{h,\ell}^{n+1}-\bu_h^n|_{\Omega_\ell}, \bv_{h,\ell}).
        $$
    For the second term on the right-hand side, by \eqref{eqn:discrete-local}, the definition of $b_{h,\ell}$ \eqref{eqn:def-b}  and the fact that $\bu_h$ is the solution to the discrete variational problem,
    $$      a_{\Omega_\ell}(\bu_{h,\ell}^{n+1}-\bu_h^n|_{\Omega_\ell}, \bv_{h,\ell})=a(\bu_h-\bu_h^n, \mathcal{R}^\top_{h,\ell}\bv_{h,\ell}) \quad \text{ for all } \bv_{h,\ell}\in \bm{U}_{h,\ell}.
        $$
        Thus, 
\begin{align*}
            a_{\Omega_\ell}(\be_{h,\ell}^{n+1}, \bv_{h,\ell}) = a_{\Omega_\ell}( (\bu_h-\bu_h^n)|_{\Omega_\ell}, \bv_{h,\ell}) - a(\bu_h-\bu_h^n, \mathcal{R}^\top_{h,\ell}\bv_{h,\ell})= b_{h,\ell}(\be_h^n, \bv_{h,\ell}),
        \end{align*}
by \eqref{eqn:def-b}.
     Finally, by \eqref{eq:tildeRh} and \eqref{eqn:discrete-combination}, $$\be_h^{n+1}  = \sum_{\ell=1}^N \widetilde{\mathcal{R}}_{h,\ell}^\top(\bu_h|_{\Omega_\ell} - \bu_{h,\ell}^{n+1}) = \sum_{\ell=1}^N \widetilde{\mathcal{R}}_{h,\ell}^\top\be_{h,\ell}^{n+1}.$$
\end{proof}


If \(\bw_{h,\ell}\in \bm U_{h,\ell}\cap
\bH_{0,\Gamma_\ell}(\curl,\Omega_\ell)\), then Proposition~\ref{prop:b}
implies \(b_{h,\ell}(\be_h^n,\bw_{h,\ell})=0\). Hence
Proposition~\ref{prop:ab} shows that \(\be_{h,\ell}^{n+1}\) is discrete
Maxwell-harmonic in \(\Omega_\ell\). This motivates the following definition.

\begin{definition}[Discrete Maxwell-harmonic space]
Let $\bm{U}_{0,h,\ell}\subset \bm{U}_{h,\ell}$ be defined by
       \begin{equation*}
                \bm{U}_{0,h,\ell} := \{\bv_{h,\ell}\in \bm{U}_{h,\ell} \mid a_{\Omega_\ell}(\bv_{h,\ell}, \bw_{h,\ell})=0\text{ for all }\bw_{h,\ell}\in \bm{U}_{h,\ell}\cap \bH_{0,\Gamma_\ell}({\rm curl}, \Omega_\ell)\}.
        \end{equation*}
         Define the product space
        \begin{equation*}
           \mathbb{U}_{0,h}:=\prod_{\ell=1}^N \bm{U}_{0,h,\ell}.
        \end{equation*}
\end{definition}
Observe that, by the definition of $\bH_{0,\Gamma_\ell}({\rm curl})$ \eqref{eq:H0G}, a zero impedance boundary condition is imposed 
in $\bm{U}_{0,h,\ell}$ on $\partial\Omega\setminus \Gamma_\ell$.

Since functions in \(\bm{U}_{0,h,\ell}\) are uniquely characterized by their impedance data on the artificial interface \(\Gamma_\ell\), it is natural to quantify them using their interface values. A formal definition of the corresponding norm, which relies on the discrete impedance operator (Definition~\ref{def:discrete_imp}), will be introduced later in Definition~\ref{def:discrete_norm}.

By Proposition \ref{prop:ab}, 
$$
    a_{\Omega_\ell}(\be_{h,\ell}^{n+1}, \bv_{h,\ell}) = b_{h,\ell}\bigg(\sum_{j=1}^N\widetilde{\mathcal{R}}_{h,j}^\top \be_{h,j}^n,\bv_{h,\ell}\bigg);
    $$
this motivates the definition of the following error transfer operator. 
\begin{definition}[Discrete error propagation operator]\label{def:error_prop}
        We define $\mathcal{T}_{h,\ell,j}:\bm{U}_{h,j}\rightarrow \bm{U}_{0,h,\ell}$. For $\bw_{h,j}\in \bm{U}_{h,j}$,  $\mathcal{T}_{h,\ell,j}\bw_{h,j}\in \bm{U}_{0, h,\ell}$ is the solution to
        \begin{equation}\label{eqn:discrete-error-prop}
                a_{\Omega_\ell}(\mathcal{T}_{h,\ell,j}\bw_{h,j}, \bv_{h,\ell}) = b_{h,\ell}(\widetilde{\mathcal{R}}_{h,j}^\top \bw_{h,j},\bv_{h,\ell}), \quad\text{for all }\bv_{h,\ell}\in \bm{U}_{h,\ell}.
        \end{equation}
\end{definition}

Just as at the continuous level, we restrict attention to strip decompositions.
Under this assumption, the discrete error transfer operator
\(\boldsymbol{\mathcal T}_h\) has the following block tridiagonal form (compare to \eqref{eq:splitT}):
   \begin{equation*}
  \boldsymbol{\mathcal{T}}_h=\begin{pmatrix}
  0               &       \mathcal{T}_{h,1,2} \\
  \mathcal{T}_{h,2,1}       &       0               &\mathcal{T}_{h,2,3}\\
                  &       \mathcal{T}_{h,3,2}       &0              &\mathcal{T}_{h,3,4}\\
                  &                       &       \ddots  &       \ddots& \ddots\\
                  &                       &       \mathcal{T}_{h,N-1,N-2}&          0       &\mathcal{T}_{h,N-1,N}\\
                  &                       &                       &       \mathcal{T}_{h,N,N-1}&0
  \end{pmatrix}.
  \end{equation*}
  We introduce the discrete error $\mathfrak{e}_h^n\in \mathbb{\bm{U}}_{0,h}$ as  
  $$
  \mathfrak{e}_h^n := (\be_{h,1}^n, \be_{h,2}^n, \dotsc, \be_{h,N}^n)^\top.
  $$
  Then
\begin{equation}\label{eqn:discrete-error-prop-h}
    \mathfrak{e}_h^n = \boldsymbol{\mathcal{T}}_h\mathfrak{e}_h^{n-1}.
\end{equation}
The operator $\boldsymbol{\mathcal T}_h$ describes the discrete error propagation in the product space of local discrete Maxwell-harmonic functions; its relationship to the usual algebraic error propagation operator of the RAS-imp
preconditioned Richardson iteration is discussed in \S\ref{sec:num3}.
  
Recall from \eqref{eq:Gamma+_}
the notation that, for a strip decomposition, $$\Gamma_\ell:=\partial\Omega_\ell\setminus\partial\Omega$$
can be written as 
$\Gamma_\ell= \Gamma_\ell^-\cup \Gamma_\ell^+$.
To distinguish
the error components entering from the left and from the right, we decompose
the discrete Maxwell-harmonic space according to the side on which the
corresponding impedance data is supported.
\begin{definition}
For $s\in\{+,-\}$,
\begin{equation} \label{eqn:U0h}
    \bm{U}_{0,h,\ell}^{s}:=\left\{\bv_{h,\ell}\in \bm{U}_{0,h,\ell}\mid a_{\Omega_\ell}(\bv_{h,\ell}, \bw_{h,\ell})=0\text{ for all }\bw_{h,\ell}\in \bm{U}_{h,\ell}\cap\bH_{0,\Gamma_\ell^s}({\rm curl},\Omega_\ell)\right\}.
\end{equation}
\end{definition} 
The following result records the directionality of the error transfer: 
\begin{proposition}\label{prop:error_range}
    \begin{equation*}
        {\rm (i)} \quad\mathcal{T}_{h,\ell,\ell-1}:\bm{U}_{h,\ell-1}\rightarrow \bm{U}_{0,h,\ell}^-\quad\text{and}\quad{\rm (ii)}\quad\mathcal{T}_{h,\ell,\ell+1}:\bm{U}_{h,\ell+1}\rightarrow \bm{U}_{0,h,\ell}^+.
    \end{equation*}
    Moreover, $\bm{\cT}_h: \mathbb{U}_h\to\mathbb{U}_{0,h}$.
\end{proposition}
\begin{proof}
We show (i) only. Let $\bw_{h,\ell-1}\in \bU_{h,\ell-1}$, since $\overline{\Omega}_j\cap\overline{\Omega}_{\ell-1}=\emptyset$ for $j\geq \ell+1$, $\widetilde{\mathcal{R}}_{h,\ell-1}^\top\bw_{h,\ell-1}=\bm{0}$ on $\bigcup\limits_{j=\ell+1}^N \Omega_j$. Therefore, $(\widetilde{\mathcal{R}}_{h,\ell-1}^\top\bw_{h,\ell-1})_T=\bm{0}$ on $\Gamma_\ell^+$. If $\bv_{h,\ell}\in \bH_{0,\Gamma_\ell^-}({\rm curl},\Omega_\ell)$, $\mathcal{R}_{h,\ell}^\top\bv_{h,\ell}$ vanishes on $\left(\bigcup\limits_{j=1}^{\ell-1} \Omega_j\right)\setminus \Omega_\ell$. Therefore, by Definition \ref{def:error_prop} and Proposition~\ref{prop:b}, 
\begin{align*}
&a_{\Omega_\ell}(\mathcal{T}_{h,\ell,\ell-1}\bw_{h,\ell-1}, \bv_{h,\ell})\\ &\quad= b_{h,\ell}(\widetilde{\mathcal{R}}_{h,\ell-1}^\top\bw_{h,\ell-1}, \bv_{h,\ell})\\
 &\quad= -\tilde{a}_{\Omega\setminus\Omega_\ell}(\widetilde{\mathcal{R}}_{h,\ell-1}^\top\bw_{h,\ell-1}, \mathcal{R}_{h,\ell}^\top\bv_{h,\ell}) + ik\langle (\widetilde{\mathcal{R}}_{h,\ell-1}^\top\bw_{h,\ell-1})_T, (\mathcal{R}_{h,\ell}^\top\bv_{h,\ell})_T\rangle_{\partial\Omega\setminus\partial\Omega_\ell} \\ &\qquad- ik\langle(\widetilde{\mathcal{R}}_{h,\ell-1}^\top\bw_{h,\ell-1})_T, (\bv_{h,\ell})_T\rangle_{\Gamma_\ell} =0.
\end{align*}
Thus, $\cT_{h,\ell,\ell-1}\bw_{h,\ell-1}\in \bm{U}_{0,h,\ell}^-$.
\end{proof}

\subsection{The discrete impedance-to-impedance map}
The operator \(\boldsymbol{\mathcal T}_h\) introduced above describes the
propagation of local discrete Maxwell-harmonic errors. In this subsection, we show that
this propagation can be equivalently described in terms of discrete
impedance data
on the interfaces. We first define the discrete
impedance data operator and then introduce the corresponding discrete
impedance-to-impedance maps. 

We define the space of tangential functions on a compact manifold $\Gamma$ that is the restriction of $\bm{U}_h$ on $\Gamma$ as
\begin{equation*}
    \bm{U}_{h,T}(\Gamma):=\left\{\bu_{\Gamma}\in \mathbf{L}_T^2(\Gamma) \mid \exists\, \bu\in \bm{U}_h \text{ such that }\bu_{\Gamma}=\bu_T \text{ on }\Gamma\right\}.
\end{equation*}
Let 
\begin{equation*}
    \bu_\Gamma = \sum_m c_m(\bm{\phi}_{\Gamma,m})_T,  
\end{equation*}
where $\{\bm{\phi}_{\Gamma,m}\}$
are all the basis functions 
satisfying $(\bm{\phi}_{\Gamma,m})_T\neq 0$ on $\Gamma$.
We denote by \(\hat{\bu}_{\Gamma,\ell}\in \bm U_{h,\ell}\) the finite element
function whose tangential degrees of freedom on \(\Gamma\) coincide with those
of \(\bu_\Gamma\), and whose remaining degrees of freedom are zero:
\begin{equation*}
    \hat{\bu}_{\Gamma,\ell}:=\sum_m c_m(\bm{\phi}_{\Gamma,m})|_{\Omega_\ell}.
\end{equation*}
  \begin{definition}[Discrete impedance data operator]\label{def:discrete_imp}
          We define the discrete impedance data operator as $\iimp^h_{\Gamma_\ell}:\bm{U}_{h,\ell}\rightarrow \bm{U}_{h,T}(\Gamma_\ell)$. For $\bv_{h,\ell}\in \bm{U}_{h,\ell}$, the discrete impedance data $\imph{\Gamma_\ell}{\bv_{h, \ell}}$ satisfies
          \begin{equation}        
                  \langle\imph{\Gamma_\ell}{\bv_{h, \ell}}, \bw_{\Gamma_\ell} \rangle_{\Gamma_\ell} = a_{\Omega_\ell}(\bv_{h,\ell}, \hat{\bw}_{\Gamma_\ell,\ell}),\quad\quad \text{ for all } \bw_{\Gamma_\ell}\in \bm{U}_{h,T}(\Gamma_\ell).
                  \label{eqn:discrete-impedance-data}
          \end{equation}
  \end{definition}

\begin{proposition}\label{prop:discrete_impedance}
          For $\bv_{h,\ell}\in \bm{U}_{0,h,\ell}$, the discrete impedance data satisfies
          \begin{equation*}
                  \langle\imph{\Gamma_\ell}{\bv_{h,\ell}}, (\bw_{h,\ell})_T\rangle_{\Gamma_\ell} = a_{\Omega_\ell}(\bv_{h,\ell}, \bw_{h,\ell})\quad\text{for all } \bw_{h,\ell}\in \bm{U}_{h,\ell}.
          \end{equation*}
  \end{proposition}
  \begin{proof}
For any $\bw_{h,\ell}\in \bm{U}_{h,\ell}$,  let $\bw_{\Gamma_\ell}\in\bU_{h,T}(\Gamma_\ell)$ be defined as $(\bw_{h,\ell})_T$ on $\Gamma_\ell$. 
$\bw_{h,\ell}-\hat{\bw}_{\Gamma_\ell,\ell}\in \bH_{0,\Gamma_\ell}({\rm curl}, \Omega_\ell)$. Since $\bv_{h,\ell}\in \bm{U}_{0,h,\ell}$, $a_{\Omega_\ell}(\bv_{h,\ell}, \bw_{h,\ell})=a_{\Omega_\ell}(\bv_{h,\ell},\hat{\bw}_{\Gamma_\ell,\ell}) + a_{\Omega_\ell}(\bv_{h,\ell}, \bw_{h,\ell}-\hat{\bw}_{\Gamma_\ell,\ell})=a_{\Omega_\ell}(\bv_{h,\ell},\hat{\bw}_{\Gamma_\ell,\ell}).$ By Definition~\ref{def:discrete_imp}, $a_{\Omega_\ell}(\bv_{h,\ell}, \bw_{h,\ell})=a_{\Omega_\ell}(\bv_{h,\ell},\hat{\bw}_{\Gamma_\ell,\ell})
  =\langle\imph{\Gamma_\ell}{\bv_{h,\ell}}, (\bw_{h,\ell})_T\rangle_{\Gamma_\ell}$.
  \end{proof}
   The last proposition establishes that for discrete Maxwell-harmonic functions, the action of the local bilinear form is fully characterized by the impedance data on the interface. 
   Based on this observation, we are now ready to formally define the discrete Maxwell-harmonic norm as follows.
\begin{definition}[Discrete Maxwell-harmonic norm]\label{def:discrete_norm}
     We define the norms on $\bm{U}_{0,h,\ell}$ and $\mathbb{U}_{0,h}$ by
     \begin{equation*}
        \|\bv_{h,\ell}\|_{\bm{U}_{0,h,\ell}}:=\|\iimp_{\Gamma_\ell}^h \bv_{h,\ell}\|_{0,\Gamma_\ell} \quad \text{and} \quad \|\mathfrak{v}_{h}\|_{\mathbb{U}_{0,h}}:=\left(\sum_{\ell=1}^N \|\bv_{h,\ell}\|_{\bm{U}_{0,h,\ell}}^2\right)^{1/2}.
     \end{equation*}
where $\bv_{h,\ell}\in \bm{U}_{0,h,\ell}$ and $\mathfrak{v}_h:=(\bv_{h,1},\dots,\bv_{h,N})^\top \in \mathbb{U}_{0,h}$.
        \end{definition}

We next record how this data is
supported on the two sides of the interface.
 \begin{lemma}\label{lem:imp_zero}
 For $\bv_{h,\ell}\in \bm{U}_{0,h,\ell}^+$, 
 \begin{equation*}
     \imph{\Gamma_\ell}{\bv_{h,\ell}}=\bm{0}\quad\text{on }\Gamma_\ell^-.
 \end{equation*}
 For $\bv_{h,\ell}\in \bm{U}_{0,h,\ell}^-$, 
 \begin{equation*}
     \imph{\Gamma_\ell}{\bv_{h,\ell}}=\bm{0}\quad\text{on }\Gamma_\ell^+.
 \end{equation*}
 \end{lemma}
 \begin{proof}
 We prove the first equation; the proof of the second is similar.
 Since $\bv_{h,\ell}\in \bm{U}_{0,h,\ell}^+$, $a_{\Omega_\ell}(\bv_{h,\ell},\hat{\bw}_{\Gamma_\ell^-,\ell})=0$ for any $\bw_{\Gamma_\ell^-}\in \bm{U}_{h,T}(\Gamma_\ell^-)$. By definition of the discrete impedance data (Definition~\ref{def:discrete_imp}), $\langle \imph{\Gamma_\ell}{\bv_{h,\ell}},\bw_{\Gamma_\ell^-}\rangle_{\Gamma_\ell^-}=0$. We thus have $\imph{\Gamma_\ell}{\bv_{h,\ell}}=\bm{0}$ on $\Gamma_\ell^-$.
 \end{proof}
 Let $s\in\{+,-\}$. For $\bv_{h,\ell}\in \bm{U}_{0,h,\ell}^s$, we define
 \begin{equation}\label{eq:imp+-}
     \imph{\Gamma_\ell^s}{\bv_{h,\ell}}:=\imph{\Gamma_\ell}{\bv_{h,\ell}}|_{\Gamma_\ell^s}.
 \end{equation}
This one-sided support property allows us to decompose a discrete
Maxwell-harmonic function into two components, corresponding to the impedance
data on \(\Gamma_\ell^-\) and \(\Gamma_\ell^+\), respectively.
\begin{lemma}
    [Decomposition of $\bm{U}_{0,h,\ell}$]\label{lem:de_space}
    For each \(\ell=1,\cdots,N\), 
\[
    \bm{U}_{0,h,\ell}
    =
    \bm{U}_{0,h,\ell}^{+}
    \oplus
    \bm{U}_{0,h,\ell}^{-}.
\]
Moreover, for $\bv_{h,\ell}
    =
    \bv_{h,\ell}^{+}
    +
    \bv_{h,\ell}^{-},
    ~
    \bv_{h,\ell}^{\pm}\in \bm U_{0,h,\ell}^{\pm},$ 
    \begin{equation}\label{eqn:norm_decomp}
    \|\bv_{h,\ell}\|_{\bm{U}_{0,h,\ell}}^2= \|\bv^+_{h,\ell}\|_{\bm{U}^+_{0,h,\ell}}^2+ \|\bv^-_{h,\ell}\|_{\bm{U}^-_{0,h,\ell}}^2,
    \end{equation}
where
$\|\cdot\|_{\bm U_{0,h,\ell}^{\pm}}
    :=
    \|\imph{\Gamma_\ell}{\cdot}\|_{0,\Gamma_\ell^\pm}.$
\end{lemma}
\begin{proof}
    For $\bw_{h,\ell}\in \bm{U}_{0,h,\ell}$, let
\(\bw_{h,\ell}^{\pm}\in\bm U_{0,h,\ell}^{\pm}\) be the discrete
Maxwell-harmonic functions determined by the one-sided impedance data
\(\imph{\Gamma_\ell^\pm}{\bw_{h,\ell}}\): for all $\bv_{h,\ell}\in\bm{U}_{h,\ell}$,
\begin{equation*}
\begin{aligned}
    a_{\Omega_\ell}(\bw_{h,\ell}^+,\bv_{h,\ell})&=\langle \imph{\Gamma_\ell^+}{\bw_{h,\ell}},(\bv_{h,\ell})_T\rangle_{\Gamma_\ell^+},\\
    a_{\Omega_\ell}(\bw_{h,\ell}^-,\bv_{h,\ell})&=\langle \imph{\Gamma_\ell^-}{\bw_{h,\ell}},(\bv_{h,\ell})_T\rangle_{\Gamma_\ell^-}.
\end{aligned}
\end{equation*}
    By Lemma \ref{lem:imp_zero} and Proposition \ref{prop:discrete_impedance}, 
    \begin{equation*}
        \begin{aligned}
            a_{\Omega_\ell}(\bw_{h,\ell}^++\bw_{h,\ell}^-,\bv_{h,\ell})&=\langle \imph{\Gamma_\ell^+}{\bw_{h,\ell}},(\bv_{h,\ell})_T\rangle_{\Gamma_\ell^+}+\langle \imph{\Gamma_\ell^-}{\bw_{h,\ell}},(\bv_{h,\ell})_T\rangle_{\Gamma_\ell^-}\\
            &=\langle \imph{\Gamma_\ell}{\bw_{h,\ell}},(\bv_{h,\ell})_T\rangle_{\Gamma_\ell}\\
            &=a_{\Omega_\ell}(\bw_{h,\ell},\bv_{h,\ell}).
        \end{aligned}
    \end{equation*}
    Thus, $\bw_{h,\ell}=\bw_{h,\ell}^++\bw_{h,\ell}^-$. 
    If \(\bw_{h,\ell}\in \bm U_{0,h,\ell}^{+}\cap \bm U_{0,h,\ell}^{-}\), then
Lemma~\ref{lem:imp_zero} gives
\[
    \imph{\Gamma_\ell^-}{\bw_{h,\ell}}=\bm{0},
    \quad
    \imph{\Gamma_\ell^+}{\bw_{h,\ell}}=\bm{0}.
\]
Hence \(\imph{\Gamma_\ell}{\bw_{h,\ell}}=\bm{0}\). By
Proposition~\ref{prop:discrete_impedance},
\[
    a_{\Omega_\ell}(\bw_{h,\ell},\bv_{h,\ell})=0,
    \quad \quad \text{ for all } \bv_{h,\ell}\in\bm U_{h,\ell}.
\]
By local well-posedness, \(\bw_{h,\ell}=\bm{0}\).
    Therefore, $\bm{U}_{0,h,\ell}=\bm{U}_{0,h,\ell}^+\oplus \bm{U}_{0,h,\ell}^-$. \eqref{eqn:norm_decomp} follows directly from the definition \ref{def:discrete_norm}.
\end{proof}
Having identified discrete Maxwell-harmonic functions with their impedance
data on the interfaces, we now define how such data is transferred from one
interface to its neighboring interfaces through a local Maxwell solve.

\begin{definition}[Discrete impedance-to-impedance map]
    \label{def:discrete_iip}
Let \(s\in\{+,-\}\).
Given \(\mathbf g_T\in \bm U_{h,T}(\Gamma_\ell^s)\), let 
 $\bw_{h,\ell}\in \bm{U}_{h,\ell}$ be the unique solution of  
\begin{equation}\label{eq:aOmegaj}
a_{\Omega_\ell}(\bw_{h,\ell},\bv_{h,\ell}) = \langle\mathbf{g}_T, (\bv_{h,\ell})_T\rangle_{\Gamma_{\ell}^s}\quad\text{for all }\bv_{h,\ell}\in \bm{U}_{h,\ell}.
\end{equation}
Let $\Omega_{j,\ell}:=\Omega_j\cap\Omega_\ell$, 
then $\cI_{\Gamma_\ell^s\to\Gamma_{\ell-1}^+}^h: \bm{U}_{h,T}(\Gamma_\ell^s)\to \bm{U}_{h,T}(\Gamma_{\ell-1}^+)$ and $\cI_{\Gamma_\ell^s\to\Gamma_{\ell+1}^-}^h: \bm{U}_{h,T}(\Gamma_\ell^s)\to \bm{U}_{h,T}(\Gamma_{\ell+1}^-)$ are defined as the solutions of 
\begin{equation}\label{eq:dis_im-im map}
\begin{aligned}
\langle \cI_{\Gamma_\ell^s\to\Gamma_{\ell-1}^+}^h \bg_T,\bv_{h,\Gamma_{\ell-1}^+}\rangle_{\Gamma_{\ell-1}^+} = a_{\Omega_{\ell,\ell-1}}(\bw_{h,\ell},\hat{\bv}^+_{h,\ell-1})-a_{\Omega_\ell}(\bw_{h,\ell},\mathcal{R}^\top_{h,\ell-1}\hat{\bv}_{h,\ell-1}^+),\\ 
\langle \cI_{\Gamma_\ell^s\to\Gamma_{\ell+1}^-}^h \bg_T,\bv_{h,\Gamma_{\ell+1}^-}\rangle_{\Gamma_{\ell+1}^-} = a_{\Omega_{\ell,\ell+1}}(\bw_{h,\ell},\hat{\bv}^-_{h,\ell+1})-a_{\Omega_\ell}(\bw_{h,\ell},\mathcal{R}^\top_{h,\ell+1}\hat{\bv}_{h,\ell+1}^-),
\end{aligned}
\end{equation}
 for all $\bv_{h,\Gamma_{\ell-1}^+}\in \bm{U}_{h,T}(\Gamma_{\ell-1}^+), 
\bv_{h,\Gamma_{\ell+1}^-}\in \bm{U}_{h,T}(\Gamma_{\ell+1}^-)$, respectively.
\end{definition}
\begin{theorem}\label{thm:dis-iip-aa}
    Let $s\in \{+,-\}$, $\bg_T\in \bm{U}_{h,T}(\Gamma_\ell^s)$ and suppose $\bw_{h,\ell}$ satisfies \eqref{eq:aOmegaj}. Then 
\begin{equation}\label{eq:lem2}
\begin{aligned}
\langle \cI_{\Gamma_\ell^s\to\Gamma_{\ell-1}^+}^h \bg_T,(\bv_{h,\ell-1})_T\rangle_{\Gamma_{\ell-1}^+}=b_{h,\ell-1}(\widetilde{\mathcal{R}}^\top_{h,\ell}\bw_{h,\ell},\bv_{h,\ell-1}), ~~\quad \text{ for all } \bv_{h,\ell-1}\in \bm{U}_{h,{\ell-1}}.\\
\langle \cI_{\Gamma_\ell^s\to\Gamma_{\ell+1}^-}^h \bg_T,(\bv_{h,\ell+1})_T\rangle_{\Gamma_{\ell+1}^-}=b_{h,\ell+1}(\widetilde{\mathcal{R}}^\top_{h,\ell}\bw_{h,\ell},\bv_{h,\ell+1}), ~~\quad \text{ for all } \bv_{h,\ell+1}\in \bm{U}_{h,{\ell+1}}.
\end{aligned}
\end{equation}
\end{theorem}
\begin{proof}
   We only provide a sketch of the proof here, as it closely follows the proof in \cite[Lemma 4.8]{gong2023convergence}. 
      Write \( \bv_{h,\ell-1} = \hat{\bv}_{h,\ell-1}^+ + \widetilde{\bv}_{h,\ell-1}^+ \), where \( (\hat{\bv}_{h,\ell-1}^+)_T \) coincides with \( (\bv_{h,\ell-1})_T \) on \( \Gamma_{\ell-1}^+ \) and is zero on other edges and \( (\widetilde{\bv}_{h,\ell-1}^+)_T \) vanishes on \( \Gamma_{\ell-1}^+ \). Thus, by Proposition \ref{prop:b}, 
\[\langle \cI^h_{\Gamma_{\ell}^s \to \Gamma_{\ell-1}^+} \bg_T, (\widetilde{\bv}_{h,\ell-1}^+)_T \rangle_{\Gamma_{\ell-1}^+} = 0 = b_{h,\ell-1} (\widetilde{\mathcal{R}}_{h,\ell}^\top \bw_{h,\ell}, \widetilde{\bv}_{h,\ell-1}^+).\]
From Definition \ref{def:discrete_iip}, Proposition \ref{prop:b} and the definition of $b_{h,\ell-1}$ in \eqref{eqn:def-b}, we can get the following result:
\begin{equation*}
\begin{aligned}
        \langle \cI^h_{\Gamma_{\ell}^s \to \Gamma_{\ell-1}^+} \bg_T, (\hat{\bv}_{h,\ell-1}^+)_T \rangle_{\Gamma_{\ell-1}^+} &= a_{\Omega_{\ell,\ell-1}}(\bw_{h,\ell},\hat{\bv}^+_{h,\ell-1})-a_{\Omega_\ell}(\bw_{h,\ell},\mathcal{R}^\top_{h,\ell-1}\hat{\bv}_{h,\ell-1}^+)\\&=b_{h,\ell-1} (\widetilde{\mathcal{R}}_{h,\ell}^\top \bw_{h,\ell}, \hat{\bv}_{h,\ell-1}^+).
\end{aligned}
\end{equation*}
\end{proof}

The next result connects the volume-based error transfer operator
\(\mathcal T_{h,j,\ell}\) with the interface-based impedance-to-impedance map.
It shows that applying \(\mathcal T_{h,j,\ell}\) to a local Maxwell-harmonic
error and then taking its impedance data is equivalent to applying the
discrete impedance-to-impedance map to the original impedance data.
\begin{theorem}\label{thm:IT}
Let $\bw_{h,\ell}\in \bm{U}_{0,h,\ell}^s$ with $s\in\{+,-\}$. For $(j,t)\in \{(\ell-1, +), (\ell+1,-)\}$,
\begin{equation}\label{eqn:imp2imp&error}
    \imph{\Gamma_j^t}{\mathcal{T}_{h,j,\ell}\bw_{h,\ell}}=\mathcal{I}^h_{\Gamma_{\ell}^s\rightarrow\Gamma_j^t}(\imph{\Gamma_\ell^s}{\bw_{h,\ell}}). 
\end{equation}
\end{theorem}
\begin{proof}
We give the proof for the case when $(j,t)=(\ell-1,+)$; the proofs of the other cases are similar. By Proposition~\ref{prop:error_range}, $\mathcal{T}_{h,\ell-1,\ell}\bw_{h,\ell}\in \bm{U}_{0,h,\ell-1}^+$. Then $\langle\imph{\Gamma_{\ell-1}^+}{\mathcal{T}_{h,\ell-1,\ell}\bw_{h,\ell}},(\bv_{h,\ell-1})_T\rangle_{\Gamma_{\ell-1}^+}=a_{\Omega_{\ell-1}}(\mathcal{T}_{h,\ell-1,\ell}\bw_{h,\ell}, \bv_{h,\ell-1})$ for all $\bv_{h,\ell-1}\in \bm{U}_{h,\ell-1}$ from Proposition \ref{prop:discrete_impedance} and Lemma \ref{lem:imp_zero}. By the definition of $\mathcal{T}_{h,\ell-1,\ell}$ (Definition~\ref{def:error_prop}),  $a_{\Omega_{\ell-1}}(\mathcal{T}_{h,\ell-1,\ell}\bw_{h,\ell},\bv_{h,\ell-1})=b_{h,\ell-1}(\widetilde{\mathcal{R}}_{h,\ell}^\top \bw_{h,\ell},\bv_{h,\ell-1})$. Therefore, 
$$\langle\imph{\Gamma_{\ell-1}^+}{\mathcal{T}_{h,\ell-1,\ell}\bw_{h,\ell}},(\bv_{h,\ell-1})_T\rangle_{\Gamma_{\ell-1}^+}=b_{h,\ell-1}(\widetilde{\mathcal{R}}_{h,\ell}^\top \bw_{h,\ell},\bv_{h,\ell-1}).$$ 
Furthermore, since $\bw_{h,\ell}\in \bm{U}_{0,h,\ell}^s$,  $\bw_{h,\ell}$ solves $a_{\Omega_\ell}(\bw_{h,\ell}, \bv_{h,\ell})=\langle \imph{\Gamma_\ell^s}{\bw_{h,\ell}},(\bv_{h,\ell})_T\rangle_{\Gamma_\ell^s}$. By Theorem \ref{thm:dis-iip-aa}, \eqref{eqn:imp2imp&error} holds.
\end{proof}

\section{Convergence analysis}\label{sec:6}
 
In this section, we study the convergence of the discrete Schwarz error
propagation operator to its continuous-level analogue.
The main
idea is to represent both the continuous and discrete error propagation
operators by corresponding impedance-to-impedance maps on the 
interfaces. This reduces the comparison of the powers
\(\boldsymbol{\mathcal T}_h^n\) and \(\boldsymbol{\mathcal T}^n\) to the
comparison of the interface operators \(\boldsymbol{\mathcal I}^h\) and
\(\boldsymbol{\mathcal I}\).
\subsection{Definition of the impedance-to-impedance map acting on the global interface data}

We first collect the impedance traces on all artificial interfaces into global
interface vectors. This allows us to express one Schwarz error-propagation step
as a single impedance-to-impedance operator acting on the global interface data. 
For $s\in\{+,-\}$, let 
\begin{equation}\label{eq:tensor_L}
    \boldsymbol{\Gamma}^s:=\prod_{\ell=1}^N\Gamma_\ell^s, \quad\mathbb{L}_T(\boldsymbol{\Gamma}^s):=\prod_{\ell=1}^N \bL^2_{T}(\Gamma_\ell^s)
\end{equation}
equipped with the norm 
$$
\|\mathbf{g}^s\|_{0,\boldsymbol{\Gamma}^s}^2:=\sum_{\ell=1}^N\|\mathbf{g}^s_\ell\|^2_{0,\Gamma_\ell^s}, \quad\text{for }\mathbf{g}^s=(\mathbf{g}^s_1,\dotso,\mathbf{g}^s_N)\in\mathbb{L}_T(\boldsymbol{\Gamma}^s).
$$
We define
$\mathbb{L}_T(\boldsymbol{\Gamma}):=\mathbb{L}_T(\boldsymbol{\Gamma}^-)\times\mathbb{L}_T(\boldsymbol{\Gamma}^+)$ equipped with the norm
$$\|\mathbf{g}\|_{0,\boldsymbol{\Gamma}}^2:=\|\mathbf{g}^-\|^2_{0,\boldsymbol{\Gamma}^-} + \|\mathbf{g}^+\|^2_{0,\boldsymbol{\Gamma}^+},\quad\text{for }\mathbf{g}=(\mathbf{g}^-,\mathbf{g}^+)\in\mathbb{L}_T(\boldsymbol{\Gamma}).
$$ 
We define $\text{imp}_{\boldsymbol{\Gamma}^s}:\mathbb{U}_0\rightarrow \mathbb{L}_T(\boldsymbol{\Gamma}^s)$ by
\begin{equation*}
    \imp{\boldsymbol{\Gamma}^s}{\mathfrak v}
:=
\big(
\imp{\Gamma_1^s}{\bv_1},\ldots,
\imp{\Gamma_\ell^s}{\bv_\ell},\ldots,
\imp{\Gamma_N^s}{\bv_N}
\big).
\end{equation*}
for $\mathfrak{v}=(\bv_1,\dotso,\bv_\ell,\dotso,\bv_N)\in \mathbb{U}_0$.
Combining $s=+$ and $s=-$, we define $\text{imp}_{\boldsymbol{\Gamma}}:\mathbb{U}_0\rightarrow \mathbb{L}_T(\boldsymbol{\Gamma})$ by
\begin{equation}\label{e:imp}
\text{imp}_{\boldsymbol{\Gamma}}:=\left(\begin{array}{c}
        \text{imp}_{\boldsymbol{\Gamma}^-}\\
         \text{imp}_{\boldsymbol{\Gamma}^+}    
         \end{array}\right).
\end{equation}

These definitions, along with the definition of $\|\cdot\|_{\mathbb{U}_0}$ in \eqref{eq:normU0}, imply the following.

\begin{lemma}\label{lem:norm_equiv}
    For any $\mathfrak{v}=(\bv_1,\cdots,\bv_N)\in\mathbb{U}_0$ with zero impedance data on $\partial\Omega$, 
    \begin{equation}
    \label{eq:normU0new}
        \|\mathfrak{v}\|_{\mathbb{U}_0}=
        \|\operatorname{imp}\mathfrak{v}\|_{0,\bm{\Gamma}},
    \end{equation}
where $\|\cdot\|_{\mathbb{U}_0}$ is defined by \eqref{eq:normU0}.
    \end{lemma}

For $t\in\{+,-\}$, we define the impedance-to-impedance map $\mathcal{I}_{s\rightarrow t}:\mathbb{L}_T(\boldsymbol{\Gamma}^s)\rightarrow\mathbb{L}_T(\boldsymbol{\Gamma}^t)$ by
\begin{align*}
\mathcal{I}_{s\rightarrow -}&:=\begin{pmatrix}
  0               &        \\
   \mathcal{I}_{\Gamma_1^s\rightarrow\Gamma_2^-}      &       0               &\\
                  &     \mathcal{I}_{\Gamma_2^s\rightarrow\Gamma_3^-}       &0              &\\
                  &                      &   \ddots      &      \ddots  & \\
                  &                       &       &       \mathcal{I}_{\Gamma_{N-2}^s\rightarrow\Gamma_{N-1}^-}         & 0&\\
                  &                       &                       &    &  \mathcal{I}_{\Gamma_{N-1}^s\rightarrow\Gamma_{N}^-}&0
  \end{pmatrix},\\
  \mathcal{I}_{s\rightarrow +}&:=\begin{pmatrix}
  0               &       \mathcal{I}_{\Gamma_2^s\rightarrow\Gamma_1^+} \\
         &       0               &\mathcal{I}_{\Gamma_3^s\rightarrow\Gamma_2^+}\\
                  &            &0              &\mathcal{I}_{\Gamma_4^s\rightarrow\Gamma_3^+}\\
                  &                      &         &      \ddots  & \ddots\\
                  &                       &       &                & 0&\mathcal{I}_{\Gamma_{N}^s\rightarrow\Gamma_{N-1}^+}\\
                  &                       &                       &    &  &0
  \end{pmatrix}.
\end{align*}
We define the impedance-to-impedance map $\boldsymbol{\mathcal{I}}:\mathbb{L}_T(\boldsymbol{\Gamma})\rightarrow \mathbb{L}_T(\boldsymbol{\Gamma})$ by
\begin{equation*}
    \boldsymbol{\mathcal{I}}:=\left(\begin{array}{cc}
         \mathcal{I}_{-\rightarrow -} &\mathcal{I}_{+\rightarrow -}\\
         \mathcal{I}_{-\rightarrow +} &\mathcal{I}_{+\rightarrow +}    \end{array}
    \right).
\end{equation*}
The following result shows that the global error propagation operator
\(\boldsymbol{\mathcal T}\) is completely represented, at the level of
interface impedance traces, by the global impedance-to-impedance map
\(\boldsymbol{\mathcal I}\).
\begin{theorem}\label{thm:impI}
Suppose that $\{\Omega_\ell\}_{\ell=1}^N$ is a strip decomposition with $N\geq 2$.
Then 
\begin{equation}\label{e:impI}
    \imp{\boldsymbol{\Gamma}}{\boldsymbol{\cT}\mathfrak{v}}=\boldsymbol{\mathcal{I}}\imp{\boldsymbol{\Gamma}}{\mathfrak{v}}
    \quad \text{ for all } \mathfrak{v}\in\mathbb{U}_0.
\end{equation}    
\end{theorem}
\begin{proof}
    For $\mathfrak{v}=(\bv_1,\dotso,\bv_N)\in \mathbb{U}_0$, by the definition of ${\rm imp}$ \eqref{e:imp} and the fact that $\boldsymbol{\cT}= \boldsymbol{\mathcal{L}}+ \boldsymbol{\mathcal{U}}$,
\begin{equation}\label{eq1}
\begin{aligned}
    \imp{\boldsymbol{\Gamma}}{\boldsymbol{\cT}\mathfrak{v}}&=\begin{pmatrix}
\imp{\boldsymbol{\Gamma}^-}{\boldsymbol{\cT}\mathfrak{v}}\\
\imp{\boldsymbol{\Gamma}^+}{\boldsymbol{\cT}\mathfrak{v}}
    \end{pmatrix}=\begin{pmatrix}
\imp{\boldsymbol{\Gamma}^-}{\boldsymbol{\mathcal{L}}\mathfrak{v}}+\imp{\boldsymbol{\Gamma}^-}{\boldsymbol{\mathcal{U}}\mathfrak{v}}\\
\imp{\boldsymbol{\Gamma}^+}{\boldsymbol{\mathcal{L}}\mathfrak{v}}+\imp{\boldsymbol{\Gamma}^+}{\boldsymbol{\mathcal{U}}\mathfrak{v}}
    \end{pmatrix}.
\end{aligned}
\end{equation}
Now $\imp{\boldsymbol{\Gamma}^-}{\boldsymbol{\mathcal{U}}\mathfrak{v}}=\bm{0}$ and $\imp{\boldsymbol{\Gamma}^+}{\boldsymbol{\mathcal{L}}\mathfrak{v}}=\bm{0}$ by the definition of PoU, and error propagation operator defined in \eqref{eqn:local_error}.
More precisely, 
by the definitions of $\boldsymbol{\mathcal{L}}$ and 
$\boldsymbol{\mathcal{U}}$,
the $j$-th row of \eqref{eq1} is 
\begin{equation*}
    \imp{\Gamma_j^-}{\cT_{j,j-1}\bv_{j-1}}+\imp{\Gamma_j^-}{\cT_{j,j+1}\bv_{j+1}}
\end{equation*}
where $\imp{\Gamma_j^-}{\cT_{j,j+1}\bv_{j+1}}=\imp{\Gamma_j^-}{\chi_{j+1}\bv_{j+1}}$ by \eqref{eqn:local_error}.
By the support property of the partition of unity, 
\(\chi_{j+1}\) vanishes in a neighbourhood of \(\Gamma_j^-\). Hence
\(\chi_{j+1}\bv_{j+1}\), and therefore its impedance trace, vanish on
\(\Gamma_j^-\). 
Consequently,
$\imp{\Gamma_j^-}{\boldsymbol{\mathcal{U}}\mathfrak{v}}=\bm{0}$. The identity $\imp{\Gamma_j^+}{\boldsymbol{\mathcal{L}}\mathfrak{v}}=\bm{0}$ follows analogously.
It remains to express
$\imp{\boldsymbol{\Gamma}^-}{\boldsymbol{\mathcal{L}}\mathfrak{v}}$ and
$\imp{\boldsymbol{\Gamma}^+}{\boldsymbol{\mathcal{U}}\mathfrak{v}}$ in terms
of the impedance-to-impedance maps.
Similar to the proof in Lemma \ref{assump:TLU}, 
by \eqref{eqn:local_error}, \eqref{eq:chi-strip} and Definition \ref{def:imp2imp},
\begin{align*}
    \imp{\Gamma_j^-}{\cT_{j,j-1}\bv_{j-1}}=\imp{\Gamma_j^-}{\chi_{j-1}\bv_{j-1}}&=\imp{\Gamma_j^-}{\bv_{j-1}}\\&=\cI_{\Gamma_{j-1}^-\to\Gamma_j^-}\imp{\Gamma_{j-1}^-}{\bv_{j-1}}+\cI_{\Gamma_{j-1}^+\to\Gamma_j^-}\imp{\Gamma_{j-1}^+}{\bv_{j-1}}.
\end{align*}
Applying the same argument to all  interfaces gives:
\begin{equation*}
    \begin{aligned}
    \imp{\boldsymbol{\Gamma}}{\boldsymbol{\cT}\mathfrak{v}}=\begin{pmatrix}
\imp{\boldsymbol{\Gamma}^-}{\boldsymbol{\mathcal{L}}\mathfrak{v}}\\
\imp{\boldsymbol{\Gamma}^+}{\boldsymbol{\mathcal{U}}\mathfrak{v}}
    \end{pmatrix}&=\begin{pmatrix}
\mathcal{I}_{-\rightarrow -}\imp{\boldsymbol{\Gamma}^-}{\mathfrak{v}}+\mathcal{I}_{+\rightarrow -}\imp{\boldsymbol{\Gamma}^+}{\mathfrak{v}}\\
\mathcal{I}_{-\rightarrow +}\imp{\boldsymbol{\Gamma}^-}{\mathfrak{v}}+\mathcal{I}_{+\rightarrow +}\imp{\boldsymbol{\Gamma}^+}{\mathfrak{v}}
    \end{pmatrix}\\
    &=\begin{pmatrix}
\mathcal{I}_{-\rightarrow -} &\mathcal{I}_{+\rightarrow -}\\
\mathcal{I}_{-\rightarrow +} &\mathcal{I}_{+\rightarrow +}    \end{pmatrix}\begin{pmatrix}
\imp{\boldsymbol{\Gamma}^-}{\mathfrak{v}}\\
\imp{\boldsymbol{\Gamma}^+}{\mathfrak{v}}
    \end{pmatrix}\\ 
    &=\boldsymbol{\mathcal{I}}\imp{\boldsymbol{\Gamma}}{\mathfrak{v}}.
    \end{aligned}
\end{equation*}
\end{proof}

The result \eqref{e:impI} is useful because, by \eqref{eq:normU0new}, the norm of the local error is
equivalent to the norm of its impedance data on the interior interfaces. 

We now above the discrete analogue of Theorem \ref{thm:impI}.
For \(s\in\{+,-\}\), let
\[
    \mathbb U_{h,T}(\boldsymbol{\Gamma}^s)
    :=
    \prod_{\ell=1}^N \bm U_{h,T}(\Gamma_\ell^s),
    \quad
    \mathbb U_{h,T}(\boldsymbol{\Gamma})
    :=
    \mathbb U_{h,T}(\boldsymbol{\Gamma}^-)
    \times
    \mathbb U_{h,T}(\boldsymbol{\Gamma}^+).
\]
Both spaces are equipped with the product \(L^2\)-norms
\[
    \|\mathbf g_h^s\|_{0,\boldsymbol{\Gamma}^s}^2
    :=
    \sum_{\ell=1}^N
    \|\mathbf g_{h,\ell}^s\|_{0,\Gamma_\ell^s}^2,
    \quad
    \|\mathbf g_h\|_{0,\boldsymbol{\Gamma}}^2
    :=
    \|\mathbf g_h^-\|_{0,\boldsymbol{\Gamma}^-}^2
    +
    \|\mathbf g_h^+\|_{0,\boldsymbol{\Gamma}^+}^2 .
\]
The discrete global
impedance-to-impedance map
\(\boldsymbol{\mathcal I}^h:\mathbb U_{h,T}(\boldsymbol{\Gamma})
\to \mathbb U_{h,T}(\boldsymbol{\Gamma})\) is defined by
\begin{equation*}
    \boldsymbol{\mathcal I}^h
    :=
    \begin{pmatrix}
         \mathcal I^h_{-\to -} & \mathcal I^h_{+\to -}\\
         \mathcal I^h_{-\to +} & \mathcal I^h_{+\to +}
    \end{pmatrix}.
\end{equation*}

Repeating the proof of Theorem \ref{thm:impI}, but now using Lemma \ref{lem:imp_zero} and Theorem \ref{thm:IT}, 
we obtain the following result.

\begin{theorem}\label{thm:impI_h} 
Suppose that $\{\Omega_\ell\}_{\ell=1}^N$ is a strip decomposition with $N\geq 2$. Then
 \begin{equation*}
    \imph{\boldsymbol{\Gamma}}{\boldsymbol{\cT}_h\mathfrak{v}_h}=\boldsymbol{\mathcal{I}}^h\imph{\boldsymbol{\Gamma}}{\mathfrak{v}_h}
    \quad \text{ for all } \mathfrak{v}_h\in \mathbb{U}_{0,h}.
\end{equation*}    
\end{theorem}

\subsection{Convergence of the discrete error propagation}
First, we show that norms of powers of the 
error propagation matrix 
equal norms of powers of the 
global impedance-to-impedance map. This
identity allows us to reduce the analysis of the error propagation to the analysis of the global impedance-to-impedance map.
\begin{theorem}\label{thm:norm_equiv_hc}
Suppose that $\{\Omega_\ell\}_{\ell=1}^N$ is a strip decomposition with $N\geq 2$. Then, for $n\geq 1$,
$$\|\bm{\cT}^n\|_{\mathbb{U}_0}=\|\bm{\cI}^n\|_{0,\bm{\Gamma}},\quad
         \|\bm{\cT}_h^n\|_{\mathbb{U}_{0,h}}=\|(\bm{\cI}^h)^n\|_{0,\bm{\Gamma}}.
   $$
\end{theorem}
\begin{proof}
We prove the first identity;
   the proof of the second is similar, using Definitions~\ref{def:discrete_norm} and \ref{def:discrete_iip}, along with Theorem~\ref{thm:impI_h}.
Let
\begin{equation*}
    \widetilde{\mathbb{U}}_0:=\{\mathfrak{v}\in \mathbb{U}_0\,:\, \iimp\mathfrak{v}=\bm{0} \text{ on } \partial\Omega\}.
\end{equation*}
Since Lemma~\ref{lem:norm_equiv} only applies to elements of $\widetilde{\mathbb{U}}_0$, we first show that the operator norm of $\bm{\cT}^n$ over $\mathbb{U}_0$ can be restricted to $\widetilde{\mathbb{U}}_0$, i.e.,
\begin{equation}\label{eq:restrict_sup}
    \|\bm{\cT}^n\|_{\mathbb{U}_0}
    = \sup_{\mathfrak{v}\in \mathbb{U}_0}\frac{\|\bm{\cT}^n\mathfrak{v}\|_{\mathbb{U}_0}}{\|\mathfrak{v}\|_{\mathbb{U}_0}}
    = \sup_{\mathfrak{v}\in \widetilde{\mathbb{U}}_0}\frac{\|\bm{\cT}^n\mathfrak{v}\|_{\mathbb{U}_0}}{\|\mathfrak{v}\|_{\mathbb{U}_0}}.
\end{equation}
The inequality $\geq$ in~\eqref{eq:restrict_sup} is immediate from $\widetilde{\mathbb{U}}_0\subset \mathbb{U}_0$. For the reverse inequality, given any $\mathfrak{v}\in\mathbb{U}_0$, by Lemma~\ref{lem:well pose} there exists a unique $\mathfrak{v}_0\in \widetilde{\mathbb{U}}_0$ such that $\iimp\mathfrak{v}_0=\iimp\mathfrak{v}$ on $\bm{\Gamma}$, and by~\eqref{eq:normU0},
\begin{equation*}
    \|\mathfrak{v}_0\|_{\mathbb{U}_0}\leq \|\mathfrak{v}\|_{\mathbb{U}_0}.
\end{equation*}
By \eqref{eqn:local_error}, $\bm{\cT}$ depends only on the interface impedance data $\iimp_{\bm{\Gamma}}$ and $\iimp_{\bm{\Gamma}}\mathfrak{v}=\iimp_{\bm{\Gamma}}\mathfrak{v}_0$; we therefore have $\bm{\cT}\mathfrak{v}=\bm{\cT}\mathfrak{v}_0$, and inductively $\bm{\cT}^n\mathfrak{v}=\bm{\cT}^n\mathfrak{v}_0$. Therefore
\begin{equation*}
    \frac{\|\bm{\cT}^n\mathfrak{v}\|_{\mathbb{U}_0}}{\|\mathfrak{v}\|_{\mathbb{U}_0}}
    =\frac{\|\bm{\cT}^n\mathfrak{v}_0\|_{\mathbb{U}_0}}{\|\mathfrak{v}\|_{\mathbb{U}_0}}
    \leq \frac{\|\bm{\cT}^n\mathfrak{v}_0\|_{\mathbb{U}_0}}{\|\mathfrak{v}_0\|_{\mathbb{U}_0}},
\end{equation*}
which gives the inequality $\leq$ in~\eqref{eq:restrict_sup}.
    Applying the norm equivalence from Lemma~\ref{lem:norm_equiv} to both the numerator and denominator, 
    and then using Theorem~\ref{thm:impI},
    we have
    \begin{equation*}
        \|\bm{\cT}^n\|_{\mathbb{U}_0} = \sup_{\mathfrak{v}\in \widetilde{\mathbb{U}}_0} \frac{\|\iimp(\bm{\cT}^n\mathfrak{v})\|_{0,\bm{\Gamma}}}{\|\iimp \mathfrak{v}\|_{0,\bm{\Gamma}}} = \sup_{\mathfrak{v}\in \widetilde{\mathbb{U}}_0} \frac{\|\bm{\cI}^n (\iimp \mathfrak{v})\|_{0,\bm{\Gamma}}}{\|\iimp \mathfrak{v}\|_{0,\bm{\Gamma}}}.
    \end{equation*}
 Finally, since the impedance map $\iimp_{\bm{\Gamma}}: \widetilde{\mathbb{U}}_0 \to \mathbb{L}_T(\bm{\Gamma})$
 (where $\mathbb{L}_T(\bm{\Gamma})$ is defined by \eqref{eq:tensor_L})
 is surjective by Lemma~\ref{lem:well pose}, 
    \begin{equation*}
        \|\bm{\cT}^n\|_{\mathbb{U}_0} = \sup_{\mathbf{g} \in \mathbb{L}_T(\bm{\Gamma})} \frac{\|\bm{\cI}^n \mathbf{g}\|_{0,\bm{\Gamma}}}{\|\mathbf{g}\|_{0,\bm{\Gamma}}} = \|\bm{\cI}^n\|_{0,\bm{\Gamma}}.
    \end{equation*}
\end{proof}

\begin{assumption}[Convergence of the discrete impedance-to-impedance map]
\label{ass:Ih-I}
As \(h\to 0\), the discrete interface map converges to the continuous one in
operator norm:~i.e.,
\begin{equation*}
        \|\boldsymbol{\mathcal I}^h-\boldsymbol{\mathcal I}\|_{0,\boldsymbol{\Gamma}}
        \to 0 \quad \text{ as } h\to 0.
\end{equation*}
\end{assumption}

\begin{remark}[Discussion of Assumption \ref{ass:Ih-I}]
\label{rem:localFEM}
In the Helmholtz case, \cite[Theorem 5.5]{gong2023convergence} proves Assumption \ref{ass:Ih-I} using a local analysis
(away from the impedance boundary)
of the FEM error for the Helmholtz interior impedance problem in a rectangular cuboid (corresponding to one subdomain in a strip decomposition). Crucial ingredients in this analysis are (i) $H^2$ regularity of the solution of the Helmholtz equation in a convex polyhedron (used in \cite[Lemma 5.8]{gong2023convergence}), and (ii) an adaption of the ``elliptic-projection" duality argument to work in a weighted norm \cite[Theorem 5.7]{gong2023convergence}.

The obstacles to establishing an analogous result in the Maxwell setting here are (i) the reduced regularity -- the Maxwell solution is only $\bH^1$ in a convex polyhedron, with this only currently proved for homogeneous impedance data \cite{NiTo:19} (see the discussion around Assumption \ref{ass:regularity}) and (ii) the fact that the Maxwell analogues of the ``elliptic-projection" duality arguments are both more involved, and the subject of recent research \cite{LuWu:25, CGS1, Lu2026preasymptotic}.
\end{remark}

\begin{theorem}
Suppose that 
$\{\Omega_\ell\}_{\ell=1}^N$ is a strip decomposition with $N\geq 2$ and Assumption \ref{ass:Ih-I} holds.
Then,
    for any fixed integer $n\geq 1$, as $h\to 0$,
    \begin{equation*}
\|\bm{\cT}_h^n\|_{\mathbb{U}_{0,h}}
    \to
    \|\bm{\cT}^n\|_{\mathbb{U}_0}
    \end{equation*}
\end{theorem}
\begin{proof}
    By Theorem \ref{thm:norm_equiv_hc}, 
    \begin{equation*}
        \begin{aligned}
\left|\|\bm{\cT}_h^n\|_{\mathbb{U}_{0,h}}- \|\bm{\cT}^n\|_{\mathbb{U}_0}\right|=\left|\|(\bm{\cI}^h)^n\|_{0,\bm{\Gamma}}- \|\bm{\cI}^n\|_{0,\bm{\Gamma}}\right|\leq \left\|(\bm{\cI}^h)^n-\bm{\cI}^n\right\|_{0,\bm{\Gamma}}.
        \end{aligned}
    \end{equation*}
    Arguing by induction (exactly as in the proof of \cite[Theorem 4.17]{gong2023convergence}), one
can show that 
    \begin{equation*}
        \left\|(\bm{\cI}^h)^n-\bm{\cI}^n\right\|_{0,\bm{\Gamma}}\leq nC^{n-1}\left\|\bm{\cI}^h-\bm{\cI}\right\|_{0,\bm{\Gamma}},
    \end{equation*}
    where the constant $C$ depends on the constant in 
    \eqref{eq:imp2imp_bound}
    (i.e., the boundedness of impedance-to-impedance maps).
        The result then follows from Assumption \ref{ass:Ih-I}.
\end{proof}

\begin{remark}[Comparison with the Helmholtz results of \cite{gong2022convergence, gong2023convergence, lafontaine2023sharp}]
\label{rem:compare}
The Maxwell results in this paper are not quite as extensive as the Helmholtz results in \cite{gong2022convergence, gong2023convergence, lafontaine2023sharp} in the following three ways:

(i) The wellposedness result of Theorem \ref{thm:wellpose2} requires the regularity assumption Assumption \ref{ass:regularity}, whereas the Helmholtz analogue of Theorem \ref{thm:wellpose2} (\cite[Theorem 2.12]{gong2022convergence}) is proved for general Lipschitz subdomains. 

(ii) We have assumed that the discrete impedance-to-impedance maps converge to their continuous counterparts (Assumption \ref{ass:Ih-I}) while the Helmholtz analgoue of this result is proved in \cite[Theorem 5.5]{gong2023convergence}.

(iii) We have not attempted to bound the impedance-to-impedance maps in the high-frequency limit at the continuous level as done for the Helmholtz case in \cite{lafontaine2023sharp}.

Regarding (i) and (ii):~the obstacles to establishing these results are because of the well-known fact that the Maxwell solution is less regular that the analogous Laplace/Helmholtz solution (see the discussion 
below Conjecture \ref{con:regularity} and 
in Remark \ref{rem:localFEM}). 

Regarding (iii):~the bounds in \cite{lafontaine2023sharp} use so-called \emph{semiclassical defect measures} (see, e.g., \cite[Chapter 5]{zworski2012semiclassical}, \cite{miller2000refraction}), and, to our knowledge, the theory of these measures has not yet been established for the time-harmonic Maxwell equations. 
    \end{remark}
    



\section{Numerical tests}\label{sec:numerical}
    The numerical experiments in this section are organized around three stages: first
we observe the convergence of the RAS-imp preconditioner, then we identify its convergence mechanism (via the impedance-to-impedance maps), and finally we report iteration counts of RAS-imp across a broader range of 
parameter regimes and geometries.
\\
\indent
More precisely:~we begin in
\S\ref{sec:num2} by plotting the residual history of the RAS-imp method applied
to a representative problem with a nonzero source term on strip decompositions,
which illustrates the basic convergence behavior of the iteration. We then turn
from residuals to the error propagation operator itself:~in \S\ref{sec:num3}, we
show that suitable powers of this operator are contractive on strip
decompositions, and in \S\ref{sec:num4} we compute the norms of the underlying
impedance-to-impedance maps, which—together with Theorem~\ref{thm:TLU}—account
for the power contractivity observed in \S\ref{sec:num3}. Moving beyond the
strip decompositions covered by our theoretical analysis, \S\ref{sec:num5}
considers checkerboard decompositions and finds the same power-contractivity
behavior. The section concludes with a broader study in \S\ref{sec:num6}, where
we report iteration counts for RAS-imp and the corresponding
right-preconditioned GMRES method under varying decompositions, overlaps,
wavenumbers, and PDE coefficients.

\subsection{Set up}\label{sec:num1}

We consider the following two-dimensional time-harmonic Maxwell's equations on a domain $\Omega$, coupled with impedance boundary conditions:
\begin{equation}\label{eq:2D_maxwell}
\begin{aligned}
\overrightarrow{\curl}(\mu^{-1}\rot\bu)-k^2\epsilon\bu&=\bm{f}\quad\mbox{in }\Omega,\\
\bm{\tau}(\mu^{-1}\rot\bu-ik\lambda\bu\cdot\bm{\tau})&=
\bg_T\quad \mbox{on }\partial\Omega,
\end{aligned}
\end{equation}
where $\overrightarrow{\curl}=\left(\dfrac{\partial}{\partial y},-\dfrac{\partial}{\partial x}\right)$, $\rot\bu=\dfrac{\partial\bu_y}{\partial x}-\dfrac{\partial\bu_x}{\partial y}$, and $\bm{\tau}=(-\bn_y,\bn_x)$ denotes the unit tangent vector.
The corresponding sesquilinear form is
\begin{equation}
a_\Omega(\bu,\bv):=\int_{\Omega}\left(\mu^{-1}\rot\bu\rot\bar{\bv}-k^2\epsilon\bu\cdot\bar{\bv}\right)\mathrm{d}\bx - i k\int_{\partial\Omega}\lambda(\bu\cdot\bm{\tau})(\bar{\bv}\cdot\bm{\tau})\mathrm{d}s,
\end{equation}
and the discrete variational problem reads: 
\begin{equation}
a_\Omega(\bu_h,\bv_h)=\int_\Omega \bm{f}\cdot\bar{\bv}_h\,\mathrm{d}\bx + \int_{\partial\Omega}\bg_T\cdot\bar{\bv}_{h}\,\mathrm{d}s.
\end{equation}
Both homogeneous media ($\epsilon=1$, $\mu=1$, $\lambda=1$) and heterogeneous media with spatially varying permittivity are considered.

All experiments are carried out on two-dimensional rectangular domains, discretized with the lowest-order Nédélec finite elements on a mesh of size $h \propto k^{-3/2}$ chosen to suppress the pollution effect \cite{Lu2026preasymptotic}.
The domain $\Omega = [0,L_x]\times[0,L_y]$ is partitioned into $N$ non-overlapping rectangular subdomains, each then enlarged by appending all neighboring elements within a distance $\delta$. The overlap $\delta$ is either prescribed as a fixed value or set to $\delta = \eta H$, where $H = L_x/N$ denotes the subdomain size and $\eta$ the overlap ratio.


\subsection{Residual histories of RAS-imp on strip decompositions}\label{sec:num2}

We consider the Maxwell equations on $\Omega=(0,\tfrac{2}{3}N)\times(0,1)$ 
with homogeneous impedance boundary data. 
We consider the oscillatory source term 
defined by 
\[
    x_0=\left(\frac{N}{3},\frac12\right), 
    \qquad 
    r=|x-x_0|,
\]
and then
\[
    \bm{f}=(0,J_0(kr)),
\]
where \(J_0\) is the Bessel function of the first kind of order zero. 
We apply the preconditioned Richardson iteration~\eqref{eq:Richardson} with $B_h^{-1}$ the RAS-imp preconditioner, terminated when $\|\br\|_2/\|\bm{b}\|_2 \le 10^{-6}$, with a maximum of $500$ iterations.

Figure~\ref{fig:k_N} shows the relative residual histories for varying wavenumbers $k$ and numbers of subdomains $N$ with fixed overlap size $\delta = H/4$. As $k$ varies from $10$ to $50$, the number of iterations needed for convergence remains essentially the same, indicating robustness with respect to $k$. Figure~\ref{fig:delta_N} shows the corresponding results for fixed $k = 30$ with varying $\delta$; the overlap size has only a minor effect. In contrast, the iteration count grows roughly linearly with $N$ in both figures. A common feature across all experiments is a distinctive stepwise decay of the residuals, which is closely related to the power contractivity of the error propagation operator analyzed in the next subsection.

\begin{figure}[h!]
    \centering
    \includegraphics[width=1\textwidth]{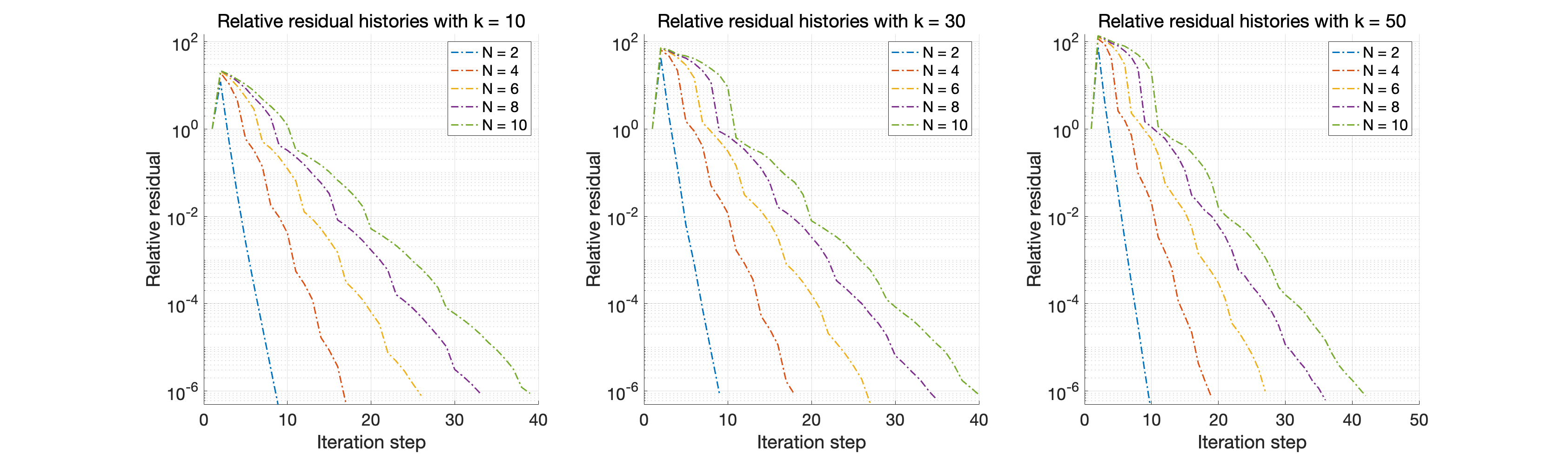}
   \caption{Relative residual histories of the preconditioned Richardson iteration for a strip domain decomposition with $\delta = H/4$, varying $N$ and $k$ on the domain $(0,\frac23N)\times (0,1)$.}
    \label{fig:k_N}
\end{figure}

\begin{figure}[h!]
    \centering
    \includegraphics[width=0.7\textwidth]{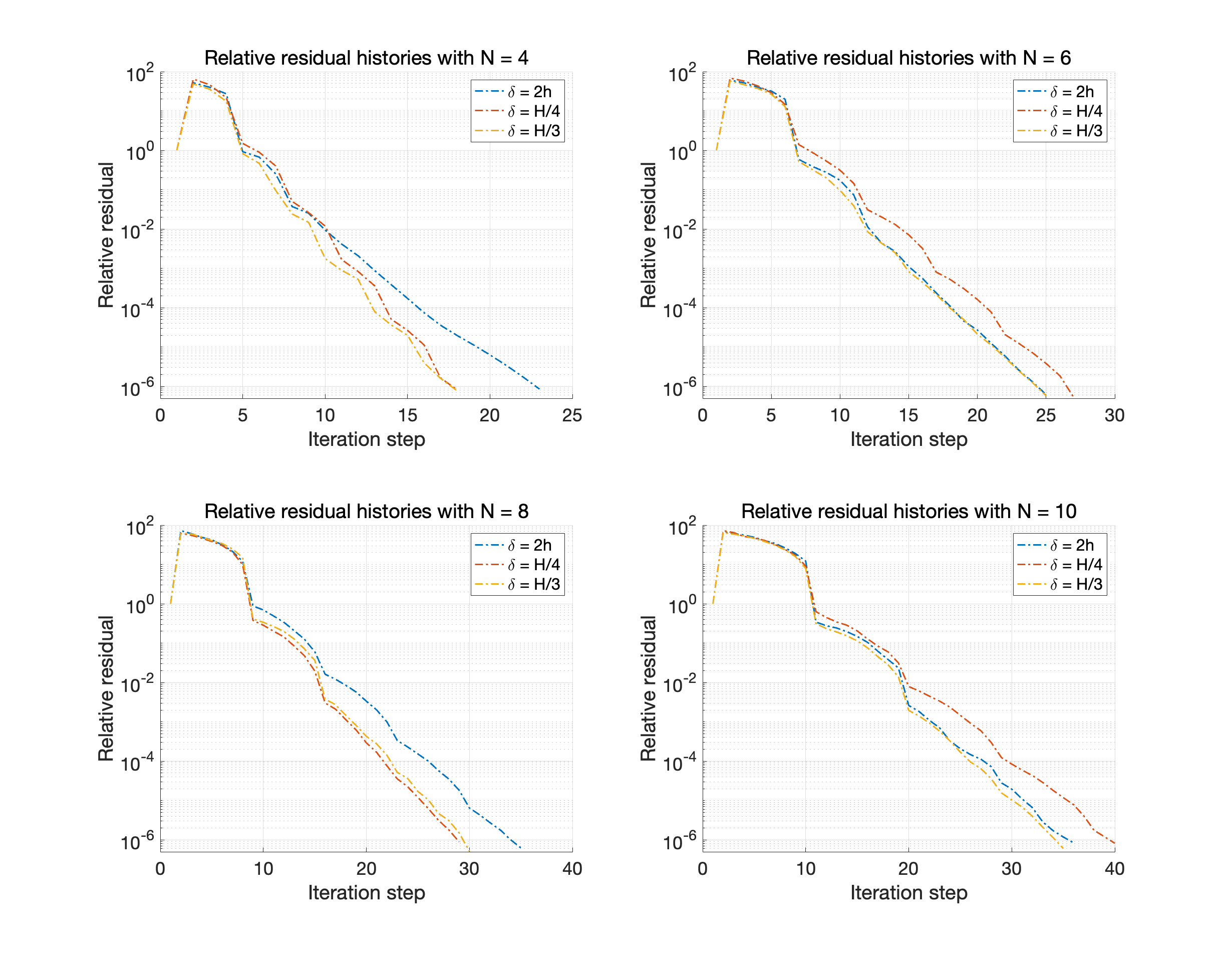}
   \caption{Relative residual histories of the preconditioned Richardson iteration for a strip domain decomposition with $k=30$, varying $N$ and $\delta$ on the domain $(0,\frac23N)\times (0,1)$.}
    \label{fig:delta_N}
\end{figure}

\subsection{Power contractivity for strip decompositions}\label{sec:num3}
In this subsection, we investigate the power contractivity of RAS-imp
 on a strip domain decomposition.
For the error equation \eqref{eqn:discrete-error-prop} arising in the parallel Schwarz method, convergence is typically established by verifying that the spectral radius $\rho(\bm{\mathcal{T}}_h)$ of the iteration matrix is strictly less than $1$. Direct computation of $\rho(\bm{\mathcal{T}}_h)$, however, is prohibitively expensive. We therefore reduce this computation to a spectral problem for a much smaller matrix $\bE$.
Let
\begin{equation}\label{eq:boldRs}
\widetilde{\mathbf{R}}_h^\mathrm{T}:=[\widetilde{\mathcal{R}}^\mathrm{T}_{h,1}, \dotso, \widetilde{\mathcal{R}}^\mathrm{T}_{h,N}], \qquad
\mathcal{\mathbf{R}}_h:=[\mathcal{R}_{h,1},\dotso, \mathcal{R}_{h,N}]^\mathrm{T}
\end{equation}
and 
 $$\mathbf{Q}_h:=[A_{h,1}^{-1}\mathcal{R}_{h,1}A_h,\dotso, A_{h,N}^{-1}\mathcal{R}_{h,N}A_h]^\mathrm{T}.$$
The definition of the RAS-imp Richardson iteration \eqref{eq:Richardson}
implies that the global error $\be_h^n:=\bu_h-\bu_h^n$ satisfies
\begin{equation*}
    \be^{n}_h=\bE\be^{n-1}_h,~~~ \bE:=I-B_h^{-1}A_h.
\end{equation*}
where 
$B_h^{-1}
=
\sum_{\ell=1}^N
\widetilde{\mathcal R}_{h,\ell}^{\mathrm T}
A_{h,\ell}^{-1}
\mathcal R_{h,\ell}.$
By the definition of $B_h^{-1}$ and the fact that $\sum_\ell\widetilde{\mathcal{R}}_{h,\ell}^T\mathcal{R}_{h,\ell}=I$
\eqref{e:discretePOU},
\begin{equation}\label{eq:E}
\bE=\sum_\ell \widetilde{\mathcal{R}}_{h,\ell}^T(\mathcal{R}_{h,\ell}-A_{h,\ell}^{-1}\mathcal{R}_{h,\ell} A_h)=\widetilde{\mathbf{R}}_h^\mathrm{T}(\mathbf{R}_h-\mathcal{\mathbf{Q}}_h).
\end{equation}
In addition, by Definition \ref{def:error_prop} of $\bm{\mathcal{T}}_h$ and the definition of $b_{h,\ell}$ in \eqref{eqn:def-b}, 
\begin{equation*}
    \begin{aligned}
\cT_{h,\ell,j}\be^n_{h,j}=&\left(\widetilde{\mathcal{R}}_{h,\ell}^T\be^n_{h,j}\right)\Big|_{\Omega_\ell}-A_{h,\ell}^{-1}\mathcal{R}_{h,\ell} A_h\widetilde{\mathcal{R}}_{h,j}^T\be^n_{h,j}
    \end{aligned}
\end{equation*}
Thus,
\begin{equation*}
    \be^{n+1}_{h,\ell}=\sum_{j=1}^N\cT_{h,\ell,j}\be^n_{h,j}=\left(\sum_{j=1}^N\widetilde{\mathcal{R}}_{h,\ell}^T\be^n_{h,j}\right)\Big|_{\Omega_\ell}-A_{h,\ell}^{-1}\mathcal{R}_{h,\ell} A_h\sum_{j=1}^N\widetilde{\mathcal{R}}_{h,j}^T\be^n_{h,j}.
\end{equation*}
Then by the definitions of 
$\mathbf{R}_h$ and $\widetilde{\mathbf{R}}_h$
\eqref{eq:boldRs},
the error propagation matrix can be written as
 \begin{equation}\label{eq:T}
	 \boldsymbol{\cT}_h = (\mathbf{R}_h -\mathcal{\mathbf{Q}}_h) \widetilde{\mathbf{R}}_h^\mathrm{T}.
 \end{equation}

By \eqref{eq:E} and \eqref{eq:T}, $\bE$ and $\bm{\mathcal{T}}_h$ are products of the same two factors in opposite orders, and therefore share the same non-zero eigenvalues; in particular, $\rho(\bm{\mathcal{T}}_h)=\rho(\bE)$. Since $\bE$ is the smaller of the two, we work with $\bE$ in what follows.
To quantify the contraction in a $k$-explicit way, we equip $\mathbb{C}^M$, where $M$ is the dimension of the finite element space, 
with the norm $\|\cdot\|_{D_k}$ induced from the norm \begin{equation}
    \label{eq:Ximp2}
  \|\bu\|_{\iimp, k, \Omega}^2 :=\|\curl\bu\|_{0,\Omega}^2 + k^2\|\bu\|_{0,\Omega}^2 + k\|\bu_T\|_{0,\partial\Omega}^2,
\end{equation}
on $X_{\rm imp}(\Omega)$. Note that the well-posedness results in \S\ref{sec:wp} were not explicit in $k$, so the norm \eqref{eq:Ximp} was left unweighted; here we are interested in the $k$-dependence of the FEM error, and \eqref{eq:Ximp2} weights the terms with $k$ in the natural way. We then compute
\begin{equation}
\|\bE^s\| := \max_{0\neq \bv\in \mathbb{C}^M}
\frac{\|\bE^s\bv\|_{D_k}}{\|\bv\|_{D_k}}, \quad s \geq 1.
\end{equation}

Table~\ref{tab:1k_norm_E} reports $\|\bE^s\|$ for varying $k$ and $N$ in the homogeneous case. In every configuration, $\|\bE^s\|$ decreases gradually with $s$ and then drops  to a value well below $1$ at $s\sim N$. 
This matches the stepwise decay observed in Figures~\ref{fig:k_N} and~\ref{fig:delta_N}: the abrupt drops in the residual curves coincide with the iterations $s\sim N$ at which $\|\bE^s\|$ contracts. 
In particular, $\rho(\bE)\le\|\bE^N\|^{1/N}<1$, so the iteration converges asymptotically.


\begin{table}[h!]
\small
    \centering
    \caption{Norms of powers of $\bm{E}$ with $\delta = H/4$ and $h = k^{-3/2}$ on $\left(0,\frac{2}{3}N\right) \times (0,1)$.}
    \label{tab:1k_norm_E}
\begin{tabular}{ccccccccc}
    \hline\hline
    \multirow{2}{*}{$k\backslash N$} & \multicolumn{4}{c}{4} & \multicolumn{4}{c}{8}                 \\\cmidrule(lr){2-5} \cmidrule(lr){6-9} 
         & $\|\bE\|$ & $\|\bE^{N-1}\|$ & $\|\bE^{N}\|$  & $\|\bE^{N+1}\|$ & $\|\bE\|$ & $\|\bE^{N-1}\|$ & $\|\bE^{N}\|$ & $\|\bE^{N+1}\|$ \\\hline
 10 & 6.52e+00 & 3.80e+00 & 2.27e-01 & 1.46e-01 & 6.81e+00 & 2.19e+00 & 1.37e-01 & 1.12e-01\\
 15 & 7.37e+00 & 5.91e+00 & 2.25e-01 & 1.11e-01 & 7.57e+00 & 4.63e+00 & 1.26e-01 & 1.05e-01\\
 20 & 9.88e+00 & 7.69e+00 & 3.74e-01 & 1.86e-01 & 1.03e+01 & 6.60e+00 & 1.71e-01 & 1.10e-01\\
 25 & 1.04e+01 & 9.24e+00 & 5.81e-01 & 2.82e-01 & 1.06e+01 & 8.32e+00 & 2.33e-01 & 1.57e-01\\
 30 & 1.26e+01 & 1.06e+01 & 5.27e-01 & 3.20e-01 & 1.30e+01 & 9.76e+00 & 3.33e-01 & 2.52e-01\\
\hline\hline
\end{tabular}
\end{table}

We next extend the study to heterogeneous media, replacing the constant permittivity by
\begin{equation}\label{eq:epsilon}
\epsilon =
\begin{cases}
1 + \dfrac{1}{2}\exp\!\left(-\dfrac{(x-\frac{L_x}{2})^2 + (y-\frac{L_y}{2})^2}{20}\right),
& \text{if } (x-\tfrac{L_x}{2})^2 + (y-\tfrac{L_y}{2})^2 < \tfrac{1}{4}, \\
1,
& \text{otherwise.}
\end{cases}
\end{equation}
Table~\ref{tab:i_1k_norm_E} reports $\|\bE^s\|$ in this variable-coefficient case. The contraction pattern matches the homogeneous one but is shifted slightly later: $\|\bE^N\|$ now exceeds $1$ at larger $k$, yet $\|\bE^{N+1}\|$ remains well below $1$ in every configuration, so power contractivity continues to hold once $s$ is sufficiently large.
\begin{table}[h!]
\small
    \centering
   \caption{Norms of powers of $\bE$ for the variable-coefficient case \eqref{eq:epsilon}. Results are shown for $\delta = H/4$ with mesh size $h = k^{-3/2}$ on $(0,\frac{2}{3}N) \times (0,1)$.}
\label{tab:i_1k_norm_E}
\begin{tabular}{ccccccccc}
    \hline\hline
    \multirow{2}{*}{$k\backslash N$} & \multicolumn{4}{c}{4} & \multicolumn{4}{c}{8}                 \\\cmidrule(lr){2-5} \cmidrule(lr){6-9} 
         & $\|\bE\|$& $\|\bE^{N-1}\|$ & $\|\bE^{N}\|$ & $\|\bE^{N+1}\|$ & $\|\bE\|$& $\|\bE^{N-1}\|$ & $\|\bE^{N}\|$ & $\|\bE^{N+1}\|$  \\\hline
 10 & 6.26e+00 & 4.26e+00 & 5.39e-01 & 1.99e-01 & 6.66e+00 & 2.45e+00 & 3.19e-01 & 1.18e-01\\
 15 & 7.95e+00 & 6.27e+00 & 8.76e-01 & 1.85e-01 & 7.96e+00 & 4.92e+00 & 7.41e-01 & 1.33e-01\\
 20 & 9.81e+00 & 8.35e+00 & 1.56e+00 & 4.20e-01 & 1.02e+01 & 7.26e+00 & 1.13e+00 & 2.46e-01\\
 25 & 1.12e+01 & 9.35e+00 & 1.30e+00 & 5.32e-01 & 1.13e+01 & 8.89e+00 & 9.13e-01 & 3.08e-01\\
 30 & 1.28e+01 & 1.08e+01 & 2.06e+00 & 7.88e-01 & 1.30e+01 & 1.04e+01 & 1.43e+00 & 4.04e-01\\
\hline\hline
\end{tabular}
\end{table}

\subsection{Boundedness of impedance-to-impedance maps ($\gamma$ and $\rho$)}\label{sec:num4}

\begin{figure}[h!]
\centering
\includegraphics[width=0.4\textwidth]{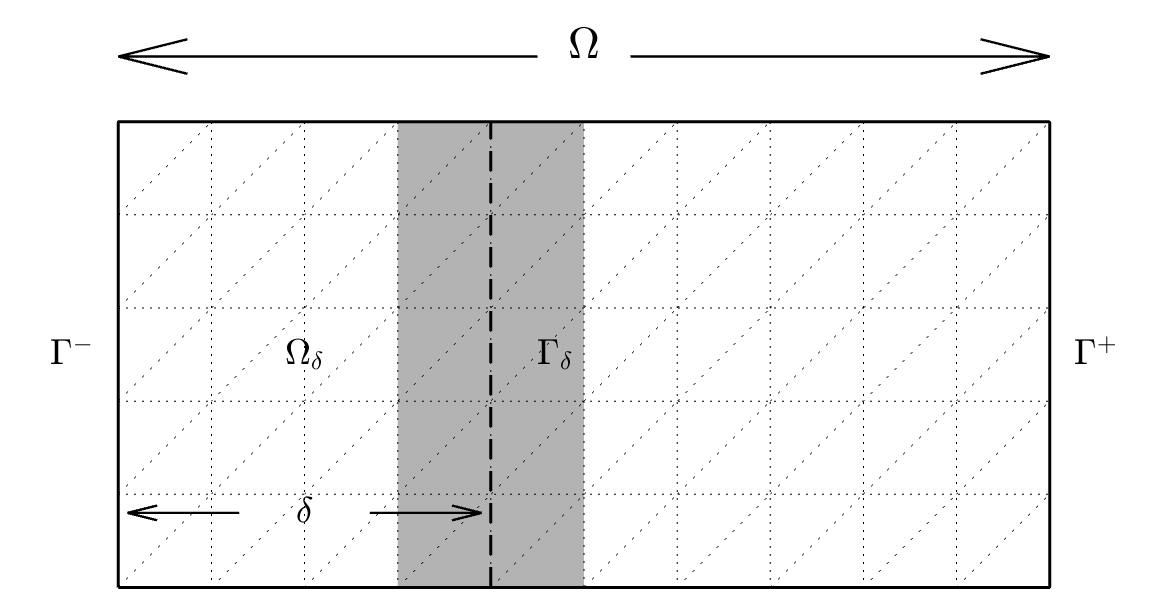}
\caption{Canonical domain for the two-dimensional Maxwell equation.}\label{fig:2D}
\end{figure}

In this subsection, we consider the Maxwell equation on the canonical domain $[0,L]\times [0,1]$ depicted in Figure~\ref{fig:2D}, with a nonzero boundary condition imposed only on $\Gamma^-$. The discrete impedance data $\cI^h\bg_T$ is then defined as in Definition \ref{def:discrete_iip}.

To explain the power contractivity observed in \S\ref{sec:num3}, we examine the boundedness of the impedance-to-impedance maps. Specifically, we compute the operator norms of the left-to-right and left-to-left maps $\cI_{-\to+}$ and $\cI_{-\to-}$ on the canonical domain of Figure \ref{fig:2D}:
\begin{equation}\label{e:rhogamma7}
    \rho(k,\delta,L) = \sup_{\bg_T\in \bL^2_T(\Gamma^-)}\frac{\|\cI_{-\to+}\bg_T\|_{\bL^2_T(\Gamma_\delta)}}{\|\bg_T\|_{\bL^2_T(\Gamma^-)}},\quad \gamma(k,\delta,L) = \sup_{\bg_T\in \bL^2_T(\Gamma^-)}\frac{\|\cI_{-\to-}\bg_T\|_{\bL^2_T(\Gamma_\delta)}}{\|\bg_T\|_{\bL^2_T(\Gamma^-)}}.
\end{equation}
We have deliberately used the notation $\rho$ and $\gamma$ here 
-- i.e., the same notation as in \eqref{eqn:rho} and \eqref{eqn:gamma} -- because, for appropriate values of  $\delta$ and $L$ (depending on $\Omega_\ell$), the maxima over $\ell$ of $\rho$ and $\gamma$ defined in \eqref{e:rhogamma7} bound the quantities in \eqref{eqn:rho} and \eqref{eqn:gamma}, respectively (see \cite[Equations 4.34 and 4.35]{gong2022convergence}).

We first verify in Table~\ref{table:rg_refined} that the discrete approximations of $\rho$ and $\gamma$ stabilize under mesh refinement, confirming that the computed values reliably represent the underlying continuous operator norms. Table~\ref{table:rg_large_k} then probes how $\rho$ and $\gamma$ depend on the subdomain length $L$ and the wavenumber $k$: $\rho$ stays well below one and $\gamma$ remains close to one for all tested $k$, and both decrease as $L$ grows. Combined with Theorem~\ref{thm:TLU} and Corollary \ref{cor:key}, this boundedness yields contractivity of suitable powers of the Schwarz error propagation operator, in turn explaining the stepwise decay of the residual histories.

\begin{table}[h]
    \centering
    \caption{Values of impedance to impedance operator $\rho$ and $\gamma$ for $L=1$ with refined mesh $h = 20^{-3/2}/2^c$.}\label{table:rg_refined}
\begin{tabular}{ccccccc}
    \hline\hline
                & \multicolumn{3}{c}{$\rho(k,\frac{L}{4},L)$} & \multicolumn{3}{c}{$\gamma(k,\frac{3L}{4},L)$} \\ \cmidrule(lr){2-4}\cmidrule(lr){5-7}
$k\backslash c$ & 0     & 1     & 2    & 0     & 1     & 2    \\\hline
5               &0.226361       &0.226214       &0.223826      &0.930556       &0.930599       &0.930327      \\
10              &0.217922       &0.217672       &0.216383      &0.996338       &0.996357       &0.996379      \\
20              &0.239081       &0.238319       &0.237041      &1.000817       &1.000866       &1.000871     \\ \hline\hline
\end{tabular}
\end{table}

\begin{table}[h!]
\centering
\caption{Values of the impedance-to-impedance operators $\rho$ and $\gamma$ for various wavenumbers $k$ and domain lengths $L$, with mesh size $h = 50^{-3/2}$.}
\label{table:rg_large_k}
\begin{tabular}{ccccc}
\hline\hline
 & \multicolumn{2}{c}{$\rho(k,\tfrac{L}{3},L)$} & \multicolumn{2}{c}{$\rho(k,\tfrac{L}{4},L)$} \\
 \cmidrule(lr){2-3}\cmidrule(lr){4-5}
$k\backslash L$ & 1 & 2 & 1 & 2 \\ \hline
20 & 0.192178 & 0.100913 & 0.237055 & 0.133153 \\
30 & 0.217799 & 0.107541 & 0.257803 & 0.150551 \\
40 & 0.237843 & 0.118756 & 0.279766 & 0.165568 \\
50 & 0.252230 & 0.132442 & 0.298163 & 0.178242 \\
\hline\hline
 & \multicolumn{2}{c}{$\gamma(k,\tfrac{2L}{3},L)$} & \multicolumn{2}{c}{$\gamma(k,\tfrac{3L}{4},L)$} \\
 \cmidrule(lr){2-3}\cmidrule(lr){4-5}
$k\backslash L$ & 1 & 2 & 1 & 2 \\ \hline
20 & 1.000117 & 0.999459 & 1.000861 & 0.999967 \\
30 & 1.001804 & 0.999961 & 1.000396 & 1.000144 \\
40 & 1.002934 & 1.000049 & 1.002630 & 1.000054 \\
50 & 1.001302 & 1.000084 & 1.000783 & 1.000244 \\
\hline\hline
\end{tabular}
\end{table}

\subsection{Power contractivity for checkerboard decompositions}
\label{sec:num5}

\begin{figure}[h!]
    \centering
    \includegraphics[width=0.3\textwidth]{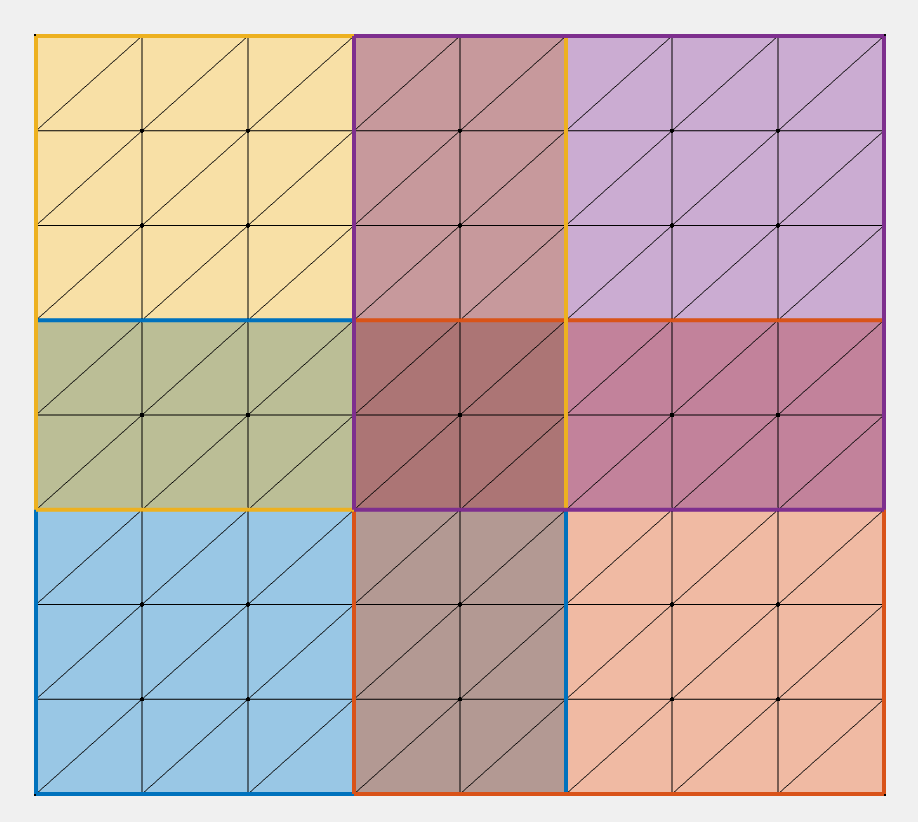} 
    \caption{Checkerboard domain decomposition into overlapping subdomains.}
    \label{fig:checkerboard}
\end{figure}
In this subsection, we consider checkerboard decompositions of rectangular domains, illustrated in Figure~\ref{fig:checkerboard}: an $m\times m$ partition with $N=m^2$ subdomains. Compared with the strip layout, each interior subdomain has more neighbors~\cite[\S6.3]{gong2022convergence}. 
The theoretical results of 
\cite{galkowski2024convergence} for one-level Schwarz methods with PML transmission conditions applied to the Helmholtz equation indicate that one expects the contractive regime -- where 
$\|\bE^s\|<1$ -- to occur roughly when 
$s\sim 2\sqrt{N}$ 
(with $2\sqrt{N}$ equal to the number of subdomains in the $x_1$ directions plus the number of subdomains in the $x_2$ direction. 
Tables~\ref{tab:checkerboard_1k_norm_E} and~\ref{tab:checkerboard_1k_norm_E_var} therefore report $\|\bE^s\|$ for $s$ near $2\sqrt{N}$, in the homogeneous and variable-coefficient (with $\epsilon$ given by~\eqref{eq:epsilon}) cases respectively. In the homogeneous case, $\|\bE^s\|<1$ is reached by $s=2\sqrt{N}+2$ in essentially every configuration, the only outlier within the displayed range being $8\times 8$ at $k=30$. In the variable-coefficient case the threshold is shifted slightly later, but a few more powers again bring $\|\bE^s\|<1$; i.e., power contractivity persists with a higher threshold. 

\begin{table}[h!]
\small
\centering
\caption{Norms of powers of $\bE$ with $\delta = H/4$ and $h = k^{-3/2}$ on 
$\left(0,\frac{2}{3}N^{1/2}\right)\times
\left(0,\frac{2}{3}N^{1/2}\right)$.}
\label{tab:checkerboard_1k_norm_E}
\begin{tabular}{cccccccc}
\hline\hline
\rule{0pt}{3.2ex}%
N
& $k$ 
& $\|\bE\|$ 
& $\|\bE^{2\sqrt{N}-1}\|$ 
& $\|\bE^{2\sqrt{N}}\|$ 
& $\|\bE^{2\sqrt{N}+1}\|$ 
& $\|\bE^{2\sqrt{N}+2}\|$ 
& $\|\bE^{2\sqrt{N}+3}\|$ 
\\[0.6ex]
\hline
\multirow{5}{*}{$2\times 2$}
& 10 & 5.62e+00 & 3.85e+00 & 1.42e+00 & 5.20e-01 & 1.95e-01  &7.29e-02\\
& 15 & 7.48e+00 & 6.04e+00 & 2.04e+00 & 7.03e-01 & 2.54e-01  &9.01e-02\\
& 20 & 9.25e+00 & 8.73e+00 & 2.96e+00 & 9.76e-01 & 3.42e-01  &1.23e-01\\
& 25 & 1.07e+01 & 1.13e+01 & 3.86e+00 & 1.43e+00 & 5.29e-01  &1.94e-01\\
& 30 & 1.28e+01 & 1.41e+01 & 5.70e+00 & 2.29e+00 & 9.33e-01  &3.80e-01\\
\hline
\multirow{5}{*}{$4\times 4$}
& 10 & 6.52e+00 & 1.56e+00 & 9.86e-01 & 6.64e-01 & 4.39e-01  &2.76e-01\\
& 15 & 7.90e+00 & 1.43e+00 & 8.22e-01 & 5.17e-01 & 3.07e-01  &1.76e-01\\
& 20 & 1.02e+01 & 2.44e+00 & 1.53e+00 & 9.79e-01 & 6.38e-01  &4.03e-01\\
& 25 & 1.06e+01 & 2.62e+00 & 1.58e+00 & 9.23e-01 & 5.28e-01  &3.01e-01\\
& 30 & 1.31e+01 & 2.67e+00 & 1.55e+00 & 9.93e-01 & 6.59e-01  &4.19e-01\\
\hline
\multirow{5}{*}{$8\times 8$}
& 10 & 6.84e+00 & 1.41e+00 & 1.21e+00 & 1.03e+00 & 8.78e-01  &7.43e-01\\
& 15 & 8.09e+00 & 1.11e+00 & 9.39e-01 & 7.76e-01 & 6.31e-01  &5.11e-01\\
& 20 & 1.05e+01 & 1.64e+00 & 1.31e+00 & 1.06e+00 & 8.82e-01  &7.39e-01\\
& 25 & 1.06e+01 & 1.75e+00 & 1.41e+00 & 1.17e+00 & 9.76e-01  &8.26e-01\\
& 30 & 1.35e+01 & 3.01e+00 & 2.53e+00 & 2.11e+00 & 1.77e+00  &1.50e+00
\\
\hline\hline
\end{tabular}
\end{table}

\begin{table}[h!]
\small
\centering
\setlength{\tabcolsep}{7pt}
\renewcommand{\arraystretch}{1.12}
\caption{Norms of powers of $\bE$ for the varying permittivity profile \eqref{eq:epsilon}. Results are shown for $h = k^{-3/2}$ and $\delta = H/4$ on 
$\left(0,\frac{2}{3}N^{1/2} \right) \times 
\left(0,\frac{2}{3}N^{1/2} \right)$.}
\label{tab:checkerboard_1k_norm_E_var}
\begin{tabular}{cccccccc}
\hline\hline
\rule{0pt}{3.2ex}%
$N$
& $k$
& $\|\bE\|$
& $\|\bE^{2\sqrt{N}-1}\|$
& $\|\bE^{2\sqrt{N}}\|$
& $\|\bE^{2\sqrt{N}+1}\|$
& $\|\bE^{2\sqrt{N}+2}\|$
& $\|\bE^{2\sqrt{N}+3}\|$
\\[0.6ex]
\hline
\multirow{5}{*}{$2\times 2$}
& 10 & 5.73e+00 & 4.04e+00 & 1.61e+00 & 6.64e-01 & 2.76e-01  &1.10e-01\\
& 15 & 7.71e+00 & 6.91e+00 & 2.83e+00 & 1.41e+00 & 6.40e-01  &2.93e-01\\
& 20 & 9.58e+00 & 9.95e+00 & 4.62e+00 & 2.07e+00 & 9.49e-01  &4.58e-01\\
& 25 & 1.15e+01 & 1.34e+01 & 6.54e+00 & 3.39e+00 & 1.76e+00  &8.95e-01\\
& 30 & 1.29e+01 & 1.76e+01 & 9.30e+00 & 5.04e+00 &2.72e+00   &1.47e+00
\\
\hline
\multirow{5}{*}{$4\times 4$}
& 10 & 6.51e+00 & 1.53e+00 & 1.03e+00 & 7.07e-01 & 4.77e-01  &3.22e-01\\
& 15 & 8.29e+00 & 2.74e+00 & 1.75e+00 & 1.15e+00 & 7.52e-01  &5.09e-01\\
& 20 & 1.02e+01 & 2.58e+00 & 1.52e+00 & 9.71e-01 & 6.25e-01  &3.98e-01\\
& 25 & 1.14e+01 & 3.93e+00 & 2.35e+00 & 1.49e+00 & 9.61e-01  &6.12e-01\\
& 30 & 1.31e+01 & 4.45e+00 & 2.78e+00 & 1.83e+00 & 1.21e+00  &7.95e-01\\
\hline
\multirow{5}{*}{$8\times 8$}
& 10 & 6.82e+00 & 1.05e+00 & 8.28e-01 & 7.13e-01 & 6.32e-01  &5.53e-01\\
& 15 & 8.32e+00 & 1.79e+00 & 1.51e+00 & 1.28e+00 & 1.08e+00  &9.05e-01\\
& 20 & 1.05e+01 & 1.73e+00 & 1.44e+00 & 1.21e+00 & 1.01e+00  &8.29e-01\\
& 25 & 1.14e+01 & 1.71e+00 & 1.36e+00 & 1.12e+00 & 9.46e-01  &7.91e-01\\
& 30 & 1.35e+01 & 2.28e+00 & 1.88e+00 & 1.53e+00 & 1.25e+00  &1.03e+00
\\
\hline\hline
\end{tabular}
\end{table}

\subsection{Iteration counts of the RAS-imp preconditioned Richardson method}\label{sec:num6}
In this subsection, we consider the source-free Maxwell equation on the unit square $\Omega=[0,1]^2$, with impedance data generated from the plane wave
\begin{equation*}
    \bu^{\rm pw}(\bx)
    =
    \left(-k_2\exp(i\bm{k}\cdot\bx),\,
    k_1\exp(i\bm{k}\cdot\bx)\right)^T,
\end{equation*}
where $\bm{k}=(k_1,k_2)^T$ satisfies $|\bm{k}|=k$. We report iteration counts of the RAS-imp preconditioned Richardson method over a range of parameters, using a direct solver for each local subproblem and terminating once the relative residual falls below $10^{-6}$ (with at most $500$ iterations). Both strip and checkerboard decompositions are considered, in homogeneous and heterogeneous media~\eqref{eq:epsilon}. For comparison, the iteration counts of the right-preconditioned restarted GMRES method (restart $50$, at most $20$ restarts) under the same stopping criterion are shown in parentheses.

Table~\ref{tab:strip_dir} reports the strip decomposition with homogeneous parameters. Convergence is achieved in all configurations, the iteration counts are robust in $k$ for fixed $N$, and grow with $N$ for fixed $k$ as expected from the longer information-propagation distance.
GMRES consistently requires fewer iterations, confirming the effectiveness of RAS-imp as a Krylov preconditioner.

Table~\ref{tab:i_strip_dir} reports the corresponding results for the variable-permittivity profile~\eqref{eq:epsilon}. Convergence is generally slower than in the homogeneous case.
The deterioration is also visible in the GMRES counts; a plausible cause is that the larger $\epsilon$ in the central region lowers the local wave speed and may promote internal reflections, hindering information transfer across subdomain interfaces.

\begin{table}[h!]
    \centering
    \caption{Richardson iteration counts for the strip domain decomposition in the homogeneous case,  with $h = k^{-3/2}$ and $\delta = H/4$. The corresponding GMRES iteration counts shown in brackets.}
    \label{tab:strip_dir}
\begin{tabular}{ccccccccc}
    \hline\hline
\multirow{2}{*}{$N\backslash k~ ($\# DOF$)$} & \multicolumn{2}{c}{20~(23941)} & \multicolumn{2}{c}{40~(191016)} & \multicolumn{2}{c}{60~(646816)} & \multicolumn{2}{c}{80~(1535105)} \\\cmidrule(lr){2-3} \cmidrule(lr){4-5} \cmidrule(lr){6-7} \cmidrule(lr){8-9}
                                                            &  & \#iter  &  & \#iter  &  & \#iter  &  & \#iter  \\\hline
2 &&10 (8)&&10 (8)&&10 (9)&&10 (9)\\
4 &&24 (16)&&19 (18)&&23 (19)&&26 (20)\\
6 &&39 (30)&&43 (29)&&40 (30)&&37 (33)\\
8 &&52 (40)&&92 (37)&&61 (44)&&82 (42)\\
\hline\hline
\end{tabular}
\end{table}

\begin{table}[h!]
    \centering
    \caption{Richardson iteration counts for the strip domain decomposition
    with the varying permittivity~\eqref{eq:epsilon}, using $h=k^{-3/2}$
    and overlap $\delta=H/4$. The corresponding GMRES iteration counts are
    shown in brackets.}
    \label{tab:i_strip_dir}
\begin{tabular}{ccccccccc}
    \hline\hline
\multirow{2}{*}{$N\backslash k~(\#\mathrm{DOF})$}
& \multicolumn{2}{c}{20~(23941)}
& \multicolumn{2}{c}{40~(191016)}
& \multicolumn{2}{c}{60~(646816)}
& \multicolumn{2}{c}{80~(1535105)} \\
\cmidrule(lr){2-3} \cmidrule(lr){4-5} \cmidrule(lr){6-7} \cmidrule(lr){8-9}
&  & \#iter &  & \#iter &  & \#iter &  & \#iter \\
\hline
2 && 14 (10) && 13 (11) && 13 (11) && 20 (13) \\
4 && 27 (22) && 36 (26) && 39 (29) && 54 (33) \\
6 && 76 (40) && 68 (44) && 95 (47) && 112 (53) \\
8 && -- (59) && -- (59) && -- (73) && -- (79) \\
\hline\hline
\end{tabular}
\end{table}


Tables~\ref{tab:checkerboard_dir} and~\ref{tab:i_checkerboard} report the checkerboard decomposition for the homogeneous and variable-coefficient cases, respectively. Compared with strip decompositions, the iteration counts are typically larger, which is consistent with the richer interface structure and the larger number of propagation paths discussed in \S\ref{sec:num5}. The same trend persists for GMRES, and Krylov acceleration again provides a substantial reduction. 
In the variable-coefficient case, however, a few configurations converge under the checkerboard layout while the strip counterpart does not, which we attribute to the additional propagation paths available in the checkerboard layout.

\begin{table}[h!]
    \centering
    \caption{Richardson iteration counts for the checkerboard domain decomposition in the homogeneous case, with $h =k^{-3/2}$ and $\delta = H/4$; the corresponding GMRES iteration counts are shown in brackets.}
    \label{tab:checkerboard_dir}
\begin{tabular}{ccccccccc}
    \hline\hline
\multirow{2}{*}{$N\backslash k~ (\# {\rm DOF})$} & \multicolumn{2}{c}{20~(23941)} & \multicolumn{2}{c}{40~(191016)} & \multicolumn{2}{c}{60~(646816)} & \multicolumn{2}{c}{80~(1535105)} \\\cmidrule(lr){2-3} \cmidrule(lr){4-5} \cmidrule(lr){6-7} \cmidrule(lr){8-9}
                                                            &  & \#iter  &  & \#iter  &  & \#iter  &  & \#iter  \\\hline
$2\times 2$ &&16 (11)&&19 (12)&&21 (12)&&23 (12)\\
$4\times 4$ & &43 (27)& &33 (26)& &34 (28)& &40 (28) \\
$6\times 6$ & &64 (45)& &50 (40)& &47 (40)& &58 (42) \\
$8\times 8$ & &126 (74)& &110 (59)& &105 (58)& &66 (57) \\
\hline\hline
\end{tabular}
\end{table}

\begin{table}[h!]
    \centering
    \caption{Richardson iteration counts for the checkerboard domain decomposition
    with the varying permittivity~\eqref{eq:epsilon}, $h=k^{-3/2}$ and $\delta=H/4$.
    The corresponding GMRES iteration counts are shown in brackets.}
    \label{tab:i_checkerboard}
\begin{tabular}{ccccccccc}
    \hline\hline
\multirow{2}{*}{$N \backslash k~(\#\mathrm{DOF})$}
& \multicolumn{2}{c}{20~(23941)}
& \multicolumn{2}{c}{40~(191016)}
& \multicolumn{2}{c}{60~(646816)}
& \multicolumn{2}{c}{80~(1535105)} \\
\cmidrule(lr){2-3} \cmidrule(lr){4-5} \cmidrule(lr){6-7} \cmidrule(lr){8-9}
&  & \#iter &  & \#iter &  & \#iter &  & \#iter \\
\hline
$2\times 2$ &&24 (13) &&27 (15) &&34 (17) &&38 (18) \\
$4\times 4$ &&60 (33) &&56 (36) &&63 (40) &&101 (46) \\
$6\times 6$ &&230 (56) &&112 (59) &&166 (64) &&128 (72) \\
$8\times 8$ &&419 (98) &&-- (80) &&135 (93) &&-- (106) \\
\hline\hline
\end{tabular}
\end{table}

\section*{Acknowledgements}
SG was supported by National Natural Science Foundation of China (No. 12371410, 12411530066).  EAS was supported by 
ERC Synergy Grant PSINumScat (101167139). CL and YY were supported by Hetao Shenzhen-Hong Kong Science and Technology Innovation Cooperation Zone Project (No.HZQSWS-KCCYB-2024016).

\footnotesize{

\bibliographystyle{plain}
\bibliography{reference}
}

\end{document}